\documentclass[11pt,a4paper,reqno]{amsart}
\usepackage{amsfonts,amssymb,amsmath,amsthm,graphicx,color,amscd,xspace,verbatim,dsfont,mathrsfs}
\usepackage{scalerel}
\makeatletter

\usepackage{hyperref}

\makeatletter
\def\@tocline#1#2#3#4#5#6#7{\relax
  \ifnum #1>\c@tocdepth 
  \else
    \par \addpenalty\@secpenalty\addvspace{#2}%
    \begingroup \hyphenpenalty\@M
    \@ifempty{#4}{%
      \@tempdima\csname r@tocindent\number#1\endcsname\relax
    }{%
      \@tempdima#4\relax
    }%
    \parindent\z@ \leftskip#3\relax \advance\leftskip\@tempdima\relax
    \rightskip\@pnumwidth plus4em \parfillskip-\@pnumwidth
    #5\leavevmode\hskip-\@tempdima
      \ifcase #1
       \or\or \hskip 1em \or \hskip 2em \else \hskip 3em \fi%
      #6\nobreak\relax
      \dotfill
      \hbox to\@pnumwidth{\@tocpagenum{#7}}
    \par
    \nobreak
    \endgroup
  \fi}
\makeatother

\newtheorem{theorem}{Theorem}[section]
\newtheorem{lemma}[theorem]{Lemma}
\newtheorem{proposition}[theorem]{Proposition}

\theoremstyle{definition}
\newtheorem{definition}[theorem]{Definition}

\newtheorem{remark}[theorem]{Remark}

\newcommand{\N}{{\mathbb N}}
\newcommand{\R}{{\mathbb R}}

\newcommand{\tr}{\mathrm{tr}^*}
\newcommand{\loc}{\mathrm{loc}}

\newcommand{\beqn}{\begin{eqnarray}}
\newcommand{\eeqn}{\end{eqnarray}}   
\newcommand{\beq}{\begin{eqnarray*}}
\newcommand{\eeq}{\end{eqnarray*}}

\newcommand{\be}{\small\begin{equation}}
\newcommand{\bel}[1]{\small\begin{equation}\label{#1}}
\newcommand{\ee}{\end{equation}\normalsize}

\newcommand{\BA}{\begin{array}}
\newcommand{\EA}{\end{array}}
\newcommand{\BAN}{\renewcommand{\arraystretch}{1.2}
\setlength{\arraycolsep}{2pt}\begin{array}}
\newcommand{\BAV}[2]{\renewcommand{\arraystretch}{#1}
\setlength{\arraycolsep}{#2}\begin{array}}

\newcommand{\BSA}{\begin{subarray}}
\newcommand{\ESA}{\end{subarray}}
\newcommand{\BAL}{\begin{aligned}}
\newcommand{\EAL}{\end{aligned}}

\newcommand{\abs}[1]{\left |#1\right |}

\newcommand{\norm}[1]{\left \|#1\right \|}

\newcommand{\supp}{\mathrm{supp}\,}
\newcommand{\dist}{\mathrm{dist}\,}
\newcommand{\sign}{\mathrm{sign}}
\newcommand{\diam}{\mathrm{diam}\,}

\newcommand{\prt}{\partial}

\newcommand{\tl}{\tilde}

\newcommand{\sbs}{\subset}

\def\dist{\mathrm{dist}}

\def \dd {\mathrm{d}}

\def\ga{\alpha}            
       \def\gd{\delta}      \def\ge{\epsilon}
\def\gth{\theta}                         
\def\gf{\phi}           
            \def\gl{\lambda}
\def\gm{\mu}        \def\gn{\nu}         
            
\def\gs{\sigma}       
      \def\gw{\omega}
                
     \def\Gd{\Delta}      

\def\Gl{\Lambda}          
\def\Gw{\Omega}              

\def\CS{{\mathcal S}}      \def\CN{{\mathcal N}}
      
\def\CA{{\mathcal A}}   \def\CB{{\mathcal B}}   \def\CC{{\mathcal C}}
      \def\CF{{\mathcal F}}
      
      \def\CL{{\mathcal L}}

\def\BBG {\mathbb G}       
\def\BBJ {\mathbb J}   \def\BBK {\mathbb K}    
   \def\BBN {\mathbb N}    
\def\BBP {\mathbb P}   \def\BBR {\mathbb R}

   \def\GTB {\mathfrak B}

\def\GTM {\mathfrak M}

\def\tr{\mathrm{tr}_{\mu,\Sigma}}

\newcommand{\ei}{\phi_{\xm,\Sigma}}
\newcommand{\xa}{\alpha}
\newcommand{\xb}{\beta}
\newcommand{\xg}{\gamma}
\newcommand{\xG}{\Gamma}
\newcommand{\xd}{\delta}
\newcommand{\xD}{\Delta}
\newcommand{\xe}{\varepsilon}
\newcommand{\xz}{\zeta}

\newcommand{\xk}{\kappa}
\newcommand{\xl}{\lambda}
\newcommand{\xL}{\Lambda}
\newcommand{\xm}{\mu}
\newcommand{\xn}{\nu}

\newcommand{\xr}{\rho}
\newcommand{\xs}{\sigma}
\newcommand{\xS}{\Sigma}
\newcommand{\xf}{\phi}
\newcommand{\xF}{\Phi}

\newcommand{\xo}{\omega}
\newcommand{\xO}{\Omega}

\newcommand{\myfrac}[2]{{\displaystyle \frac{#1}{#2} }}
\newcommand{\myint}[2]{{\displaystyle \int_{#1}^{#2}}}

\def \dd {\mathrm{d}}
\def \dx {\mathrm{d}x}
\def \dy {\mathrm{d}y}

\def\Nthb{\BBN_{\xa}}

\newcommand{\ap}{{\xa_{\scaleto{+}{3pt}}}}
\newcommand{\am}{{\xa_{\scaleto{-}{3pt}}}}

\newcommand\1{{\ensuremath {\mathds 1} }}
\def\bal#1\eal{\small\begin{align*}#1\end{align*}\normalsize}
\def\ba#1\ea{\small\begin{align}#1\end{align}\normalsize}
\numberwithin{equation}{section}

\begin{document}

\title[Semilinear elliptic equations]{Semilinear elliptic equations involving power nonlinearities and hardy potentials with boundary singularities}
\author[K. T. Gkikas]{Konstantinos T. Gkikas}
\address{K. T. Gkikas, Department of Mathematics, National and Kapodistrian University of Athens, 15784 Athens, Greece}
\email{kugkikas@math.uoa.gr}

\author[P.T. Nguyen]{Phuoc-Tai Nguyen}
\address{Phuoc-Tai Nguyen, Department of Mathematics and Statistics, Masaryk University, Brno, Czech Republic}
\email{ptnguyen@math.muni.cz}


\begin{abstract} Let $\Omega \subset\mathbb{R}^N$ ($N\geq 3$) be a $C^2$ bounded domain and $\Sigma \subset \partial\Omega$ be a $C^2$ compact submanifold without boundary, of dimension $k$, $0\leq k \leq N-1$. We assume that  $\Sigma = \{0\}$ if  $k = 0$ and $\Sigma=\partial\Omega$ if $k=N-1$. Denote $d_\Sigma(x)=\dist(x,\Sigma)$ and put $L_\mu=\Delta + \mu d_{\Sigma}^{-2}$ where $\mu$ is a parameter. In this paper, we study boundary value problems for equations $-L_\mu u \pm |u|^{p-1}u = 0$ in $\Omega$ with prescribed condition $u=\nu$ on $\partial \Omega$, where $p>1$ and $\nu$ is a given measure on $\partial \Omega$. The nonlinearity $|u|^{p-1}u$ is referred to as \textit{absorption} or \textit{source} depending whether the plus sign or minus sign appears. The distinctive feature of the problems is characterized by the interplay between the concentration of $\Sigma$, the type of nonlinearity, the exponent $p$ and the parameter $\mu$. The absorption case and the source case are sharply different in several aspects and hence require completely different approaches. In each case, we establish various necessary and sufficient conditions expressed in terms of appropriate capacities. In comparison with related works in the literature, by employing a fine analysis, we are able to treat the supercritical ranges for the exponent $p$, and the critical case for the parameter $\mu$, which justifies the novelty of our paper.
	
\medskip

\noindent\textit{Key words: Hardy potentials, boundary singularities, capacities, critical exponents, removable singularity, Keller-Osserman estimates}

\medskip

\noindent\textit{Mathematics Subject Classification:} 35J60, 35J75, 35J10, 35J66.
\end{abstract}

\maketitle
\tableofcontents

\section{Introduction}

Let $N\geq 3,$ $\xO\subset\mathbb{R}^N$ be a $C^2$ bounded domain and $\xS \subset \partial\Omega$ be a $C^2$ compact submanifold in $\mathbb{R}^N$ without boundary, of dimension $k$, $0\leq k \leq N-1$. We assume that  $\xS = \{0\}$ if  $k = 0$ and $\xS=\partial\xO$ if $k=N-1.$ Let $d_{\partial \Omega}(x)=\dist(x,\partial \Omega)$ and $d_\xS(x)=\dist(x,\xS)$. Two typical semilinear elliptic equations involving power nonlinearities and Hardy type potentials are of the form
\ba \label{eq:0} \tag{$E_\pm$}
L_\xm u \pm |u|^{p-1}u = 0\quad\text{in}\;\xO,
\ea
where $p>1$, $\xm \in \R$ is a parameter and
\bal L_\mu u: =\Delta u + \frac{\mu}{d_\xS^2}u.
\eal
The nonlinearity $|u|^{p-1}u$ in \eqref{eq:0} is referred to as an \textit{absorption} or a \textit{source} depending whether the plus sign or minus sign appears in \eqref{eq:0}.

Boundary value problems for \eqref{eq:0} with $\mu=0$ became a central research subject in the area of partial differential equations with abundant literature because of their applications in physics, engineering, and other applied scientific disciplines. A rich theory has been developed for the boundary value problem with a power absorption in case $\xm=0$, $\xS=\partial\xO$, namely for the problem
\be\label{withouthardy} \left\{ \BAL
-\xD u+|u|^{p-1}u&=0\qquad &&\text{in }\;\Gw,\\
u&=\xn \qquad&&\text{on }\;\partial\Gw,
\EAL \right.
\ee
where $\xn$ is a measure on $\partial \Omega$. Throughout this paper, we denote by $\GTM(\partial \Omega)$ and $\GTM^+(\partial \Omega)$ the space of finite measures on $\partial\xO$ and its positive cone respectively. The first study of  \eqref{withouthardy} was carried out by Gmira and V\'eron in \cite{GV} where the existence of a solution was obtained for any $\xn\in\mathfrak{M}(\partial\xO)$ in the subcritical case $1<p<\frac{N+1}{N-1}$. In the supercritical case $p \geq \frac{N+1}{N-1}$, a breakthrough was achieved by Marcus and V\'eron \cite{MV-JMPA01}, asserting that problem \eqref{withouthardy} possesses a solution if and only if $\xn$ is absolutely continuous with respect to  $\mathrm{Cap}_{\frac{2}{p},p'}^{\partial\xO}$, namely $\xn(E)=0$ for any Borel set $E\subset \partial\xO$ such that $\mathrm{Cap}_{\frac{2}{p},p'}^{\partial\xO}(E)=0$ (see \eqref{Capsub} for the definition of the Bessel capacities and see \eqref{abs-cont} for the meaning of the absolute continuity).

Equations \eqref{eq:0} with $\mu \neq 0$ are closely associated to nonlinear Schr\"odinger equations. A fundamental theory for the linear version of \eqref{eq:0} with more general potentials has been established by Ancona in \cite{An1,An2}. Recently, the study of boundary value problems for \eqref{eq:0} in measure frameworks has been pushed forward by several works (see  \cite{MarNgu,GkV,MarMor,MT,GkiNg_2019,GkiNg_absorption,GkiNg_source,CV-JDE,CQZ-PAFA,CV} and references therein), especially in the special cases $\Sigma=\partial \Omega$ and $\Sigma=\{0\}$.
In the case $\xS=\partial\xO$ and $\xm\neq0,$ the existence of solutions to the corresponding boundary value problem depends strongly on the value of $\mu$ in comparison with the optimal Hardy constant
\bal
C_{\xO,\partial\xO}:=\inf_{u\in H_0^1(\xO)\setminus \{0\}}\frac{\int_\xO|\nabla u|^2dx}{\int_\xO |u|^2 d_{\partial \Omega}^{-2}dx}
\eal
and the sign of the first eigenvalue of $-L_\mu$
\ba\label{eigenvalueboundary}
\xl_{\xm,\partial\xO}:=\inf_{u\in H_0^1(\xO)\setminus \{0\}}\frac{\int_\xO|\nabla u|^2dx-\xm\int_\xO |u|^2 d_{\partial \Omega}^{-2}dx}{\int_\xO |u|^2 dx}.
\ea
It is well known that $C_{\xO,\partial\xO}\in (0,\frac{1}{4}]$ and $C_{\xO,\partial\xO}=\frac{1}{4}$ if $\xO$ is convex (see e.g. \cite{BM,hardy-marcus}). When $\xm\leq C_{\xO,\partial\xO}$ then $\xl_{\xm,\partial\xO}>-\infty$ and the corresponding eigenfunction $\xf_{\xm,\partial\xO}$ satisfies $\xf_{\xm,\partial\xO}\approx d_{\partial \Omega}^\xa$ (see, e.g., \cite{DD,FMT}) where $\xa:=\frac{1}{2}+\sqrt{\frac{1}{4}-\xm}$. Marcus and Nguyen introduced the notion of normalized boundary trace (see \cite[Definition 1.2]{MarNgu}), denoted here by $\text{tr}^*_{\xm,\partial\xO}$, and studied the following problem
\ba \label{boundaryabsorption} \left\{ \BAL
- \xD u-\xm\frac{u}{d_{\partial \Omega}^2}+|u|^{p-1}u&=0\qquad \text{in }\;\Gw,\\
\text{tr}^*_{\xm,\partial\xO}(u)&=\xn.
\EAL \right.
\ea
In particular, they showed in \cite{MarNgu} that if $0<\xm<C_{\xO,\partial\xO}$ and $1<p<\frac{N+\xa}{N-\xa-2}$ then problem \eqref{boundaryabsorption} admits a unique solution for any $\xn\in \mathfrak{M}^+(\partial\xO)$. Independently, Gkikas and V\'eron \cite{GkV} introduced the notion of the boundary trace in a dynamic way, denoted here by $\mathrm{tr}_{\xm,\partial\xO}$, and obtained a similar type of existence result for a more general absorption in a slightly different context. In the supercritical case $p\geq\frac{N+\xa}{N-\xa-2},$ Gkikas and V\'eron \cite{GkV} showed that problem \eqref{boundaryabsorption} with $\mathrm{tr}_{\xm,\partial\xO}$ in place of $\mathrm{tr}_{\xm,\partial\xO}^*$ possesses a unique solution if and only if $\xn$ is absolutely continuous with respect to  $\mathrm{Cap}_{2-\frac{\xa+1}{p'},p'}^{\partial\xO}$ provided $\xl_{\xm,\partial\xO}>0.$ When $\xm<0,$ Marcus and Moroz \cite{MarMor} pointed out the existence of a second critical power. In particular, they showed in \cite{MarMor} that there is no solution to \eqref{boundaryabsorption} for any $\xn\in \mathfrak{M}^+(\partial\xO)$ if $\xm<0$ and $p\geq \frac{\xa+1}{\xa-1}$.

In the case $\xS\subsetneq\partial\xO$ with $0 \leq k\leq N-2$, the linear equation $L_\xm u=0$ was extensively investigated in recent papers \cite{MT,BGT} where the optimal Hardy constant
\ba\label{hardyineq}
C_{\xO,\xS}:=\inf_{u\in H_0^1(\xO)\setminus \{0\}}\frac{\int_\xO|\nabla u|^2dx}{\int_\xO |u|^2 d_\xS^{-2}dx}
\ea
is deeply involved. It is known (see, e.g., \cite{FF}) that  $C_{\xO,\xS} \in (0,\frac{(N-k)^2}{4}]$ . Moreover, when  $\xm< \frac{(N-k)^2}{4}$, there exists a minimizer $\ei\in H_0^1(\xO)$ of the eigenvalue problem
\ba\label{eigenvalue}
\xl_{\xm,\Sigma}:=\inf_{u\in H_0^1(\xO)\setminus \{0\}}\frac{\int_\xO|\nabla u|^2dx-\xm\int_\xO |u|^2 d_\Sigma^{-2}dx}{\int_\xO |u|^2 dx}>-\infty.
\ea
If $\xm=\frac{(N-k)^2}{4}$ then \eqref{eigenvalue} holds true, however there is no minimizer in $H_0^1(\xO).$  The reader is referred to \cite{FF} for more detail. In addition, by \cite[Proposition A.2]{BGT} (see also\cite[Lemma 2.2]{MT}), the corresponding eigenfunction $\ei$ satisfies the following pointwise estimate
\ba \label{eigenfunctionestimates}
\ei \approx d_{\partial \Omega}\,d^{-\am}_\Sigma \quad \text{in } \Omega ,
\ea
where
\ba\label{apmintro}
\xa_\pm:=H\pm\sqrt{H^2-\xm} \quad\text{and}\quad H:=\frac{N-k}{2}.
\ea

For $\xm<\min \left\{C_{\xO,\Sigma},\frac{2(N-k)-1}{4}\right\}$, Marcus and Nguyen \cite{MT} generalized the notion of normalized boundary trace in \cite{MarNgu}, denoted here by $\textrm{tr}^*_{\xm,\xS}$, to formulate and to study the boundary value problem
\ba\label{submanifoldabsorption} \left\{ \BAL
- \xD u-\xm\frac{u}{d_{\Sigma}^2}+|u|^{p-1}u&=0\qquad \text{in }\;\Gw,\\
\text{tr}^*_{\xm,\xS}(u)&=\xn.
\EAL \right.
\ea
An interesting feature of problem \eqref{submanifoldabsorption} is that there are different scenarios for the existence of a (unique) solution depending whether the support of $\nu$ is contained in $\partial \Omega \setminus \Sigma$ or in $\Sigma$.
More precisely, it was showed in \cite{MT} that problem \eqref{submanifoldabsorption} admits a nonnegative solution  for any $\xn\in \mathfrak{M}^+(\partial\xO)$ either if $1<p<\frac{N+1}{N-1}$ and $\xn$ has compact support in $\partial\xO\setminus \xS$ or if $1<p<\frac{N-\am+1}{N-\am-1}$ and $\xn$ has compact support in $\xS$. Very recently,  Barbatis, Gkikas and Tertikas gave a concept of boundary trace $\tr(u)$ in a dynamic way (see \cite[Definition 2.10]{BGT}) and employed it to demonstrate that the same results hold true provided  $\xm\leq H^2$ and $\xl_{\xm,\Sigma}>0$. In the supercritical cases, it was discovered in \cite{MT} that problem \eqref{submanifoldabsorption} has no solution either if $\xn=\xd_y$ with  $y\in \partial\xO\setminus\xS$ and $p\geq\frac{N+1}{N-1}$ or if $\xn=\xd_y$ with  $y\in\xS$ and $p\geq\frac{N-\am+1}{N-\am+1}$. Here by $\xd_y$ we mean the Dirac measure concentrated at $y$. The special case $\xS=\{0\}\subset\partial\xO$ was treated by Chen and V\'eron in \cite{CV}.

There is also a vast literature on boundary value problems for equations with a source term. We list below some relevant works (far from being a compete list of references). Let $\xr>0$ and $\xS=\partial\xO,$ the problem with power source
\be\label{boundarysource} \left\{ \BAL
- \xD u-\xm\frac{u}{d_{\partial \Omega}^2}-|u|^{p-1}u&=0\qquad \text{in }\;\Gw,\\
\text{tr}_{\xm,\partial\xO}(u)&=\xr\xn,
\EAL \right.
\ee
was thoroughly studied by Bidaut-V\'eron, Hoang, Nguyen and V\'eron in \cite{BHV} (see also \cite{BVi} for $\xm=0$). To be more precise, the authors in \cite{BHV} showed that problem \eqref{boundarysource} admits a solution for some small $\xr>0$ if and only if there exists a positive constant $C>0$ such that
\bal
\xn(E)\leq C\mathrm{Cap}_{2-\frac{\xa+1}{p'},p'}^{\partial\xO}(E), 
\eal
for any Borel set $E\subset \partial\xO$. Afterwards, various necessary and sufficient conditions for the existence of a solution to problem  \eqref{boundarysource} were obtained by Gkikas and Nguyen \cite{GkiNg_2019}.

When $\xS\subset\xO,$ the corresponding boundary problems with an absorption and a source involving operator $L_\xm $ were extensively studied by Gkikas and Nguyen in \cite{GkiNg_absorption} and \cite{GkiNg_source} respectively; see also the papers by D\'avila and  Dupaigne \cite{DD1,DD2}, Dupaigne and Nedev \cite{DN}, Fall \cite{Fal}, and Chen and  Zhou \cite{CheZho} for related results on semilinear equations with a source term.

The \textit{objective} of this paper is to study  boundary value problems for \eqref{eq:0} where $\xS\subsetneq\partial\xO$. More precisely, we aim to establish existence and nonexistence results under different sharp conditions on the ranges containing the power $p$, the value of $\mu$,  the geometry of $\Sigma$, and the concentration of the boundary measure datum $\xn$. In addition, we will give various necessary and sufficient conditions expressed in the terms of Bessel capacities for the existence of solutions.

We point out here that the optimal Hardy constant $C_{\xO,\xS}$, as well as the asymptotic behavior of the first eigenfunction $\ei$, the Green function and the Martin kernel, are different from those in the case $\xS\subset\xO.$  As a result, we derive different critical exponents for existence and nonexistence phenomena from those in \cite{GkiNg_absorption,GkiNg_source}. In addition, new essential difficulties arise, which require a very delicate approach to overcome.

\textit{Throughout this paper, we assume that}
\ba \label{assumpt}
\mu \leq \left( \frac{N-k}{2} \right)^2 \quad \text{and} \quad \lambda_{\mu,\Sigma}>0.
\ea
Under assumption \eqref{assumpt}, a theory for linear equations involving $L_\mu$ was established in \cite{BGT}, which forms a basis for the study of equation \eqref{eq:0}.
In order to formulate the boundary value problem for \eqref{eq:0}, we use the notion of boundary trace, introduced in \cite{BGT}, whose definition is recalled below.

A family $\{\Omega_n\}$ is called a $C^2$ {\it exhaustion} of $\Omega$ if $\Omega_n$ is a $C^2$ bounded domain,  $\Omega_n \Subset \Omega_{n+1} \Subset \Omega$ for any $n \in \N$ and $\cup_{n \in \N}\Omega_n = \Omega$.

Let $x_0 \in \Omega$ be a fixed reference point.
\begin{definition}[Boundary trace] \label{nomtrace-0}
We say that	a function $u\in W^{1,\kappa}_{\loc}(\xO)$ ($\kappa>1$) possesses a \emph{boundary trace}  if there exists a measure $\nu \in\GTM(\partial \Omega )$ such that for any $C^2$ exhaustion  $\{ \xO_n \}$ of $\Omega$ containing $x_0$, there  holds
	\be\label{trab-0}
	\lim_{n\rightarrow\infty}\int_{ \partial \xO_n}\phi u\, \dd \omega_{\xO_n}^{x_0}=\int_{\partial \Omega} \phi \,\dd \nu \quad\forall \phi \in C(\overline{\Omega}).
	\ee
	The boundary trace of $u$ is denoted by $\tr(u)$. Here $\omega_{\Omega_n}^{x_0}$ is the $L_\mu$-harmonic measure on $\partial \Omega_n$ relative to $x_0$ (see Subsection \ref{subsec:boundarytrace}).
\end{definition}

For $q \in [1,+\infty)$, denote by $L^q(\Omega;\phi_{\mu,\Sigma})$ the weighted Lebesgue space  \bal L^q(\Omega;\phi_{\mu,\Sigma}) := \{ u: \Omega \to \R \text{ measurable such that } \int_{\Omega} |u|^q \phi_{\mu,\Sigma} \, \dx < +\infty \}
\eal 
endowed with the norm 
\bal \| u \|_{L^q(\Omega;\phi_{\mu,\Sigma})} := \left( \int_{\Omega} |u|^q \phi_{\mu,\Sigma} \, \dx \right)^{\frac{1}{q} }.
\eal
Let $H^1(\xO;\ei^2)$ be the weighted Sobolev space
\bal
H^1(\xO;\ei^2) :=\{u \in H^1_{\loc}(\xO):\;|u|\ei+|\nabla u|\ei\in L^2(\xO)\}
\eal
endowed with the norm
\bal
\norm{u}_{H^1(\xO;\ei^2)} :=\left(\int_\xO|u|^2\ei^2 \, \dx+\int_\xO|\nabla u|^2\ei^2 \, \dx \right)^{\frac{1}{2} }.
\eal
We also denote by $H^1_0(\xO;\ei^2)$ the closure of $C_0^\infty(\xO)$ with respect to the norm $\norm{\cdot}_{H^1(\xO;\ei^2)}.$ It is worth mentioning here that $H^1_0(\xO;\ei^2)=H^1(\xO;\ei^2)$ (see \cite[Theorem 4.5]{BGT}).

Weak solutions of boundary value problem for \eqref{eq:0} with prescribed boundary trace are defined below.
\begin{definition}
Let $p>1.$ We say that $u$ is a {\it weak solution} of
\be\tag{$\mathrm{BVP}_{\pm}$} \label{BPpm} \left\{ \BAL
- L_\gm u\pm|u|^{p-1}u&=0\qquad \text{in }\;\Gw,\\
\tr(u)&=\xn,
\EAL \right.
\ee
if $u\in L^1(\Omega;\ei)$, $|u|^p \in L^1(\Omega;\ei)$  and
	\be \label{nlinearweakformintro}
	- \int_{\xO}u L_{\xm }\zeta \, \dd x\pm \int_{\xO}|u|^{p-1}u\zeta \, \dd x=- \int_{\Gw} \mathbb{K}_{\xm}[\xn]L_{\xm }\zeta \, \dd x
	\qquad\forall \zeta \in\mathbf{X}_\xm(\xO),
	\ee
where

\be \label{Xmu} {\bf X}_\mu(\Gw):=\{ \zeta \in H_{\loc}^1(\Omega): \ei^{-1} \zeta \in H^1(\Gw;\ei^{2}), \, \ei^{-1}L_\mu \zeta \in L^\infty(\Omega)  \}.
\ee
\end{definition}

If $\xm<C_{\xO,\xS}$ then it is well known that there exists the Green function $G_\mu$ associated with $-L_\mu$. In addition, Ancona \cite{An2} showed the existence of the corresponding Martin kernel which is unique up to a normalization. Marcus and Nguyen \cite{MarNgu} applied results for a class of more general Schr\"odinger operators in \cite {Marcus} to the model case $L_\xm$ and showed two-sided estimates concerning the Green function $G_\mu$ and the Martin kernel $K_\xm$. Recently, Barbatis, Gkikas and Tertikas \cite{BGT} followed a different approach to obtain sharp two-sided estimates for $G_\xm$ and $K_\xm$ even in the critical case $\xm=\frac{(N-k)^2}{4}$ provided $\xl_{\frac{(N-k)^2}{4}}>0.$ These estimates will be quoted in Subsection \ref{subsec:GreenMartin}.

We denote by $\GTM(\Omega;\ei)$ the space of measures $\tau$ such that $\int_{\Omega}\ei \,\dd|\tau|<+\infty$ and by $\GTM^+(\Omega;\ei)$ the positive cone of $\GTM(\Omega;\ei)$. Recall that the Green operator and the Martin operator are respectively defined by
\ba \label{Grt-0}
\BBG_\mu[\tau](x) :=\int_{\xO }G_{\mu}(x,y) \, \dd\tau(y), \quad \tau \in \GTM(\Omega;\ei), \\
\label{Mrt-0} \mathbb{K}_\mu[\gn](x) :=\int_{\partial\xO }K_{\mu}(x,y) \, \dd\xn(y), \quad \gn\in \mathfrak{M}(\partial\xO).
\ea
Main properties of the above operators  were established in \cite{BGT} and will be listed in Subsections \ref{subsec:boundarytrace} and  \ref{subsec:linear}. We remark that in light of  \cite[Theorem 2.12]{BGT}, a function $u$ is a weak solution to problem \eqref{BPpm} if and only if
\ba \label{represent-BP}
u \pm \BBG_\mu[|u|^{p-1}u] = \BBK_\mu[\nu] \quad \text{a.e. in } \Omega.
\ea

In order to state the main results of the paper, we need to introduce some notations and recall some results in \cite{BGT}. For any $\xb>0,$ we set \bal \xS_\xb:=\{x\in\mathbb{R}^N\setminus \xS:\;d_\xS(x)<\xb\}.
\eal
Let $\xb_1$ be small enough such that for any $x\in \xO_{\xb_1}$ there exists a unique $\xi_x\in\partial\xO$ satisfying $d_{\partial \Omega}(x)=|x-\xi_x|.$ Then we put
\bal
\tilde{d}_{\xS}(x):=\sqrt{|{\rm dist}^{\partial\xO}(\xi_x,\xS)|^2+|x-\xi_x|^2},
\eal
where ${\rm dist}^{\partial\xO}(\xi_x,\xS)$ denotes the geodesic distance of $\xi_x$ to $\xS$ measured on $\partial\xO.$

\begin{definition}\label{defweightintro}
Let $\xb_0>0$ be the constant in Lemma \ref{propdK}. Let $\eta_{\xb_0}$ be  a smooth cut-off function $0\leq\eta_{\xb_0}\leq1$ such that $\eta_{\xb_0}=1$ in $\overline{\xS}_{\frac{\xb_0}{4}}$ with compact support in $\xS_{\frac{\xb_0}{2}}.$ We define
\ba \label{Wintro} W(x):=\left\{ \BAL &(d_{\partial \Omega}(x)+\tilde d_\xS(x)^2)\tilde d_{\xS}(x)^{-\ap},\qquad&&\text{if}\;\mu <\frac{(N-k)^2}{4}, \\
&(d_{\partial \Omega}(x)+\tilde d_\xS(x)^2)d_{\xS}(x)^{-\frac{N-k}{2}}|\ln \tilde d_{\xS}(x)|,\qquad&&\text{if}\;\mu =\frac{(N-k)^2}{4},
\EAL \right. \quad x \in \Omega \cap \xS_\xb,
\ea
and
\bel{tildeWintro}
\tilde W(x):=(1-\eta_{\beta_0}(x))+\eta_{\beta_0}(x)W(x) , \quad x \in \Omega. \ee
Here $\ap$ is defined in \eqref{apmintro}.
\end{definition}

Let $z \in \Omega $ and $h\in C(\partial\Omega ).$ Then by \cite[Lemma 6.8]{BGT}, there exists a unique solution $v_h$ of the Dirichlet problem
\be \label{linear} \left\{ \BAL
L_{\mu}v&=0\qquad \text{in}\;\;\xO\\
v&=h\qquad \text{on}\;\;\partial\xO.
\EAL \right. \ee
The boundary value condition in \eqref{linear} is understood in the sense that
\bal
\lim_{\dist(x,F)\to 0}\frac{v(x)}{\tilde W(x)}=h \quad \text{for every compact set } \; F\subset \partial \Omega.
\eal

Our first theorem provides a removability result when $p\geq \frac{\ap+1}{\ap-1}$.

\begin{theorem} \label{remov-1}
Assume $\mu \leq H^2$ if $k<N-2$ or $\xm<1$ if $k=N-2$ and $p\geq \frac{\ap+1}{\ap-1}$. We additionally assume that $\xO$ is a $C^3$ open bounded domain. If $u \in C^2(\xO)$ is a nonnegative solution of
\be \tag{$E_+$} \label{eq:power-a}
	-L_\mu u+|u|^{p-1}u=0 \quad \text{in } \Omega
	\ee
such that
	\be\label{as1intro}
	\lim_{x\in\xO,\;x\rightarrow\xi}\frac{u(x)}{\tilde W(x)}=0\qquad\forall \xi\in \partial\xO\setminus\xS,
	\ee
	locally uniformly in $\partial\xO\setminus\xS$, then $u\equiv 0.$
\end{theorem}

Let us briefly explain the proof of the above theorem. In the Appendix, under the assumption that $\Omega$ is a $C^3$ bounded domain, we construct appropriate barriers, which are different from those in \cite{GkiNg_absorption}, and we use them to obtain Keller-Osserman type estimates (see Theorem \ref{prop19}). By applying these estimates to solutions satisfying \eqref{eq:power-a}-\eqref{as1intro} and using the representation Theorem 2.9 in \cite{BGT}, we then show that such solutions admit a boundary trace with compact support in $\xS.$ In addition, we prove that this boundary trace has uniformly bounded total mass. This leads to a contradiction, because, on the other hand, assuming the existence of a non-trivial solution satisfying \eqref{eq:power-a}-\eqref{as1intro}, we may construct solutions satisfying \eqref{eq:power-a}-\eqref{as1intro} and having boundary trace with non-uniformly bounded total mass.

When $\frac{N-\am+1}{N-\am-1}\leq p< \frac{\ap+1}{\ap-1}$ if $\ap>1$ or $\frac{N}{N-2}\leq p$ if $\ap=1$, by combining localization techniques and Keller-Osserman type estimates, we are able to show the removability of isolated singularities.
\begin{theorem} \label{remove-2}
Let $\mu \leq H^2$, $z\in \Sigma$ and $\frac{N-\am+1}{N-\am-1}\leq p< \frac{\ap+1}{\ap-1}$ if $\ap>1$ or $\frac{N}{N-2}\leq p$ if $\ap=1$. We additionally assume that $\xO$ is a $C^3$ bounded domain. If $u\in C^2(\xO)$ is a nonnegative solution of  \eqref{eq:power-a} such that
	\be\label{as2intro}
	\lim_{x\in\xO,\;x\rightarrow\xi}\frac{u(x)}{\tilde W(x)}=0\qquad\forall \xi\in \partial\xO\setminus \{z\},
	\ee
	locally uniformly in $\partial\xO\setminus \{z\}$, then $u\equiv 0.$
\end{theorem}

By Theorem \cite[Theorem 2.13]{BGT}, problem
\be\tag{$\mathrm{BVP}_+$} \label{BP+} \left\{ \BAL
- L_\gm u + |u|^{p-1}u&=0\qquad \text{in }\;\Gw,\\
\tr(u)&=\xn,
\EAL \right.
\ee
admits a unique weak solution for any $\xn\in\mathfrak{M}(\partial\xO)$ if $1<p<$\small{$\min\big\{\frac{N+1}{N-1},\frac{N-\am+1}{N-\am-1}\big\}$}. In order to study the solvability of problem \eqref{BP+} in the supercritical case $p\geq${\small$\min\big\{\frac{N+1}{N-1},\frac{N-\am+1}{N-\am-1}\big \}$}, we will employ an appropriate capacitary framework. To this purpose, let us introduce some notations. For $\theta\in\R$, we define the Bessel kernel of order $\ga$ in $\R^d$ by $\CB_{d,\theta}(\xi):=\CF^{-1}\left((1+|.|^2)^{-\frac{\theta}{2}}\right)(\xi)$, where $\CF$ is the Fourier transform in the space $\CS'(\R^d)$ of moderate distributions in $\BBR^d$.
For $\kappa>1$, the Bessel space $L_{\theta,\kappa}(\BBR^d)$ is defined by
\bal L_{\theta,\kappa}(\BBR^d):=\{f=\CB_{d,\theta} \ast g:g\in L^{\kappa}(\BBR^d)\},
\eal
with norm
\bal \|f\|_{L_{\theta,\kappa}}:=\|g\|_{L^\kappa}=\|\CB_{d,-\theta}\ast f\|_{L^\kappa}.
\eal
The Bessel capacity is defined for
compact set $A \subset\BBR^d$ by
\ba \label{Bessel-cap} \mathrm{Cap}^{\BBR^d}_{\theta,\kappa}(A):=\inf\{\|f\|^\kappa_{L_{\theta,\kappa}}:\; g\in L^\xk_+(\BBR^d),\,f=\CB_{d,\theta} \ast g\geq \1_A \},
\ea
and is extended to open sets and arbitrary sets in $\R^d$ in the standard way. Here $\1_A$ denotes the indicator function of $A$.

We denote by $B^d(x,r)$ the open ball of center $x \in \R^d$ and radius $r>0$ in $\R^d$.

\begin{definition}\label{besselcapacities}
If $\Gamma \subset \partial\xO$ is a $C^2$ submanifold without boundary, of dimension $d$ with $1 \leq d \leq N-1$ then there exist open sets $O_1,...,O_m$ in $\BBR^N$, diffeomorphisms $T_i: O_i \to B^{d}(0,1)\times B^{N-d-1}(0,1)\times (-1,1)$, $i=1,\ldots,m$, and compact sets $K_1,...,K_m$ in $\Gamma$ such that

(i) $K_i \sbs O_i$, $1 \leq i \leq m$ and $ \Gamma= \cup_{i=1}^m K_i$;

(ii) $T_i(O_i \cap \Gamma)=B^d(0,1) \times \{ (x_{d+1},\ldots,x_{N-1}) = 0_{\mathbb{R}^{N-d-1}} \}\times\{x_N=0\}$, $T_i(O_i \cap \Gw)={B^{d}(0,1)}\times B^{N-d-1}(0,1)\times(0,1)$;

(iii) For any $x \in O_i \cap \xO$, there exists $y \in O_i \cap  \xG$ such that $d_\Gamma(x)=|x-y|$ (here $d_\Gamma(x)$ denotes the distance from $x$ to $\Gamma$).

\smallskip

We then define the $\mathrm{Cap}_{\gth,\kappa}^{\Gamma}-$capacity of a compact set $E \sbs \Gamma$ by
\ba \label{Capsub} \mathrm{Cap}_{\gth,\kappa}^{\Gamma}(E):=\sum_{i=1}^m \mathrm{Cap}_{\gth,\kappa}^{\mathbb{R}^d}(\tl T_i(E \cap K_i)),
\ea
where $T_i(E \cap K_i)=\tl T_i(E \cap K_i)  \times \{ (x_{d+1},\ldots,x_{N-1}) = 0_{\mathbb{R}^{N-d-1}} \}\times\{x_N=0\}$.
\end{definition}
\noindent We remark that the definition of the capacities does not depend on $O_i$, $i=1,\ldots,m$.

In the sequel, we will say that $\nu \in \GTM^+(\partial \Omega)$ is \textit{absolutely continuous} with respect to a capacity $\mathscr{C}$ if
\ba \label{abs-cont}
	\forall \text{ Borel set } E \subset \partial \Omega  \text{ such that } \mathscr{C}(E)=0\Longrightarrow \xn(E)=0.
\ea

We distinguish two cases according to whether the support of the given boundary datum $\nu$ is in $\Sigma$ or in $\partial\xO \setminus \Sigma$. Recall that $\ap$ and $\am$ are defined in \eqref{apmintro}.

\begin{theorem}\label{supcrK} Assume $k \geq 1$, $\mu \leq H^2$, $\frac{N-\am+1}{N-\am-1}\leq p<\frac{\ap+1}{\ap-1}$ and $\xn\in\mathfrak M^+(\partial\Gw)$ with compact support in $\Sigma$.
If $\nu$ is absolutely continuous with respect to  $\mathrm{Cap}^{\xS}_{\vartheta,p'}$, where $p'=\frac{p}{p-1}$, then problem \eqref{BP+} admits a unique weak solution.
\end{theorem}

Let us comment on the proof of the above theorem. In the spirit of \cite{MV-Pisa} (see also \cite{GkiNg_absorption}), we first show that $\mathbb{K}_{\xm}[\xn]\in L^p(\xO;\ei)$ if and only if $\xn$ belongs to the dual of a certain Besov space (see Theorem \ref{potest}). This allows us to utilize an approximation argument in order to prove the existence of a unique weak solution to problem \eqref{BP+}.

In the case where $\xn$ has compact support in $\partial\xO\setminus\xS$, the critical exponent as well as the associated Bessel capacities are different from those in Theorem \ref{supcrK}.
\begin{theorem}\label{supcromega} Assume $\mu \leq H^2$, $p \geq \frac{N+1}{N-1}$ and $\xn\in\mathfrak M^+(\partial\Gw).$  Then problem \eqref{BP+} admits a unique weak solution $u$ with $\tr(u)=\1_{\partial\xO\setminus\xS}\nu$ if and only if $\1_{F}\nu$ is absolutely continuous with respect to  $\mathrm{Cap}^{\partial\xO}_{\frac{2}{p},p'}$ for any compact set $F\subset \partial\xO\setminus\xS.$
\end{theorem}

Next we discuss existence results for problem
\be\tag{$\mathrm{BVP}_-^{\sigma}$} \label{BP-} \left\{ \BAL
- L_\gm u - |u|^{p-1}u&=0\qquad \text{in }\;\Gw,\\
\tr(u)&=\sigma \xn,
\EAL \right.
\ee
where $\sigma>0$ is a parameter and $\nu \in \GTM^+(\partial \Omega)$.

In the following theorems, for any $\xn\in\mathfrak{M}^+(\partial\xO)$, we extend it to be a measure defined on $\overline{\xO}$ by setting $\xn(\xO)=0$ and use the same notation $\nu$ for the extension.

We first consider the case where $\nu$ is concentrated on $\Sigma$.

\begin{theorem}\label{subm}
	Assume that $\mu< \frac{N^2}{4}$ and
\ba\label{p-cond}
1<p<\frac{\am+1}{\am-1}\;\; \text{if}\;\; \am>1\quad \text{or}\quad p>1 \;\;\text{if}\;\; \am\leq1.
\ea
	 Let $\xn \in \GTM^+(\partial\xO)$ with compact support in $\xS$.
	Then the following statements are equivalent.
	
	1. Problem \eqref{BP-} has a positive weak solution for $\gs>0$ small.
	
	2. For any Borel set $E \subset \overline{\xO} $, there holds
	\ba \label{Kp<nu} \int_E \BBK_\xm[\1_E\xn]^p \ei \,\dd x \leq C\,\xn(E).
	\ea
	
	3. The following inequality holds
	\ba \label{GKp<K} \BBG_\xm[\BBK_\xm[\xn]^p]\leq C\,\BBK_\xm[\xn]<\infty\quad \text{a.e. in } \Omega.
	\ea

    Assume, in addition, that
	\be \label{p-cond-2}
	k \geq 1 \quad \text{and} \quad \max\left\{1,\frac{N-k-\am}{N-2-\am} \right\}< p<\frac{\ap+1}{\ap-1}.
\ee
 Put
	\ba\label{gamma} \vartheta: = \frac{\ap+1-p(\ap-1)}{p}.
\ea
Then any of the above statements is equivalent to the following statement
	
	4. For any Borel set $E \subset  \Sigma$, there holds
	\bal\xn(E)\leq C\, \mathrm{Cap}_{\vartheta,p'}^{\xS}(E).\eal
\end{theorem}

Next we focus on the case that $\nu$ is concentrated on $\partial \Omega \setminus \Sigma$.

\begin{theorem}\label{th:existnu-prtO}
	Assume that $\mu\leq \frac{N^2}{4}$, $p$ satisfies \eqref{p-cond} and  $\xn \in \GTM^+(\partial\xO)$ with compact support in $\partial \xO\setminus\xS$.
	Then the following statements are equivalent.
	
	1. Equation \eqref{BP-} has a positive solution for $\gs>0$ small.
	
	2. For any Borel set $E \subset  \overline{\Omega} $, \eqref{Kp<nu} holds.
	
	3. Estimate \eqref{GKp<K} holds.
	
	4. For any Borel set $E \subset \partial \Omega $, there holds $\xn(E)\leq C\, \mathrm{Cap}_{\frac{2}{p},p'}^{\partial \Omega}(E)$.	
\end{theorem}

\bigskip

\noindent \textbf{Organization of the paper.} In Section \ref{Preliminaries}, we provide the local representation of $\Sigma$ and $\Omega$, quote two-sided estimates of the Green function and the Martin kernel from \cite{BGT}, introduce the notion of boundary trace and recall results for linear and semilinear equations with an absorption established in \cite{BGT}. In Section \ref{harmesure}, we show the relation between the $L_\mu$-harmonic measures and the surface measure on $\partial \Omega$ and prove identities regarding the Poisson kernel and Martin kernel. Section \ref{sec:BVP-abs} is devoted to the derivation of various results for equations with a power absorption term such as a prior estimates, removable singularities and existence results. In Section \ref{sec:powercase}, we focus on equations with a power source term and demonstrate necessary and sufficient conditions for the existence of a weak solution. Finally, in the Appendix \ref{app:A}, we construct a barrier function for solutions under assumption that $\Omega$ is a $C^3$ bounded domain. \medskip

\noindent \textbf{Notations.} We denote by $c,c_1,C...$ the constant which depend on initial parameters and may change from one appearance to another. The notation $A \gtrsim B$ (resp. $A \lesssim B$) means $A \geq c\,B$ (resp. $A \leq c\,B$) where the implicit $c$ is a positive constant depending on some initial parameters. If $A \gtrsim B$ and $A \lesssim B$, we write $A \approx B$. \textit{Throughout the paper, most of the implicit constants depend on some (or all) of the initial parameters such as $N,\Omega,\Sigma,k,\mu$ and we will omit these dependencies in the notations (except when it is necessary).} For a set $D \subset \R^N$, $\1_D$ denotes the indicator function of $D$. \medskip

\noindent \textbf{Acknowledgement.} K. T. Gkikas acknowledges support by the Hellenic Foundation for Research and Innovation
(H.F.R.I.) under the “2nd Call for H.F.R.I. Research Projects to support Post-Doctoral Researchers” (Project Number: 59).
 P.-T. Nguyen was supported by Czech Science Foundation, Project GA22-17403S.

\section{Preliminaries} \label{Preliminaries}
\subsection{Local representation of $\Sigma$ and $\partial\xO$.} \label{assumptionK} In this subsection, we present the local representation of $\xS$ and $\partial\xO.$

Let $k \in \N$ such that $0 \leq k \leq N-1$. For $x=(x_1,...,x_k,x_{k+1},...,x_N) \in \R^N$, we write $x=(x',x'')$ where $x'=(x_1,..,x_k) \in \R^k$ and $x''=(x_{k+1},...,x_N) \in \R^{N-k}$. For $\beta>0$, we denote by $B(x,\beta)$ the ball in $\R^N$ with center $x \in \R^N$ and radius $\beta$, and by $B^k(x',\beta)$ the ball in $\R^k$ with center at $x' \in \R^k$ and radius $\beta$. For any $\xi\in \Sigma$, we set
\bal  \Sigma_\beta &:=\{ x \in \R^N \setminus \Sigma: d_\Sigma(x) < \beta \}, \\
\label{Vxi}
V_{\Sigma}(\xi,\xb)&:=\{x=(x',x''): |x'-\xi'|<\beta,\; |x_i-\Gamma_i^\xi(x')|<\xb,\;\forall i=k+1,...,N\},
\eal
for some functions $\Gamma_i^\xi: \R^k \to \R$, $i=k+1,...,N$.

Since $\Sigma$ is a $C^2$ compact submanifold in $\mathbb{R}^N$ without boundary, we may assume the existence of $\xb_0$ such that the followings hold.

\begin{itemize}
\item $\Sigma_{6\beta_0}\Subset \Omega$ and for any $x\in \Sigma_{6\beta_0}$, there is a unique $\xi \in \Sigma$  satisfying $|x-\xi|=d_\Sigma(x)$.

\item $d_\Sigma \in C^2(\Sigma_{4\beta_0})$, $|\nabla d_\Sigma|=1$ in $\Sigma_{4\beta_0}$ and there exists $\eta \in L^\infty(\Sigma_{4\beta_0})$ such that
\be \label{laplaciand}
\Delta d_\Sigma(x)=\frac{N-k-1}{d_\Sigma(x)}+ \eta(x) \quad \text{in } \Sigma_{4\beta_0} .
\ee
(See \cite[Lemma 2.2]{Vbook} and \cite[Lemma 6.2]{DN}.)

\item For any $\xi \in \Sigma$, there exist $C^2$ functions $\Gamma_{i,\xS}^\xi \in C^2(\R^k;\R)$, $i=k+1,...,N$, such that for any $\beta \in (0,6\beta_0)$ and $V_\xS(\xi,\beta) \subset \Omega$  (upon relabeling and reorienting the coordinate axes if necessary), there holds
\be \label{straigh}
V_\xS(\xi,\beta) \cap \Sigma=\{x=(x',x''): |x'-\xi'|<\beta,\;  x_i=\Gamma_{i,\xS}^\xi (x'), \; \forall i=k+1,...,N\}.
\ee

\item There exist $m_0 \in \N$ and points $\xi^{j} \in \Sigma$, $j=1,...,m_0$, and $\beta_1 \in (0, \beta_0)$ such that
\be \label{cover}
\Sigma_{2\xb_1}\subset \cup_{j=1}^{m_0} V_{\xS}(\xi^j,\beta_0).
\ee
\end{itemize}

Now for $\xi \in \Sigma$, set
\bel{dist2} \xd_\Sigma^\xi(x):=\left(\sum_{i=k+1}^N|x_i-\Gamma_{i,\xS}^\xi(x')|^2\right)^{\frac{1}{2}}, \qquad x=(x',x'')\in V_\xS(\xi,4\beta_0).\ee

Then we see that there exists a constant $C_\xS$ depending only on  $N,\Sigma$ such that
\be\label{propdist}
d_\Sigma(x)\leq	\xd_\Sigma^{\xi}(x)\leq C_\xS \| \Sigma \|_{C^2} d_\Sigma(x),\quad \forall x\in V_\xS(\xi,2\beta_0),
\ee
where $\xi^j=((\xi^j)', (\xi^j)'') \in \Sigma$, $j=1,...,m_0$, are the points in \eqref{cover} and
\be \label{supGamma}
\| \Sigma \|_{C^2}:=\sup\{  || \Gamma_{i,\xS}^{\xi^j} ||_{C^2(B_{5\beta_0}^k((\xi^j)'))}: \; i=k+1,...,N, \;j=1,...,m_0 \} < \infty.
\ee
Moreover, $\beta_1$ can be chosen small enough such that for any $x \in \Sigma_{\beta_1}$,
\bel{BinV} B(x,\beta_1) \subset V_\xS(\xi,\beta_0),
\ee
where $\xi \in \Sigma$ satisfies $|x-\xi|=d_\Sigma(x)$.

When $\xS=\partial\xO$, we assume that
\ba \label{straigh2}
V_{\partial\xO}(\xi,\beta) \cap \xO=\bigg\{x=(x_1,\ldots,x_N): \sum_{i=1}^{N-1}|x_i-\xi_i|^2<\beta^2,\;
0< x_N -\Gamma_{N,\partial\xO}^\xi (x_1,...,x_{N-1}) <\beta \bigg \}.
\ea
We also find that \eqref{straigh} with $\Sigma=\partial \Omega$ becomes
\ba \label{straigh2-b}
V_{\partial\xO}(\xi,\beta) \cap \partial \Omega=\bigg\{x=(x_1,\ldots,x_N): \sum_{i=1}^{N-1}|x_i-\xi_i|^2<\beta^2,\;
x_N = \Gamma_{N,\partial\xO}^\xi (x_1,...,x_{N-1}) \bigg \}.
\ea
Thus, when  $\xS\subset\partial\xO$ is a $C^2$ compact submanifold in $\mathbb{R}^N$ without boundary, of dimension $0 \leq k \leq N-1$, for any $x \in \Sigma$, we have that
\ba
\xG_{N,\xS}^\xi(x')=\xG_{N,\partial\xO}^\xi(x',\xG_{k+1,\xS}^\xi(x'),...,\xG_{N-1,\xS}^\xi(x')).\label{dist3}
\ea

Let $\xi\in \xS$. For any $x\in V_\xS(\xi,\xb_0)\cap \xO,$ we define
\ba \label{dxi}
\xd^{\xi}(x):=x_N-\Gamma_{N,\partial\xO}^\xi (x_1,...,x_{N-1}),
\ea
and
\ba \label{d2Sigma} \xd_{2,\xS}^{\xi}(x):=\left(\sum_{i=k+1}^{N-1}|x_i-\Gamma_{i,\xS}^\xi(x')|^2 \right)^{\frac{1}{2}} .
\ea
Then by \eqref{dist3}, there exists a constant $A>1$ which depends only on $N$, $k$,
$\|\Sigma\|_{C^2}$, $\|\partial\xO\|_{C^2}$ and $\beta_0$ such that
\ba \label{propdist2}
A^{-1}(\xd_{2,\xS}^\xi(x)+\xd^\xi(x))\leq\xd_\xS^\xi(x)\leq A(\xd_{2,\xS}^\xi(x)+\xd^\xi(x)), \quad \forall x \in V_{\Sigma}(\xi,\beta_0) \cap \Omega.
\ea
Thus by \eqref{propdist} and \eqref{propdist2}, for any $\xg\in\BBR,$ we can show that there exists a constant $C>1$ which depends on $ N,k,\|\Sigma\|_{C^2}, \|\partial\xO\|_{C^2},\beta_0,\xg $ such that
\be
C^{-1}\xd(x)^2(\xd_{2,\xS}(x)+\xd(x))^\xg\leq d_{\partial \Omega}(x)^2d_\xS(x)^\gamma \leq C\xd(x)^2(\xd_{2,\xS}(x)+\xd(x))^\xg.\label{ddK}
\ee

We set
\begin{align} \nonumber
\mathcal{V}_\xS(\xi,\xb_0):=\left\{(x',x''): |x'-\xi'|<\beta_0,\; |\xd(x)|<\xb_0,\; |\xd_{2,\xS}|<\xb_0\right\}.
\end{align}
We may then assume that
\begin{align*}
\mathcal{V}_\xS(\xi,\xb_0)\cap\xO=\left\{(x',x''): |x'-\xi'|<\beta_0, 0<\xd(x)<\xb_0,\; |\xd_{2,\xS}|<\xb_0\right\}
\end{align*}
and there exist $m_1 \in \N$ and points $\xi^{j} \in \Sigma$, $j=1,...,m_1$, and $\beta_1 \in (0, \beta_0)$ such that
\be \label{cover2}
\Sigma_{2\xb_1} \cap \Omega \subset \cup_{j=1}^{m_1} \mathcal{V}_{\xS}(\xi^j,\beta_0)\cap\xO.
\ee

We can choose $\xb_1$ small enough such that for any $x\in \xO_{\xb_1}$ there exists a unique $\xi_x\in\partial\xO,$ which satisfies $d_{\partial \Omega}(x)=|x-\xi_x|.$ Now set
\ba\label{newdistance}
\tilde{d}_{\xS}(x):=\sqrt{|{\rm dist}^{\partial\xO}(\xi_x,\xS)|^2+|x-\xi_x|^2},
\ea
where $\dist^{\partial \Omega}$ denotes the geodesic distance on $\partial \Omega$.

The properties of $\tilde d_{\Sigma}$ is provided in the following result.

\begin{proposition}[{\cite[Lemma 2.1]{FF}}]\label{expansion} There exists $\xb_1=\xb_1(\xS,\xO)$ small enough such that, for any $x\in \xO\cap \xS_{\xb_1},$ the following expansions hold
\bal
&\tilde{d}_{\xS}^2(x) =d^2_{\xS}(x)(1+f_1(x)), \\
&\nabla d_{\partial \Omega}(x) \cdot \nabla \tilde{d}_{\xS}(x)=\frac{d_{\partial \Omega}(x)}{\tilde{d}_{\xS}(x)}, \\
&|\nabla\tilde{d}_{\xS}(x)|^2=1+f_2(x), \\
&\tilde{d}_{\xS}(x)\xD \tilde{d}_{\xS}(x)=N-k-1+ f_3(x),
\eal
where $f_i$, $i=1,2,3$, satisfy
\be \label{frag}
\sum_{i=1}^3 |f_i(x)|\leq C_1(\xb_1,N) \tilde{d}_{\xS}(x),\quad\forall x\in\xO\cap \xS_{\xb_1}.
\ee

\label{propdK}
\end{proposition}

\subsection{Two-sided estimates on Green function and Martin kernel} \label{subsec:GreenMartin} In this subsection, we recall sharp two-sided estimates on the Green function $G_\xm$ and the Martin kernel $K_\xm$ associated to $-L_\mu$ in $\Omega$, as well as the representation formula for nonnegative $L_\xm$-harmonic functions.

Estimates on the Green function and the Martin kernel are stated in the following Propositions.

\begin{proposition}[{\cite[Proposition 5.3]{BGT}}] \label{Greenkernel} Assume that $\xm \leq \frac{(N-k)^2}{4}$ and $\lambda_{\mu,\Sigma}>0$.
	
$(i)$ If $\am<\frac{N-k}{2}$ or $\am=\frac{N-k}{2}$ and $k\neq 0,$ then for any $x,y\in\xO$ and $x\neq y,$ there holds
\ba \label{Greenesta}
	G_{\xm}(x,y)&\approx
	\min\left\{\frac{1}{|x-y|^{N-2}},\frac{d_{\partial \Omega}(x)d_{\partial \Omega}(y)}{|x-y|^N}\right\} \left(\frac{\left(d_\xS(x)+|x-y|\right)\left(d_\xS(y)+|x-y|\right)}{d_\xS(x)d_\xS(y)}\right)^{\am}.
\ea

$(ii)$ If $\am=\frac{N}{2}$ and $k=0,$ then for any $x,y\in\xO$ and $x\neq y,$ there holds
\ba \label{Greenestb}\BAL
	G_{\xm}(x,y)&\approx
\min\left\{\frac{1}{|x-y|^{N-2}},\frac{d_{\partial \Omega}(x)d_{\partial \Omega}(y)}{|x-y|^N}\right\}\left(\frac{\left(|x|+|x-y|\right)\left(|y|+|x-y|\right)}{|x||y|}\right)^{-\frac{N}{2}}\\
	&\quad  +\frac{d_{\partial \Omega}(x)d_{\partial \Omega}(y)}{(|x||y|)^{\frac{N}{2}}}\left|\ln\left(\min\left\{|x-y|^{-2},(d_{\partial \Omega}(x)d_{\partial \Omega}(y))^{-1}\right\}\right)\right|.
\EAL
\ea
\end{proposition}

\begin{proposition}[{\cite[Theorem 2.8]{BGT}}]\label{Martin} ~~
Assume that $\xm \leq \frac{(N-k)^2}{4}$ and $\lambda_{\mu,\Sigma}>0$.

$(i)$ If $\xm <\frac{(N-k)^2}{4}$ or $\xm=\frac{(N-k)^2}{4}$ and $k>0$ then
\bel{Martinest1}
K_{\xm}(x,\xi) \approx\frac{d_{\partial \Omega}(x)}{|x-\xi|^N}\left(\frac{\left(d_\xS(x)+|x-\xi|\right)^2}{d_\xS(x)}\right)^{\am} \quad \text{for all } x \in \Omega, \, \xi \in \partial\xO.
\ee

$(ii)$ If  $\xm=\frac{N^2}{4}$ and $k=0$ then
\bel{Martinest2}
K_{\mu}(x,\xi) \approx \frac{d_{\partial \Omega}(x)}{|x-\xi|^N}\left(\frac{\left(|x|+|x-\xi|\right)^2}{|x|}\right)^{\frac{N}{2}}
+\frac{d_{\partial \Omega}(x)}{|x|^{\frac{N}{2}}}\left|\ln\left(|x-\xi|\right)\right|,\quad
\text{ for all } x \in \Omega, \, \xi \in \partial\xO.
\ee
\end{proposition}

Recall that the Green operator and Martin operator are respectively defined by
\ba \label{Grt}
\BBG_\mu[\tau](x)=\int_{\xO }G_{\mu}(x,y) \, \dd\tau(y), \quad \tau \in \GTM(\Omega;\ei), \\
\label{Mrt} \mathbb{K}_\mu[\gn](x)=\int_{\partial\xO }K_{\mu}(x,y) \, \dd\xn(y), \quad \gn\in \mathfrak{M}(\partial\xO).
\ea

A function $u \in L_{\loc}^1(\Omega)$ is called an {\it $L_\mu$-harmonic function} in $\Omega$ if $L_\mu u = 0$ in the sense of distributions in $\Omega$.

Next we state the representation theorem which provides a bijection between the class of positive $L_\mu$-harmonic functions in $\Omega$ and $\GTM^+(\partial \Omega)$.

\begin{theorem}[{\cite[Theorem 2.9]{BGT}}] \label{th:Rep} For any $\nu \in \GTM^+(\partial \Omega )$, the function $\BBK_{\mu}[\nu]$ is a positive $L_\mu$-harmonic function in $\Omega $. Conversely, for any positive $L_\mu$-harmonic function $u$  in $\Omega $, there exists a unique measure $\nu \in \GTM^+(\partial \Omega)$ such that $u=\BBK_{\mu}[\nu]$.
\end{theorem}


\subsection{The boundary trace of Green operator and Martin operator} \label{subsec:boundarytrace} We start this Subsection by presenting the notion of the boundary trace introduced in \cite{BGT}.

Let $h\in C(\partial\xO)$ and $u_h$ be the solution of \eqref{linear}. Set $\CL_{\mu,z}(h):=u_h(z),$ then the mapping $h\mapsto \CL_{\mu,z}(h)$ is a linear positive functional on $C(\partial\Omega )$. Thus there exists a unique Borel measure on $\partial\Omega $, called {\it $L_{\mu}$-harmonic measure on $\partial \Omega $ relative to $z$} and  denoted by $\omega_{\Omega }^{z}$, such that
\bal
v_{h}(z)=\int_{\partial\Omega}h(y) \, \dd\omega_{\Omega}^{z}(y).
\eal
Let $x_0 \in \Omega $ be a fixed reference point. Let $\{\xO_n\}$ be a $C^2$ \textit{exhaustion} of $\Omega$, i.e. $\{\Omega_n\}$ is an increasing sequence of bounded $C^2$ domains  such that
\ba\label{Omegan}  \overline{\xO}_n\subset \xO_{n+1}, \quad \cup_n\xO_n=\xO, \quad \mathcal{H}^{N-1}(\partial \Omega_n)\to \mathcal{H}^{N-1}(\partial \Omega),
\ea
where $\mathcal{H}^{N-1}$ denotes the $(N-1)$-dimensional Hausdorff measure in $\R^N$.

Then $-L_\mu$ is uniformly elliptic and coercive in $H^1_0(\xO_n)$ and its first eigenvalue $\lambda_{\mu,\Sigma}^{\xO_n}$ in $\xO_n$ is larger than its first eigenvalue $\lambda_{\mu,\Sigma}$ in $\Omega$.

For $h\in C(\prt \xO_n)$, the following problem
\be\label{sub12} \left\{ \BAL
-L_{\xm } v&=0\qquad&&\text{in } \xO_n\\
v&=h\qquad&&\text{on } \prt \xO_n,
\EAL \right.
\ee
admits a unique solution which allows to define the $L_{\xm }$-harmonic measure $\omega_{\xO_n}^{x_0}$ on $\prt \xO_n$
by
\be\label{redu2}
v(x_0)=\myint{\prt \xO_n}{}h(y) \,\dd\gw^{x_0}_{\xO_n}(y).
\ee
Let $G^{\xO_n}_\xm(x,y)$ be the Green kernel of $-L_\mu$ on $\xO_n$.  Then $G^{\xO_n}_\xm(x,y)\uparrow G_\mu(x,y)$ for $x,y\in\xO, x \neq y$.

We recall below the definition of boundary trace in \cite{BGT} which is defined in a \textit{dynamic way}.

\begin{definition}[Boundary trace] \label{nomtrace}
We say that	a function $u\in W^{1,\kappa}_{\loc}(\xO)$ ($\kappa>1$) possesses a \emph{boundary trace}  if there exists a measure $\nu \in\GTM(\partial \Omega )$ such that for any $C^2$ exhaustion  $\{ \xO_n \}$ of $\Omega$, there  holds
	\be\label{trab}
	\lim_{n\rightarrow\infty}\int_{ \partial \xO_n}\phi u\, \dd \omega_{\xO_n}^{x_0}=\int_{\partial \Omega} \phi \,\dd \nu \quad\forall \phi \in C(\overline{\Omega}).
	\ee
	The boundary trace of $u$ is denoted by $\tr(u)$.
\end{definition}

In the following two propositions, we recall some properties of the boundary trace.
\begin{proposition}[{\cite[Proposition 7.7]{BGT}}]\label{22222}
Assume  $x_0\in \xO_1$. Then for every $Z\in C(\overline{\xO}),$
\be\label{2.27}
\lim_{n\rightarrow\infty}\int_{\partial \xO_n}
Z(x)\tilde{W}(x)\, \dd\gw^{x_0}_{\Gw_n}(x)=\int_{\partial \xO}Z(x)\, \dd\gw^{x_0}_\xO(x).
\ee
\end{proposition}

\begin{proposition}[{\cite[Lemmas 8.1 and 8.2]{BGT}}] \label{traceKG} ~~
	
	$(i)$ For any $\nu \in \GTM(\partial \Omega )$, $\tr(\BBK_{\mu}[\nu])=\nu$.
	
	$(ii)$ For any $\tau \in \GTM(\Omega;\ei)$, $\tr(\BBG_\mu[\tau])=0$.
\end{proposition}


\subsection{Linear boundary value problems}
\label{subsec:linear}
\begin{definition}
 Let $\tau\in\mathfrak{M}(\xO;\ei)$ and $\nu \in \mathfrak{M}(\partial\xO)$. We say that $u$ is a \textit{weak solution} of
\be\label{NHL} \left\{ \BAL
- L_\gm u&=\tau\qquad \text{in }\;\Gw,\\
\tr(u)&=\xn,
\EAL \right.
\ee
if $u\in L^1(\xO;\ei)$ and $u$ satisfies
\be \label{lweakform}
	- \int_{\Gw}u L_{\xm }\xi \, \dd x=\int_{\Gw } \xi \, \dd \tau - \int_{\Gw} \mathbb{K}_{\xm}[\xn]L_{\xm }\xi \, \dd x
	\qquad\forall \xi \in\mathbf{X}_\xm(\xO),
	\ee
where ${\bf X}_\mu(\Gw)$ has been defined in \eqref{Xmu}.
\end{definition}

The next result provides the existence, uniqueness and Kato type inequalities for solutions to problem \eqref{NHL}.

\begin{theorem}[{\cite[Theorem 2.12]{BGT}}] \label{linear-problem}
Let $\tau\in\mathfrak{M}(\xO;\ei)$ and $\xn \in \mathfrak{M}(\partial\xO)$.  Then there exists a unique weak solution $u\in L^1(\xO;\ei)$ of \eqref{NHL}. Furthermore
\be \label{reprweaksol}
u=\mathbb{G}_{\mu}[\tau]+\mathbb{K}_{\xm}[\xn] \quad \text{a.e. in } \Omega,
\ee
and for any $\zeta \in \mathbf{X}_\xm(\xO),$ there holds
\ba \label{esti2}
\|u\|_{L^1(\Omega;\ei)} \leq \frac{1}{\lambda_\mu}\| \tau \|_{\GTM(\Omega;\ei)} + C \| \nu \|_{\GTM(\partial\Omega)},
\ea
where $C=C(N,\Omega,\Sigma,\mu)$.
In addition, if $\dd \tau=f\dd x+\dd\rho$ with $\rho\in\mathfrak{M}(\xO;\ei)$ and $f\in L^1(\xO;\ei)$, then, for any $0 \leq \zeta \in \mathbf{X}_\xm(\xO)$, the following Kato type inequalities are valid
\be\label{poi4}
-\int_{\Gw}|u|L_{\xm }\zeta \, \dd x\leq \int_{\Gw}\sign(u)f\zeta\, \dd x +\int_{\Gw}\zeta \, \dd|\rho|-
\int_{\Gw}\mathbb{K}_{\xm}[|\xn|] L_{\xm }\zeta \, \dd x,
\ee
\be\label{poi5}
-\int_{\Gw}u^+L_{\xm }\zeta \, \dd x\leq \int_{\Gw} \sign^+(u)f\zeta\, \dd x +\int_{\Gw }\zeta\, \dd\rho^+-
\int_{\Gw}\mathbb{K}_{\xm}[\nu^+]L_{\xm }\zeta \,\dd x.
\ee
Here $u^+=\max\{u,0\}$.
\end{theorem}

\subsection{Semilinear problem with an absorption}

In this subsection, we give a sufficient condition for the existence of solutions to the problem

\be\label{nonlinear} \left\{ \BAL
- L_\gm u+g(u)&=0\qquad \text{in }\;\Gw,\\
\tr(u)&=\xn,
\EAL \right.
\ee
where $g\in C(\BBR)$ is a nondecreasing function such that $g(0)=0.$

Let us give the definition of weak solutions of \eqref{nonlinear}.

\begin{definition} \label{def:weak-sol} Let $g\in C(\BBR)$ and $\xn\in \mathfrak{M}(\partial\xO).$ A function $u$ is a \textit{weak solution} of \eqref{nonlinear}, if $u\in L^1(\Omega;\ei)$, $g(u) \in L^1(\Omega;\ei)$  and
	\be \label{nlinearweakform}
	- \int_{\xO}u L_{\xm }\zeta \, \dd x+ \int_{\xO}g(u)\zeta \, \dd x=- \int_{\Gw} \mathbb{K}_{\xm}[\xn]L_{\xm }\zeta \, \dd x
	\qquad\forall \zeta \in\mathbf{X}_\xm(\xO).
	\ee
\end{definition}

An existence result for problem \eqref{nonlinear} was proved in \cite{BGT}.
\begin{proposition}[{\cite[Theorem 9.7]{BGT}}]\label{subcr}
Let $\xn\in \mathfrak{M}(\partial\xO)$ and $g\in C(\BBR)$ be a nondecreasing function such that $g(0)=0$ and $g(\mathbb{K}_{\xm}[\xn_+]),g(\mathbb{K}_{\xm}[\xn_-])\in L^1(\xO;\ei).$ Then there exists a unique weak solution of \eqref{nonlinear}.
\end{proposition}

\section{The $L_\mu$-harmonic measure and the surface measures on $\partial\Omega$} \label{harmesure} In this section, we study the relation between the $L_\mu$-harmonic measure and the $(N-1)$-dimensional surface measure $S_{\partial \Omega}$ on $\partial\Omega$. To this purpose, we need to introduce the notion of Poisson kernel.

By the standard elliptic theory, we can easily show that for any $x \in \xO$,  $G_{\mu}(x,\cdot) \in C^{1,\gamma}(\overline{\Omega} \setminus (\Sigma \cup \{x\})) \cap C^2(\Omega \setminus  \{x\})$ for all $\xg\in (0,1)$. Therefore, we may define the \textit{Poisson kernel} associated to $-L_\mu$ in $\Omega \times (\partial \Omega \setminus \Sigma)$ as
\be
\label{Poisson}  P_\mu(x,y):=-\frac{\partial G_\mu}{\partial{\bf n}}(x,y), \quad x \in \Omega , \; y \in \partial \Omega\setminus\xS,
\ee
where ${\bf n}$ is the unit outer normal vector of $\partial \xO$. This kernel satisfies the following properties.

\begin{proposition}  \label{measureboundary}
Let $x_0 \in \Omega $ be a fixed reference point.

$(i)$ Then there holds
\ba\label{KP}
P_\mu(x,y)=P_\mu(x_0,y)K_\mu(x,y), \quad x \in \Omega , \; y \in \partial \Omega\setminus\xS.
\ea

$(ii)$ For any $h \in L^1(\partial \Omega ; \dd \omega^{x_0}_{\xO})$ with compact support in $\partial \Omega\setminus\xS$, there holds
\be \label{Pformula}
\int_{\partial\Omega} h(y) \, \dd \omega^{x_0}_{\xO}(y) = \BBP_\mu[h\tilde W](x_0).
\ee
Here
\be \label{BBP} \mathbb{P}_{\mu}[h\tilde W](x):=\int_{\partial\Omega}P_\mu(x,y)h(y) \tilde W(y)\, \dd S_{\partial \Omega}(y), \quad x \in \Omega,
\ee
where $S_{\partial \Omega}$ is the $(N-1)$-dimensional surface measure on $\partial\Omega$.
\end{proposition}
\begin{proof}[\textbf{Proof}]
(i) We note that
$P_\mu(\cdot,y)$ is $L_\mu$-harmonic in $\xO$ and
\bal
\lim_{x\in\xO,\;x\rightarrow\xi}\frac{P_\mu(x,y)}{\tilde W(x)}=0\qquad\forall \xi\in \partial\xO\setminus \{y\}\quad\text{and}\quad y\in\partial\xO\setminus\xS.
\eal
Hence, $\frac{P_\mu(x,y)}{P_\mu(x_0,y)}$ is a kernel function with pole at $y$ and
 basis at $x_0$ in the sense of \cite[Definition 2.7]{BGT}. This, together with the fact that any kernel function with pole at $y$ and
 basis at $x_0$ is unique (see \cite[Proposition 7.3]{BGT}), implies \eqref{KP}.

(ii) Let $\xz\in C(\partial\xO)$ with compact support in $\partial\xO\setminus\xS$ such that $\dist(\supp \xz,\xS)=r>0$. Let $Z\in C(\overline{\xO})$ be such that $Z(y)=\xz(y)$ for any $y\in \partial\xO$ and $Z(y)=0$ in $\xS_{\frac{r}{2}}.$  Set $r_0=\frac{1}{4}\min\{\xb_1,r\}$ where $\xb_1$ is the constant in \eqref{cover2}. We consider a decreasing sequence of bounded $C^2$ domains $\{\Sigma_n\}$ such that
\ba\label{Kn} \Sigma \subset \Sigma_{n+1}\subset\overline{\Sigma}_{n+1}\subset \Sigma_{n}\subset\overline{\Sigma}_{n} \subset\xS_{\frac{r_0}{4}}, \quad \cap_n \Sigma_n=\Sigma.
\ea

Let $\phi_*$ be the unique solution of
\be\label{fi} \left\{  \BAL
-L_{\mu }\phi_*&=0 \qquad&&\text{in } \xO\\
\phi_* &=1\qquad&&\text{on } \prt\Omega,
\EAL \right.
\ee
where the boundary condition in \eqref{fi} is understood as
\bal
\lim_{\dist(x,F)\to 0}\frac{\phi_*(x)}{\tilde W(x)}=1 \quad \text{for every compact set } \; F\subset \partial \Omega.
\eal
Then, by \cite[Lemma 6.8 and estimate (6.21)]{BGT}, there exist constants $c_1=c_1(\Omega,\Sigma,\Sigma_n,\mu)$ and $c_2=c_2(\Omega,\Sigma,N,\mu)$ such that
$0<c_1 \leq\phi_*(x)\leq c_2 d_\Sigma(x)^{-\ap}$ for all $x \in \Omega\setminus\xS_n$.
By the standard elliptic theory, $\xf_*\in C^2(\Omega ) \cap C^{1,\xg}(\overline{\xO}\setminus\xS)$ for any $0<\xg<1$.

Now, for any $\eta\in C(\partial\xO),$ we can easily show that $u_\eta$ is a solution of
\bal \left\{  \BAL
-L_{\mu }v&=0 \qquad&&\text{in } \xO\setminus\xS_n\\
v&=\eta\qquad&&\text{on } \prt(\xO\setminus \Sigma_n)
\EAL \right.
\eal
if and only if $w_\eta=\frac{u_\eta}{\xf_*}$ is a solution of
\bal \left\{  \BAL
-\text{div}(\phi_*^2\nabla w)&=0 \qquad&&\text{in }  \xO\setminus\xS_n\\
w&=\frac{\eta}{\phi_*}\qquad&&\text{on } \prt( \xO\setminus\xS_n).
\EAL \right.
\eal
Since the operator $L_{\xf_*}w:=-\text{div}(\phi_*^2\nabla w)$ is uniformly elliptic and has smooth coefficients, we may deduce the existence of
$L_\mu$-harmonic measure $\omega_n^{x}$  on $\partial ( \xO\setminus\xS_n)$ and the Green kernel $G^n_{\mu}$ of $-L_\mu$ in $\Omega \setminus\xS_n$.

Let $v_n$ be the unique solution of

\be\label{harmonicrel1} \left\{  \BAL
-L_{\mu }v&=0 \qquad&&\text{in } \xO\setminus\xS_n\\
v&=Z\tilde W\qquad&&\text{on } \prt(\xO\setminus \Sigma_n).
\EAL \right.
\ee
Then by the representation formula we have
\be \label{wnrep1}
v_n(x)=\int_{\partial( \xO\setminus\xS_n)}Z\tilde W\, \dd\xo_n^x(y)=\int_{\partial\xO\cap\supp \xz}\xz\tilde W\, \dd\xo_n^x(y).
\ee
Proceeding as in the proof of \cite[Proposition 7.7]{BGT}, we may show that
\bal
v_n(x)\rightarrow \int_{\partial\Omega\cap\supp \xz } \zeta(y) \,\dd \omega^{x}_{\xO}(y)=:v(x) \quad \text{as } n \to \infty.
\eal

On the other hand, the Poisson kernel $P_\xm^n$ of $-L_\mu$ in $\Omega \setminus \Sigma_n$ is well defined and is given  by
\bal P_\mu^n(x,y)=-\frac{\partial G_\mu^n}{\partial {\bf n}^n}(x,y), \quad x \in \Omega\setminus \xS_n, \, y \in \partial (\xO\setminus\xS_n),
\eal
where ${\bf n}^n$ is the unit outer normal vector to $\partial(\xO\setminus\xS_n)$. Hence,
\ba \label{wnrep2} v_n(x) = \int_{\partial(\xO\setminus\xS_n)}P_\mu^n(x,y)Z(y)\tilde W(y)\,\dd S_{\partial \Omega}(y)=\int_{\partial\xO\cap\supp \xz }P_\mu^n(x,y)\xz(y)\tilde W(y)\,\dd S_{\partial \Omega}(y),
\ea
where $S_{\partial \Omega}$ is the $(N-1)$-dimensional surface measure on $\partial(\xO\setminus\xS_n)$. Combining all above, we obtain
\be \label{zP}
\int_{\partial \Omega\cap\supp \xz}\zeta(y)\, \dd \omega_n^{x}(y)=\int_{\partial \Omega\cap\supp \xz}P_\mu^n(x,y)\tilde W(y)\zeta(y)\, \dd S_{\partial \Omega}(y).
\ee

Put $\beta=\frac{1}{2}\min\{d_{\partial \Omega}(x), r_0\}$. Since $G^n_{\mu}(x,y)\nearrow G_{\mu}(x,y)$ for any $x\neq y$ and $x,y\in \xO,$  $\{G_{\mu}^n(x,\cdot)\}_n$ is uniformly bounded in $W^{2,\kappa}(\Omega_\beta\setminus\xS_{r_0})$ for any $\kappa>1$. Thus, by the standard  compact Sobolev embedding, there exists a subsequence, still denoted by index $n$, which converges to $G_{\mu}(x,\cdot)$ in $C^1(\overline{\Omega_\beta\setminus\xS_{r_0}})$ as $n \to \infty$. This implies that $P_\mu^n(x,\cdot) \to P_\mu(x,\cdot)$ uniformly on $\partial \Omega\setminus\xS_{r_0}$ as $n \to \infty$.

Therefore, by letting $n \to \infty$ in \eqref{zP}, we obtain
\be \label{zP1} \begin{aligned} \int_{\partial\Omega} \zeta(y) \,\dd \omega^{x}_{\xO}(y)&=\lim_{n\to\infty}\int_{\partial\Omega} \zeta(y) \tilde W(y) \,\dd \omega^{x}_n(y) \\
&=\lim_{n\to\infty}\int_{\partial\Omega} P_{\mu}^n(x,y)\zeta(y)\tilde W(y) \, \dd S_{\partial \Omega}(y)=\int_{\partial\Omega} P_{\mu}(x,y)\zeta(y)\tilde W(y) \,\dd S_{\partial \Omega}(y).
\end{aligned} \ee

By \eqref{zP1} and the fact that
\bal
\inf_{y \in \partial \Omega\setminus\xS_{r}}P_\mu(x_0,y)>0\quad\forall\;r>0,
\eal
we may show that
\be \label{PE} \omega^{x}_{\xO}(E)=\BBP_\mu[\1_E\tilde W](x)
\ee
for any Borel set  $E\subset\overline{E}\subset \partial\xO\setminus\xS$. This implies in particular that $\omega_{\Omega}^{x_0}$ and $S_{\partial \Omega}$ are mutually absolutely continuous with respect to compact subsets of $\partial \Omega \setminus \Sigma$.

Now, assume $0 \leq h \in L^1(\partial \Omega; \dd\omega_{\Omega}^{x_0})$ has compact support in $\partial \Omega \setminus \Sigma$ and $\dist(\supp h, \Sigma)=4r>0$. Then there exists a sequence of nonnegative functions $\{ h_n \} \subset C(\partial \Omega)$ with compact support in $\partial \Omega \setminus \Sigma$ such that $\dist(\supp h_n,\Sigma)=2r>0$ for any $n \in \N$ and $h_n \to h$ in $L^1(\partial \Omega; \dd\omega_{\Omega}^{x_0})$ as $n \to \infty$.

Applying \eqref{zP1} with $\zeta$ replaced by $|h_n-h_m|$ for $m, n \in \N$ and using the fact that $\inf_{y \in \Sigma_r \cap \partial \Omega}(P_\mu(x_0,y) \tilde W(y))>0$, we have
\ba \label{hmhn} \begin{aligned} \int_{\partial \Omega} |h_m(y)-h_n(y)|\, \dd\omega_{\Omega}^{x_0}(y) &=\int_{\partial \Omega} P_{\mu}(x_0,y)|h_n(y)-h_n(y)| \tilde W(y) \, \dd S_{\partial \Omega}(y) \\
&\geq c\int_{\partial \Omega} |h_n(y)-h_n(y)| \, \dd S_{\partial \Omega}(y).
\end{aligned} \ea	
This implies that $\{h_n\}$ is a Cauchy sequence in $L^1(\partial \Omega)$. Therefore, there exists $\tilde h \in L^1(\partial \Omega)$ such that $h_n \to \tilde h$ in $L^1(\partial \Omega)$. Since  $\omega_{\Omega}^{x_0}$ and $S_{\partial \Omega}$ are mutually absolutely continuous with respect to compact subsets of $\partial \Omega \setminus \Sigma$ and $h$ and $\tilde h$ have compact support in $\partial \Omega \setminus \Sigma$, we deduce that $h=\tilde h$ $\omega_{\Omega}^{x_0}$-a.e. and $S_{\partial \Omega}$-a.e. in $\partial \Omega$. In particular, $h \in L^1(\partial \Omega)$.
	
Applying \eqref{zP1} with $\zeta$ replaced by $h_n$, for any $n \in \N$, we have
\ba \label{fP1} \int_{\partial \Omega} h_n(y) \, \dd\omega_{\Omega}^{x_0}(y)=\int_{\partial \Omega} P_{\mu}(x_0,y)h_n(y) \tilde W(y) \,\dd S_{\partial \Omega}(y).
\ea
By letting $n \to \infty$ in \eqref{fP1}, we obtain \eqref{Pformula}.

Next, we assume $h \in L^1(\partial \Omega; \dd\omega_{\Omega}^{x_0})$ and drop the assumption that $h \geq 0$, then we write $h=h_+ - h_-$ where $h_\pm \in L^1(\partial \Omega; \dd\omega_{\Omega}^{x_0})$. By applying \eqref{Pformula} for $h_\pm$, we deduce that \eqref{Pformula} holds true for $h \in L^1(\partial \Omega; \dd\omega_{\Omega}^{x_0})$. Moreover, we can show that $h_\pm \in L^1(\partial \Omega)$, which implies $h \in L^1(\partial \Omega)$.
\end{proof}

\begin{proposition}\label{weaksolution0}
$(i)$ For any $h \in L^1(\partial\xO ;\dd \omega^{x_0}_{\xO})$ with compact support in $\partial \Omega\setminus\xS$, there holds
\be \label{weakfor1}
-\int_\xO \BBK_\mu[h\dd \omega^{x_0}_{\xO}] L_\xm\eta \, \dd x=-\int_{\partial\xO}\frac{\partial \eta}{\partial{\bf n}}(y)h(y)\tilde W(y)\, \dd S_{\partial \Omega}(y),\quad \forall \eta \in \mathbf{X}_\xm(\xO).
\ee

$(ii)$ For any $\xn\in \mathfrak{M}(\prt\xO )$ with compact support in $\partial \Omega\setminus\xS$, there holds
\be \label{weakfor2}
-\int_\xO \BBK_\mu[\nu] L_\xm\eta \, \dd x=-\int_{\partial\xO}\frac{\partial \eta}{\partial{\bf n}}(y)\frac{1}{P_\mu(x_0,y)}\, \dd \nu(y),\quad \forall \eta \in \mathbf{X}_\xm(\xO),
\ee
where $P_\mu(x_0,y)$ is defined in \eqref{Poisson} and ${\bf X}_\mu(\Gw)$ is defined by \eqref{Xmu}.

\end{proposition}
\begin{proof}[\textbf{Proof}] (i) Let $\xz\in C(\partial\xO)$ with compact support in $\partial\xO\setminus\xS$ such that $\dist(\supp \xz,\xS)=r>0$. We consider a function $Z\in C(\overline{\xO})$ such that $Z(y)=\xz(y)$ for any $y\in C(\partial\xO)$ and $Z(y)=0$ in $\xS_{\frac{r}{2}}.$  Set $r_0=\frac{1}{4}\min\{\xb_1,r\}$ where $\xb_1$ is the constant in \eqref{cover2}. Let $\{\Sigma_n\}$ be as in \eqref{Kn}, $\eta\in \mathbf{X}_\xm(\xO)$  and $v_n$ be the solution of \eqref{harmonicrel1}.

In view of the proof of Proposition \ref{measureboundary}, $v_n\in C(\overline{\xO\setminus\xS_n})$ and
\bal
v_n(x)=\int_{\partial\xO} \xz(y) \tilde W(y)\, \dd\xo^{x}_n(y) =\int_{\partial\Omega} P_{\mu}^n(x,y)\tilde W(y)\zeta(y) \, \dd S_{\partial \Omega}(y).
\eal
Put
\bal
v(x)=\int_{\partial\xO} \zeta(y) \, \dd\xo_{\Omega}^{x}(y) \quad \text{and} \quad w(x)=\int_{\partial\xO} |\xz(y)| \, \dd\xo_{\Omega}^{x}(y).
\eal
Then $v_n(x)\rightarrow v(x)$ for any $x\in\xO$ and $|v_n(x)|\leq w(x)$ in $\xO\setminus\xS_n$. By \cite[Proposition 1.3.7]{MVbook},
\be \label{vntildeZ}
-\int_{\xO\setminus\xS_n}v_n L_\xm \tilde Z \, \dd x=-\int_{\partial\xO}\tilde W\xz\frac{\partial \tilde Z}{\partial{\bf n}} \, \dd S_{\partial \Omega},\quad \forall \tilde Z \in C^2_0(\xO \setminus \Sigma_n).
\ee
By approximation, the above equality is valid for any $\tilde Z \in C^{1,\xg}(\overline{\xO\setminus\xS_n}),$ for some  $\xg\in(0,1)$ and $\xD \tilde Z\in L^\infty(\Omega \setminus \Sigma_n)$. Hence, we may choose $\tilde Z=\eta_n$ in \eqref{vntildeZ}, where $\eta_n$ satisfies
\bal \left\{  \BAL
-L_{\mu }\eta_n&=-L_\xm\eta \qquad&&\text{in } \xO\setminus\xS_n\\
\eta_n &=0\qquad&&\text{on } \prt( \xO\setminus\xS_n)
\EAL \right.
\eal
to obtain
\be \label{vneta}
-\int_{\xO\setminus\xS_n} v_n L_\xm \eta \, \dd x=-\int_{\partial\xO}\tilde W\xz\frac{\partial \eta_n}{\partial{\bf n}} \, \dd S_{\partial \Omega}.
\ee
We note that $\eta_n\rightarrow \eta$ locally uniformly in $\xO$ and in $C^1(\overline{\xO\setminus\xS_1})$. Therefore by the dominated convergence theorem, letting $n \to \infty$ in \eqref{vneta}, we obtain
\be
-\int_{\xO} v L_\xm \eta \, \dd x=-\int_{\partial\xO}\tilde W \xz \frac{\partial \eta}{\partial{\bf n}}\, \dd S_{\partial \Omega}.\label{weak1}
\ee

Now, let $h \in L^1(\partial \Omega ;\dd \omega^{x}_{\xO})$ with compact support in $\partial \Omega\setminus\xS$ such that $\dist(\supp h,\xS)=4r>0.$  By \eqref{Pformula} we may construct a sequence $\{h_n\}\subset C(\partial\xO)$ such that $h_n$ has compact support in $\partial \Omega\setminus\xS$ with $\dist(\supp h_n,\xS)>r$ for any $n\in\BBN.$ In addition, the same sequence can be constructed such that $h_n\rightarrow h$ in $L^1(\partial\xO ;\dd \omega^{x_0}_{\xO})$ and in  $L^1(\partial \Omega).$

Set
\bal u_n(x)=\int_{\partial \Omega} K_\mu(x,y)h_n(y)\,\dd \omega^{x_0}_{\xO}(y)= \BBK_\mu[h_n\, \dd \omega_{\Omega}^{x_0}](x), \quad x \in \Omega.
\eal
Since $K_\mu(\cdot,y)\in C^2(\xO),$ by the above equality, we deduce that $u_n \to u$ locally uniformly in $\Omega,$ where
\bal u(x)=\int_{\partial\xO}K_{\mu}(x,y)h(y) \,\dd\omega_{\Omega }^{x_0}(y)= \BBK_\mu[h\, \dd \omega_{\Omega}^{x_0}](x), \quad x \in \Omega.
\eal
By \eqref{weak1} with $v=u_n$ and $\zeta=h_n$, there holds
\be
-\int_{\xO} u_n L_\xm \eta \, \dd x=-\int_{\partial\xO}\tilde W  h_n\frac{\partial \eta}{\partial{\bf n}}\, \dd S_{\partial \Omega}.\label{weak2}
\ee

Now, by \cite[Theorem 9.2]{BGT}, there exists a positive constant $C=C(N,\Omega,\Sigma,\mu, \kappa)$ such that
$\| u_n \|_{L^\kappa(\Omega;\ei)}  \leq C\int_\xO |h_n|\dd\omega_{\Omega }^{x_0}(y)$  for all $n \in \N$ and for any $1<\xk<\min\left\{\frac{N+1}{N-1},\frac{N+\am+1}{N+\am-1}\right\}$.
This in turn implies that $\{u_n\}$ is equi-integrable in $L^1(\Omega;\ei)$. Therefore, by Vitali's convergence theorem,
$u_n \to u$ in $L^1(\Omega;\ei)$. Letting $n\to\infty$ in \eqref{weak2}, we obtain \eqref{weakfor1}.
 \medskip

(ii) Assume $\dist(\supp \xn,\xS)=4r>0$ and let $\{h_n\}$ be a sequence of functions in $C(\partial\xO )$ with compact support in $\partial \Omega\setminus\xS$ such that $\dist(\supp h_n,\xS)\geq r$ and $h_n\rightharpoonup\xn$, i.e.
\be  \label{weakcon}  \int_{\partial\xO}\xz h_n \, \dd S_{\partial \Omega}\rightarrow \int_{\partial\xO}\xz \dd \nu \quad \forall \zeta \in C(\partial\xO).
\ee
In addition we assume that $\|h_n\|_{L^1(\partial\xO)} \leq C\|\xn\|_{\mathfrak{M}(\prt\xO)}$ for every $n \geq 1$, for some positive constant independent of $n.$

Set
\bal u_n(x)=\int_{\partial\xO}K_{\mu}(x,y)\frac{h_n(y)}{\tilde W(y) P_\mu(x_0,y)} \,\dd\omega_{\Omega }^{x_0}(y).
\eal
By \eqref{Pformula} and \eqref{weakcon}, we have
\bal u_n(x)=\int_{\partial \Omega}K_{\mu}(x,y)h_n(y)\,\dd S_{\partial \Omega}(y)
\rightarrow\int_{\partial\xO}K_{\mu}(x,y)\,\dd \nu(y)=:u(x).
\eal
It means $u_n \to u$  a.e. in $\Omega $.

Finally, equality \eqref{weakfor2} can be obtained by proceeding as in the proof of (i) and hence we omit its proof.
\end{proof}

\section{ Semilinear equations with a power absorption nonlinearity} \label{sec:BVP-abs}

\subsection{Keller-Osserman estimates}
\label{sec:KOestimate}
In this section, we prove Keller-Osserman type estimates for nonnegative solutions of equation
\be \tag{$E_+$} \label{eq:power1}
-L_\mu u+|u|^{p-1}u=0\quad  \text{in}\;\; \Omega.
\ee

For $\varepsilon>0$, set
\bal \xO_\xe:=\{x\in\xO:\;d_{\partial \Omega}(x)<\xe\}.
\eal
\begin{lemma} \label{lem:KO-1} Assume $p>1$. Let $u\in C^2(\xO)$ be a nonnegative solution of equation \eqref{eq:power1}. Then there exists a positive constant $C=C(\xO,\Sigma,\xm,p)$ such that
	\be
	0 \leq u(x) \leq Cd_{\partial \Omega}(x)^{-\frac{2}{p-1}},\quad\forall x\in\xO.\label{ko}
	\ee
\end{lemma}
\begin{proof}[\textbf{Proof}]
Let $\beta_0$ be as in Subsection \ref{assumptionK} and  $\eta_{\beta_0} \in C_c^\infty(\R^N)$ such that
\bal 0\leq\eta_{\xb_0}\leq1, \quad \eta_{\xb_0}=1 \text{ in } \overline{\xO}_{\frac{\xb_0}{4}} \quad \text{and} \quad  \supp\eta_{\beta_0} \subset \xO_{\frac{\xb_0}{2}}.
\eal
For some $\varepsilon \in (0,\frac{\xb_0}{16})$, we define
\bal V_\xe:=1-\eta_{\xb_0}+\eta_{\xb_0}(d_{\partial \Omega}-\xe)^{-\frac{2}{p-1}} \quad \text{in }  \Omega \setminus \overline{\xO}_\varepsilon.
\eal
Then $V_\varepsilon \geq 0$ in $ \Omega \setminus \overline{\xO_\epsilon}$. It  can be checked that there exists  $C=C(\xO,\xb_0,\xm,p)>1$ such that the function $W_\xe:=CV_\xe$ satisfies
\be \label{Wep} -L_\xm W_\xe+W_\xe^p = C(-L_\mu V_\varepsilon + C^{p-1}V_\varepsilon^p) \geq 0 \quad \text{in } \xO\setminus \overline{\xO}_\xe.
\ee
Combining \eqref{eq:power1} and \eqref{Wep} yields
\be \label{u-Wep} -L_\mu (u-W_\varepsilon) + u^{p} - W_\varepsilon^p \leq 0 \quad \text{in } \Omega \setminus \overline{\xO}_\varepsilon.
\ee
We see that $(u-W_\xe)^+\in H_0^1(\xO\setminus\overline{\xO}_\xe)$ and  $(u-W_\xe)^+$ has compact support in $\xO\setminus\overline{\xO}_\xe$. By using $(u-W_\varepsilon)^+$ as a test function for \eqref{u-Wep}, we deduce that
\bal
	0 &\geq \int_{\xO\setminus \xO_\xe}|\nabla (u-W_\xe)^+|^2 \dd x - \xm\int_{\xO\setminus \xO_\xe}\frac{[(u-W_\xe)^+]^2}{d_\Sigma^2} \dd x+\int_{\xO\setminus \xO_\xe}(u^{p}-W_\xe^p)(u-W_\varepsilon)^+ \dd x\\
	&\geq \int_{\xO\setminus \xO_\xe}|\nabla (u-W_\xe)^+|^2 \dd x - \xm\int_{\xO\setminus \xO_\xe}\frac{[(u-W_\xe)^+]^2}{d_\Sigma^2} \dd x \geq \xl_{\xm,\Sigma}\int_{\xO\setminus \xO_\xe}| (u-W_\xe)^+|^2 \dd x.
\eal
	This and the assumption $\lambda_{\mu,\Sigma} >0$ imply $(u-W_\xe)^+ = 0$, whence $u\leq W_\xe$ in $\xO\setminus\overline{\xO}_\xe$.  Letting $\xe\to0,$ we obtain the desired result.	
\end{proof}

The following theorem is the main tool in the study of the boundary removable singularities for nonnegative solutions of equation \eqref{eq:power1}. It is required additionally that $\Omega$ is a $C^3$ domain which is needed to apply Proposition \ref{barr}.
\begin{theorem}\label{prop19}
	Let $p>1,$ $F\subset \Sigma$ be a compact subset of $\Sigma$ and $d_F(x)=\dist(x,F)$. We additionally assume that $\xO$ is a $C^3$ bounded domain. If $u\in C^2(\xO)$ is a nonnegative solution of \eqref{eq:power1} satisfying
	\ba\label{as1}
	\lim_{x\in\xO,\;x\rightarrow\xi}\frac{u(x)}{\tilde W(x)}=0\qquad\forall \xi\in \partial\xO\setminus F,\quad  \text{locally uniformly in } \partial\xO\setminus F,
	\ea
	then there exists a positive constant $C=C(N,\Omega,\Sigma,\mu,p)$ such that
	\ba \label{3.4.24}
	u(x)&\leq Cd_{\partial \Omega}(x)d_\Sigma(x)^{-\am}d_F(x)^{-\frac{2}{p-1}+\am-1}\qquad\forall x\in \xO, \\ \label{3.4.24*}
	|\nabla u(x)|&\leq Cd_\Sigma(x)^{-\am} d_F(x)^{-\frac{2}{p-1}+\am-1}\qquad\forall x\in \xO.
	\ea
\end{theorem}  
\begin{proof}[\textbf{Proof}]
	The proof is in the spirit of \cite[Proposition 3.4.3]{MVbook}. Let $\xb_2$ be the positive constant defined in Proposition \ref{barr}. Let $\xi\in (\Sigma_{\xb_2}\cap\partial\xO)\setminus F$
	and put $d_{F,\xi} = \frac{1}{16} d_F(\xi)<1$. Denote
	\bal \xO^\xi:= \frac{1}{d_{F,\xi}}\Omega = \{y\in\mathbb{R}^N:\;d_{F,\xi}\,y\in\xO\}\quad\text{and}\quad \Sigma^\xi:=\frac{1}{d_{F,\xi}}\Sigma =\{y\in\mathbb{R}^N:\;d_{F,\xi}\,y\in \Sigma\}.
	\eal
	If $u$ is a nonnegative solution of \eqref{eq:power1} in $\xO$ then the function
	\bal
	u^\xi (y): = d_{F,\xi}^{\frac{2}{p-1}}
	u(d_{F,\xi}y),\quad y\in\xO^\xi,
	\eal
	is a nonnegative solution of
	\be \label{eq:uxi} -\xD u^\xi- \frac{\mu}{|\mathrm{dist}(y,\Sigma^\xi)|^2} u^\xi +\left(u^\xi\right)^p=0
	\ee
	in $\Omega^\xi.$

Now we note that $u^\xi$ is a nonnegative $L_\xm$ subharmonic function and  satisfies (by \eqref{as1})
\bal \lim_{y\in\xO^\xi,\;y\rightarrow P}\frac{u^\xi(y)}{\tilde W^\xi(y)}=0\qquad\forall P\in B\left(\frac{1}{d_{F,\xi}}\xi,2\right) \cap \partial\xO^\xi,
\eal
where
\bal \tilde W^\xi(y) = 1-\eta_{\frac{\beta_0}{d_{F, \xi}}}+\eta_{\frac{\beta_0}{d_{F, \xi}}}W^\xi(y) \quad \text{in } \Omega^\xi \setminus \Sigma^\xi,
\eal
and
\bal
W^\xi(y)=\left\{ \BAL &(d_{\partial \xO^\xi}(y)+d_{\Sigma^\xi}(y))^{2}d_{\Sigma^\xi}(y)^{-\ap}\qquad&&\text{if}\;\mu <H^2, \\
&(d_{\partial \xO^\xi}(y)+d_{\Sigma^\xi}(y))^{2}d_{\Sigma^\xi}(y)^{-H}|\ln d_{\Sigma^\xi}(y)|\qquad&&\text{if}\;\mu =H^2,
\EAL \right. \quad x \in \Omega^\xi \setminus \Sigma^\xi.
\eal

Set $R_0=\min\{\xb_2,1\}$. In view of the proof of \cite[Lemma 6.2 and estimate (6.7)]{BGT}, there exists a positive constant $c$ depending on $\xO,\xS,\xm$ and
\ba\label{201}
\int_{B\left(\frac{1}{d_F(\xi)}\xi,2R_0\right)\cap\xO^\xi} u^\xi(y)d_{\partial \xO^\xi}(y)d_{\xS^\xi}(y)^{-\am} \dd y
\ea
such that
\be \label{sim22}
	u^\xi(y) \leq c\,d_{\partial \xO^\xi}(y)d_{\xS^\xi}(y)^{-\am}\quad \forall y\in B\left(\frac{1}{d_F(\xi)}\xi,R_0\right)\cap\xO^\xi .
	\ee

Let $r_0=\frac{R_0}{16}$ and let $w_{r_0,\xi}$ be the supersolution of \eqref{eq:uxi} in $\mathscr{B}(\frac{1}{d_{F,\xi}}\xi,r_0)\cap\xO^\xi$ constructed in Proposition \ref{barr} with $R=r_0$ and $z=\frac{1}{d_{F,\xi}}\xi$.

Taking into account of \eqref{sim22} and using a similar  argument as in the proof of Lemma \ref{lem:KO-1}, we can show that
	\bal
	u^\xi(y) \leq w_{r_0,\xi}(y)\quad \forall y\in \mathscr{B}\left(\frac{1}{d_{F,\xi}}\xi,r_0 \right)\cap\xO^\xi.
	\eal
By \eqref{201}, \eqref{sim22} and the above inequality, we may deduce that
\be \label{sim}
	u^\xi(y) \leq c\,d_{\partial \xO^\xi}(y)d_{\xS^\xi}(y)^{-\am}\quad \forall y\in B\left(\frac{1}{d_F(\xi)}\xi,\frac{r_0}{16}\right)\cap\xO^\xi,
	\ee
where $c$ depends only on $\xO,\xS,\mu,p,$ the $C^3$ characteristic of $\Omega^\xi$ and the $C^2$ characteristic of $\Sigma^\xi$. As $d_{F, \xi} \leq 1$ the $C^3$ characteristic of $\xO$ (respectively the $C^2$ characteristic of $\Sigma$) is also a $C^3$ characteristic of $\xO^\xi$ (respectively a $C^2$ characteristic of $\Sigma^\xi$), therefore this constant $c$ can be taken to be independent of $\xi$. Thus, for any $\xi \in (\Sigma_{\xb_2}\cap\partial\xO)\setminus F$ such that $d_{F}(x)\leq16$, there holds
	\be \label{sim2}
	u(x)\leq c\,d_{\partial \Omega}(x)d_\Sigma(x)^{-\am}d_{F}(\xi)^{-\frac{2}{p-1}+\am-1}\quad\forall x\in B\left(\xi,r_1d_F(\xi)\right)\cap \xO,
	\ee
where $r_1=\frac{r_0}{16^2}.$

Now, we consider three cases.

\medskip
	
\noindent	\textbf{Case 1:} $x\in \Sigma_{\frac{r_1}{32}}\cap\xO$ and $ d_F(x) < 1$.
	If $d_{\partial \Omega}(x)\leq \frac{r_1}{8+r_1}d_F(x)$ then there exits a unique point in  $\xi\in\partial\xO\setminus F$ such that $|x-\xi|=d_{\partial \Omega}(x).$ Hence,
	\be \label{sim3a}
	d_F(\xi)\leq d_{\partial \Omega}(x)+d_F(x) \leq 2\frac{4+r_1}{8+r_1}d_F(x)<16,
	\ee
$d_F(x)\leq \frac{8+r_1}{8}d_{F}(\xi)$ and $d_{\partial \Omega}(x)\leq \frac{r_1}{8}d_{F}(\xi).$ This, combined with \eqref{sim2}, \eqref{sim3a} and the fact that $d_F(x)\approx d_{F}(\xi)$, yields
	\bal
	u(x)\leq C d_\Sigma(x)^{-\am}d_{F}(\xi)^{-\frac{2}{p-1}+\am-1} \leq Cd_\Sigma(x)^{-\am}d_F(x)^{-\frac{2}{p-1}+\am-1}.
	\eal

	If   $d_{\partial \Omega}(x)> \frac{r_1}{8+r_1}d_F(x)\geq\frac{r_1}{8+r_1}d_\xS(x) $ then by \eqref{ko} and the fact that $d_{\partial \Omega}(x)\approx d_F(x)\approx d_\xS(x)$, we obtain
	\bal u(x) \leq Cd_{\partial \Omega}(x)^{-\frac{2}{p-1}} \leq Cd_{\partial \Omega}(x)d_\Sigma(x)^{-\am}d_F(x)^{-\frac{2}{p-1}+\am-1}.
	\eal
	Thus (\ref{3.4.24}) holds for every $x\in \Sigma_{\frac{r_1}{4}}$ such that $d_F(x) <1$.
	
\medskip
	
\noindent	\textbf{Case 2:} $x\in \Sigma_{\frac{r_1}{32}}\cap\xO$ and $d_F(x) \geq 1$.
	Let $\xi$ be the unique point in $\partial\xO\setminus F$ such that $|x-\xi|=d_{\partial \Omega}(x)$. Since $u$ is an $L_\mu$-subharmonic function in $B(\xi,r_1) \cap \Omega$, in view of the proof of \eqref{sim}, we obtain
	\bal
	u(x)\leq Cd_{\partial \Omega}(x) d_\Sigma(x)^{-\am} \leq  Cd_{\partial \Omega}(x)d_\Sigma(x)^{-\am} d_F(x)^{-\frac{2}{p-1}+\am-1}\qquad\forall x\in B\left(\xi,\frac{\xb_0}{2}\right)\cap \Omega ,
	\eal
where $C$ depend only on $\xO,\xS,\xm,p.$

\medskip	

\noindent \textbf{Case 3:} $x \in \Omega \setminus \Sigma_{\frac{r_1}{32}}$. The proof is similar to the one of \cite[Proposition A.3]{GkV} and we omit it.

(ii) Let $x_0\in\xO$. Put $\ell=d_{\partial \Omega}(x_0)$ and
	\bal \Omega^{\ell}:=\frac{1}{\ell}\Omega = \{y\in\mathbb{R}^N:\;\ell y\in\xO\}\quad\text{and}\quad\xS^{\ell}:=\frac{1}{\ell}\xS = \{y\in\mathbb{R}^N:\;\ell y\in\xS\} .
	\eal
	If $x\in B(x_0,\frac{\ell}{2})$ then $y=\ell^{-1}x$ belongs to $B(y_0,\frac{1}{2}),$ where $y_0=\ell^{-1}x_0$. Also we have that
	$\frac{1}{2}\leq d_{\Omega^\ell}(y)\leq \frac{3}{2}$ and $\frac{1}{2}\leq d_{\xS^\ell}(y)$ for each $y\in B(y_0,\frac{1}{2})$. Set $v(y)=u(\ell y)$ for $y\in B(y_0,\frac{1}{2})$ then $v$ satisfies
	\bal
	-\xD v- \frac{\xm}{d_{\xS^\ell}^2} v + \ell^2 \left|v\right|^p=0\qquad\mathrm{in}\;\;B(y_0,\frac{1}{2}).
	\eal
	By the standard elliptic estimates and  \eqref{ko} we have
	\bal
	\sup_{y\in B(y_0,\frac{1}{4})}|\nabla v(y)|\leq C\sup_{y\in B(y_0,\frac{1}{3})}|v(y)|\leq C v(y_0),
	\eal
This, together with the equality $\nabla v(y)=\ell \nabla u(x)$, estimate \eqref{3.4.24}
 implies	
	\bal
	|\nabla u(x_0)|&\leq C\ell^{-1} d_{\partial \Omega}(x_0) d_\Sigma^{-\am}(x_0)d_F(x_0)^{-\frac{2}{p-1}+\am-1}\\
	&\leq C d_\Sigma(x_0)^{-\am}d_F(x_0)^{-\frac{2}{p-1}+\am-1}.
	\eal
Therefore estimate \eqref{3.4.24*} follows since $x_0$ is an arbitrary point. The proof is complete.
\end{proof}

\subsection{Removable boundary singularities} \label{sec:removable}
This subsection is devoted to the study of removable boundary singularities for nonnegative solutions of equation \eqref{eq:power-a} in the supercritical range.
More precisely we will prove Theorems \ref{remov-1} and \ref{remove-2}.
\begin{proof}[\textbf{Proof of Theorem \ref{remov-1}}]
We will only consider the case $\xm<H^2$ and $p=\frac{\ap+1}{\ap-1}$, since the proof in the other cases is
very similar. Let $u$ be a nonnegative solution of \eqref{eq:power-a} satisfying \eqref{as1intro}. By \eqref{3.4.24} with $F=\xS$, there holds
\ba\label{fragma1}
u(x)&\leq Cd_{\partial \Omega}(x)d_\Sigma(x)^{-\ap}\qquad\forall x\in \xO,
\ea
for some constant $C$ independent of $u.$

Let $\{\xO_n\}$ be a $C^2$ exhaustion of $\xO$ and we write
\bal
\BBG_\xm^{\xO_n}[u^\frac{\ap+1}{\ap-1}](x)=\int_{\xO_n}G^{\xO_n}_\xm(x,y)u(y)^\frac{\ap+1}{\ap-1}\;\dd y.
\eal

By the representation formula in $\xO_n,$ we have that
\ba\label{102}
u(x_0)+\BBG_\xm^{\xO_n}[u^\frac{\ap+1}{\ap-1}](x_0)=\int_{\partial\xO_n}u(y)\, \dd \xo^{x_0}_{\xO_n}(y).
\ea
By \eqref{fragma1}, Proposition \ref{22222} and the definition of $\tilde W,$ we have
\ba\label{103}
\limsup_{n\to\infty}\int_{\partial\xO_n}u(y)\, \dd \xo^{x_0}_{\xO_n}(y)\leq C\limsup_{n\to\infty}\int_{\partial\xO_n}\tilde W(y)\, \dd \xo^{x_0}_{\xO_n}(y)=C\xo^{x_0}_{\xO}(\partial\xO).
\ea
Since $G^{\xO_n}_\xm(x,y)\uparrow G_\mu(x,y)$ for $x,y\in\xO, x \neq y,$ \eqref{102} and \eqref{103} yield
\bal
\BBG_\xm[u^\frac{\ap+1}{\ap-1}](x_0)\leq C\xo^{x_0}_{\xO}(\partial\xO).
\eal
Hence there exists a positive constant $C_0$ independent of $u$ such that
\ba\label{fragma2}
\int_\xO u(y)^\frac{\ap+1}{\ap-1}\ei(y)\;\dd y\leq C_0.
\ea

Consequently, the function $v=u+\BBG_\xm[u^\frac{\ap+1}{\ap-1}]$ is a nonnegative $L_\xm$-harmonic in $\xO.$ This, together with Theorem \ref{th:Rep}, implies the existence of a  measure $\nu \in \GTM^+(\partial \Omega )$ such that
\ba \label{uGKa}
u+\BBG_\xm[u^\frac{\ap+1}{\ap-1}]=\BBK_\xm[\xn] \quad\text{in}\;\xO.
\ea
By Proposition \ref{22222} and \eqref{fragma1}, we may deduce that $\xn$ has compact support in $\xS.$

Next we will show that $\xn\equiv 0$. Suppose by contradiction that $\xn\not\equiv0$. Let $1<M\in\mathbb{N}$ and $v_{M,n}$ be the positive solution of
	\be\label{sub} \left\{ \BAL
	-L_{\xm }^{\xO_n}v_{M,n} + v_{M,n}^\frac{\ap+1}{\ap-1}&=0\qquad&&\text{in } \xO_n\\
	v_{M,n}&=M u\qquad&&\text{on } \prt \xO_n.
	\EAL \right.
	\ee

Since $M u$ is a supersolution of the above problem, we have that $0\leq v_{M,n}\leq Mu$ in $\xO_n$ for any $n\in\BBN.$ As a result, there exists $v_M\in C^2(\xO)$ such that $v_{M,n}\to v_M$ locally uniformly in $\xO$ and $L_\xm v_M+ v_{M}^\frac{\ap+1}{\ap-1}=0$ in $\xO.$ Since $v_M\leq Mu$, by\eqref{3.4.24} with $F=\xS$, there holds
\ba\label{fragma3}
v_M(x)&\leq Cd_{\partial \Omega}(x)d_\Sigma(x)^{-\ap}\qquad\forall x\in \xO,
\ea
for some constant $C$ independent of $v_M,$ which in turn implies that
\ba\label{fragma4}
\int_\xO v_M(y)^\frac{\ap+1}{\ap-1}\ei(y)\;\dd y\leq C_0,
\ea
for some constant $C_0$ independent of $v_M.$ Also, by the representation formula we have
\ba\label{104}
v_{M,n}(x)+\BBG_\xm^{\xO_n}[v_{M,n}^\frac{\ap+1}{\ap-1}](x)=M\int_{\partial\xO_n}u(y)\, \dd \xo^{x}_{\xO_n}(y), \quad \forall x \in \Omega.
\ea
From \eqref{uGKa}, we have
\bal
\int_{\partial\xO_n}u(y)\, \dd \xo^{x}_{\xO_n}(y)&=\int_{\partial\xO_n}(\BBK_\xm[\xn](y)-\BBG_\xm[u^\frac{\ap+1}{\ap-1}](y))\, \dd \xo^{x}_{\xO_n}(y)\\
&=\BBK_\xm[\xn](x)-\int_{\partial\xO_n}\BBG_\xm[u^\frac{\ap+1}{\ap-1}](y)\, \dd \xo^{x}_{\xO_n}(y).
\eal
By Proposition \ref{traceKG} (ii), we find
\bal
\lim_{n \to \infty}\int_{\partial\xO_n}\BBG_\xm[u^\frac{\ap+1}{\ap-1}](y)\, \dd \xo^{x}_{\xO_n}(y) = 0.
\eal
Hence, letting $n\to \infty$ in \eqref{104}, we obtain
\bal
v_M+\BBG_\xm[v_M^\frac{\ap+1}{\ap-1}]=M\BBK_\xm[\xn]\quad\text{in}\;\xO, \quad\forall M>0.
\eal
Hence, we can easily see that the above display contradicts with \eqref{fragma3} and \eqref{fragma4}.
\end{proof}

Next we turn to

\begin{proof}[\textbf{Proof of Theorem \ref{remove-2}}]
Without loss of generality, we may assume that $z=0.$ Let $\xz:\mathbb{R}\to[0,\infty)$ be a smooth function such that $0\leq \xz\leq 1,$ $\xz(t)=0$ for $|t|\leq 1$ and $\xz(t)=1$ for $|t|>2$. For $\varepsilon>0$, we set $\xz_\xe(x)=\xz(\frac{|x|}{\xe}).$ Since $u\in C^2(\xO),$ there holds
	\bal L_\mu (\xz_\xe u)=u\xD\xz_{\xe}+\xz_\xe u^p + 2\nabla\xz_\xe \cdot \nabla u \quad\text{in }\;\xO .
	\eal
By \eqref{3.4.24}, \eqref{3.4.24*}, \eqref{eigenfunctionestimates}, the estimate $\int_{\Sigma_{\beta}}d_{\Sigma}(x)^{-\alpha}\dd x \lesssim \beta^{N-\alpha}$ for $\alpha<N-k$, and the assumption $p\geq\frac{N-\am+1}{N-\am-1},$ we have
\be \label{Lzetau1} \begin{aligned}
&\int_\Omega \xz_\xe u^p \ei \, \dd x \lesssim \xe^{-\frac{2p}{p-1}+(\am- 1)p}\int_{\xO\cap\{|x|>\xe\}}d_\xS(x)^{(p+1)(1-\am)} \, \dd x \lesssim \varepsilon^{-\frac{2p}{p-1}+(\am -1)p}, \\
&\int_\Omega u|\Delta\xz_{\xe}| \ei \, \dd x \leq \xe^{-\frac{2}{p-1}+\am-3}\int_{\xO\cap\{\xe<|x|<2\xe\}} d_\xS(x)^{2-2\am} \, \dd x \lesssim \varepsilon^{N-\frac{2}{p-1}-\am-1} \lesssim 1, \\
&\int_\Omega |\nabla\xz_\xe||\nabla u|\ei \dd x \lesssim \xe^{-\frac{2}{p-1}+\am-2}\int_{\xO\cap\{\xe<|x|<2\xe\}} d_\xS(x)^{-2\am+1} \, \dd x \lesssim \varepsilon^{N-\frac{2}{p-1}-\am-1} \lesssim 1.
\end{aligned} \ee

Estimates \eqref{Lzetau1} imply that $L_\mu(\zeta_{\varepsilon}u) \in L^1(\Omega;\ei)$. By \cite[Lemma 8.5]{BGT}, we have
	\bal
-\int_{\xO} \xz_\xe u L_\xm\eta \,\dd x=-\int_{\xO} \left(u\xD\xz_{\xe}+\xz_\xe u^p + 2\nabla\xz_\xe \cdot \nabla u\right)\eta \,\dd x,\quad\forall \eta \in \mathbf{X}_\xm(\xO).
\eal
	Taking $\eta=\ei,$ we obtain
\bal
\xl_{\xm,\Sigma}\int_{\xO} \xz_\xe u \ei \, \dd x+\int_{\xO} \xz_\xe u^p\ei \, \dd x=-\int_{\xO} \left(u\xD\xz_\xe+2\nabla\xz_\xe \cdot \nabla u\right)\ei \, \dd x.
\eal
By the last two lines in \eqref{Lzetau1}, we have
	\bal
	\xl_{\xm,\Sigma}\int_{\xO} \xz_\xe u\ei \,\dd x+\int_{\xO} \xz_\xe u^p\ei \,\dd x\leq  C\varepsilon^{N-\frac{2}{p-1}-\am-1}.\label{lpfragma}
	\eal
By letting $\xe\to 0$ and Fatou's lemma, we deduce that
	\bal
	\xl_{\xm,\Sigma}\int_{\xO}  u\ei \,\dd x+\int_{\xO}  u^p\ei \,\dd x\lesssim\left\{
\BAL
&0\quad\text{if}\;\;p>\frac{N+\am+1}{N-\am-1},\\
&1\quad\text{if}\;\;p=\frac{N+\am+1}{N-\am-1}.
\EAL\right.
	\eal
This implies that $u\equiv0$ if $p>\frac{N-\am+1}{N-\am-1}$ or  $u\in L^p(\xO;\ei)$ if $p=\frac{N-\am+1}{N-\am-1}.$

The rest of the proof can proceed as in the proof of Theorem \ref{remov-1} and we omit it.
\end{proof}

\subsection{Existence of solutions in the supercritical range} \label{sec:goodmeasure}

In this subsection we discuss the existence of solutions for the following problem
\be\label{mainproblempower} \left\{ \BAL -L_\mu u + \abs{ u}^{p-1}u  &= 0 \quad \text{in } \Gw ,\\
\tr(u) &= \xn,
\EAL \right. \ee
where $p>1$ and $\nu \in \GTM(\partial \Omega )$. We will focus on the supercritical case $p\geq\min\{\frac{N+1}{N-1},\frac{N-\am+1}{N-\am-1}\}$. In particular, we will give various sufficient conditions for the existence of solutions to \eqref{mainproblempower}.

In this direction, we need to recall some notations concerning Besov space (see, e.g., \cite{Ad, Stein}). For $\gs>0$, $1\leq \kappa<\infty$, we denote by $W^{\gs,\kappa}(\BBR^d)$ the Sobolev space over $\BBR^d$. If $\gs$ is not an integer the Besov space $B^{\gs,\kappa}(\BBR^d)$ coincides with $W^{\gs,\kappa}(\BBR^d)$. When $\gs$ is an integer we denote $\Gd_{x,y}f:=f(x+y)+f(x-y)-2f(x)$ and
\bal B^{1,\kappa}(\BBR^d):=\left\{f\in L^\kappa(\BBR^d): \myfrac{\Gd_{x,y}f}{|y|^{1+\frac{d}{\kappa}}}\in L^\kappa(\BBR^d\times \BBR^d)\right\},
\eal
with norm
\bal \|f\|_{B^{1,\kappa}}:=\left(\|f\|^\kappa_{L^\kappa}+\int_{\BBR^d} \int_{\BBR^d}\frac{|\Gd_{x,y}f|^\kappa}{|y|^{\kappa+d}}\,\dd x \, \dd y\right)^{\frac{1}{\kappa}}.
\eal
Then
\bal B^{m,\kappa}(\BBR^d):=\left\{f\in W^{m-1,\kappa}(\BBR^d): D_x^\ga f\in B^{1,\kappa}(\BBR^d)\;\forall\ga\in \BBN^d \text{ such that } |\ga|=m-1\right\},
\eal
with norm
\bal \|f\|_{B^{m,\kappa}}:=\left(\|f\|^\kappa_{W^{m-1,\kappa}}+\sum_{|\ga|=m-1}\int_{\BBR^d} \int_{\BBR^d}\frac{|D_x^\ga\Gd_{x,y}f|^\kappa}{|y|^{\kappa+d}}\,\dd x \, \dd y\right)^{\frac{1}{\kappa}}.
\eal
These spaces are fundamental because they are stable under the real interpolation method  developed by Lions and Petree.

It is well known that if $1<\kappa<\infty$ and $\ga>0$, $L_{\ga,\kappa}(\BBR^d)=W^{\ga,\kappa}(\BBR^d)$ if $\ga\in\BBN.$ If $\ga\notin\BBN$ then the positive cone of their dual coincide, i.e. $(L_{-\ga,\kappa'}(\BBR^d))^+=(B^{-\ga,\kappa'}(\BBR^d))^+$, always with equivalent norms.

\begin{lemma}\label{besov}
	Let $k\geq1$, $ \max\left\{1,\frac{N-k-\am+1}{N-\am-1}\right\}< p<\frac{\ap+1}{\ap-1}$ and $\xn\in \mathfrak{M}^+(\mathbb{R}^k)$ with compact support in $B^k(0,\frac{R}{2})$ for some $R>0$. Set
\ba \label{varth-b} \vartheta: = \frac{\ap+1-p(\ap-1)}{p}.
\ea
	For $x \in \mathbb{R}^{k+1}$, we write $x=(x_1,x') \in \mathbb{R}\times \mathbb{R}^{k}$.
	Then there exists a constant $C=C(R,N,k,\mu,p)>1$ such that
	\be\label{est-Bnu0} \BAL
	&C^{-1}\norm{\xn}^p_{B^{-\vartheta,p}(\mathbb{R}^k)}\\
	&\leq \int_{B^k(0,R)}\int_{0}^R x_{1}^{N-k-1-(p+1)(\am-1)}\left(\int_{B^k(0,R)}\left(x_1+|x'-y'|\right)^{-(N-2\am)}\dd \nu(y')\right)^p\,\dd x_1\,\dd x'\\
	&\leq C \norm{\xn}^p_{B^{-\vartheta,p}(\mathbb{R}^k)}. \EAL
	\ee
\end{lemma}
\begin{proof}[\textbf{Proof}]
The proof is very similar to that of \cite[Lemma 8.1]{GkiNg_absorption}, and hence we omit it.
\end{proof}

\begin{theorem}\label{potest}
	Let $k\geq1,$ $ \max\left\{1,\frac{N-k-\am+1}{N-\am-1}\right\}< p<\frac{\ap+1}{\ap-1}$ and $\xn\in \mathfrak{M}^+(\partial\xO)$ with compact support in $\Sigma$. Then there exists a constant $C=C(\xO,\Sigma,\mu)>1$ such that
	\be \label{2sideKnu} \BAL
	&C^{-1}\norm{\xn}_{B^{-\vartheta,p}(\Sigma)}\leq \norm{\BBK_{\mu}[\xn]}_{L^p(\xO;\ei)}\leq C \norm{\xn}_{B^{-\vartheta,p}(\Sigma)},\EAL
	\ee
	where $\vartheta$ is given in \eqref{varth-b}.
\end{theorem}
\begin{proof}[\textbf{Proof}]
	By \eqref{cover2}, there exists $\xi^j \in \Sigma$, $j=1,2,...,m_0$ (where $m_0 \in\N$ depends on $N,\Sigma$), and $\xb_1 \in (0,\frac{\xb_0}{4})$ such that
	$
\xO\cap\Sigma_{\beta_1} \subset \cup_{j=1}^{m_0} \mathcal{V}_{\xS}(\xi^j,\frac{\beta_0}{4})\cap \xO.
	$	
	
Assume $\xn\in \mathfrak{M}^+(\partial\xO)$ with compact support in $\xS\cap  \mathcal{V}_{\xS}(\xi^j,\frac{\beta_0}{4})$ for some $j\in\{1,...,m_0\}$.

On one hand, from \eqref{eigenfunctionestimates}, \eqref{Martinest1} and since $p< \frac{\ap+1}{\ap-1}$ and $\ap\geq\am$, we have
\bal
\int_{\xO\cap\mathcal{V}_{\xS}(\xi^j,\frac{\beta_0}{2})}\BBK_\xm[\xn]^p\ei \,\dx\gtrsim \xn(\partial\xO)^p\int_{\xO\cap\mathcal{V}_{\xS}(\xi^j,\frac{\beta_0}{2})}d_{\partial \Omega}(x)^{p+1}d_\Sigma(x)^{-(p+1)\am} \,\dx\gtrsim\xn(\partial\xO)^p.
\eal
On the other hand,
\bal
\int_{\xO\setminus \mathcal{V}_{\xS}(\xi^j,\frac{\beta_0}{2})}\BBK_\xm[\xn]^p\ei \dx\lesssim \xn(\partial\xO)^p\int_{\xO\setminus \mathcal{V}_{\xS}(\xi^j,\frac{\beta_0}{2})}d_{\partial \Omega}(x)^{p+1}d_\Sigma(x)^{-(p+1)\am}\dx\lesssim \xn(\partial\xO)^p.
\eal
Combining the above estimates, we obtain
	\be \label{K-split1}
\BAL	
\int_{ \Omega}\ei \BBK_{\mu}[\xn]^p \,\dd x
	&= \int_{\Omega \setminus  \mathcal{V}_{\xS}(\xi^j,\frac{\beta_0}{2})}\ei \BBK_{\mu}[\xn]^p \,\dd x+\int_{\xO\cap\mathcal{V}_{\xS}(\xi^j,\frac{\beta_0}{2}) }\ei \BBK_{\mu}[\xn]^p \,\dd x\\
&\approx\int_{\xO\cap\mathcal{V}_{\xS}(\xi^j,\frac{\beta_0}{2}) }\ei \BBK_{\mu}[\xn]^p \,\dd x.
\EAL	
\ee
	
For any $x \in \R^N$, we write {\small$x=(x',x''',x_N)$} where {\small$x'=(x_1,\ldots,x_k),$} {\small$x'''=(x_{k+1},\ldots,x_{N-1})$} and define the $C^2$ function
\bal
\xF(x):=(x',x_{k+1}-\xG_{k+1,\xS}^{\xi^j}(x'),...,x_{N-1}-\xG_{N-1,\xS}^{\xi^j}(x'),x_N-\xG_{N,\partial\xO}^{\xi^j}(x_1,\ldots,x_{N-1})).
\eal	
Taking into account the local representation of $\xS$ and $\partial\xO$ in Subsection \ref{assumptionK}, we may deduce that $\xF:\mathcal{V}_\xS(\xi^j,\beta_0)\to B^k(0,\xb_0)\times B^{N-1-k}(0,\xb_0)\times (-\xb_0,\xb_0) $ is $C^2$ diffeomorphism and
	$\xF(x)=(x',0_{\R^{N-k}})$ for $x=(x',x''',x_N) \in \Sigma$.
	In view of the proof of \cite[Lemma 5.2.2]{Ad}, there exists a measure $\overline{\xn}\in \GTM^+(\mathbb{R}^k)$ with compact support in $ B^k(0,\frac{\xb_0}{4})$ such that for any Borel $E\subset B^k(0,\frac{\xb_0}{4}) ,$ there holds
	$\overline{\xn}(E)=\xn(\xF^{-1}(E\times \{0_{\R^{N-k}}\}))$.
	
	Set $\psi=(\psi',\psi''',\psi_N)=\xF(x)$ then
	\bal
	\psi'=x', \, \psi'''=(x_{k+1}-\xG_{k+1,\xS}^{\xi^j}(x'),...,x_{N-1}-\xG_{N-1,\xS}^{\xi^j}(x')) \;\; \text{and} \;\; \psi_N=x_N-\xG_{N,\partial\xO}^{\xi^j}(x_1,\ldots,x_{N-1}).
	\eal
	By \eqref{eigenfunctionestimates}, \eqref{ddK} and \eqref{Martinest1}, we have
	\bal
	&\ei(x)\approx \psi_N(\psi_N+|\psi'''|)^{-\am},\\
	 &K_{\mu}(x,y)\approx \psi_N(\psi_N+|\psi'''|)^{-\am}(\psi_N+|\psi'''|+|\psi'-y'|)^{-(N-2\am)},\\
	 &\qquad\qquad\qquad\qquad\forall x\in  \mathcal{V}(\xi^j,\beta_0)\cap\xO,\;\forall y=(y',y''',y_N) \in  \mathcal{V}(\xi^j,\beta_0)\cap \Sigma.
	\eal
	Therefore
	\ba \label{21} \BAL
	&\int_{ \xO\cap \mathcal{V}(\xi^j,\beta_0/2) }\ei \BBK_{\mu}^p[\xn]\,\dd x\\
	&\approx \int_{B^k(0,\frac{\xb_0}{2})}\int_0^{\frac{\xb_0}{2}}\int_{B^{N-k-1}(0,\frac{\xb_0}{2})}\psi_N^{p+1}(\psi_N+|\psi'''|)^{-(p+1)\am}\\
	&\qquad\qquad \left(\int_{B^k(0,\frac{\xb_0}{4})}(\psi_N+|\psi'''|+|\psi'-y'|)^{-(N-2\am)}\dd \overline{\nu}(y')\right)^p \, \dd\psi'''\dd \psi_N \, \dd\psi'\\
	&\approx\int_{B^k(0,\frac{\xb_0}{2})}\int_{0}^{\frac{\xb_0}{2}}r^{N-k-1-(p+1)(\am-1)}
	\left(\int_{B^k(0,\frac{\xb_0}{2})}(r+|\psi'-y'|)^{-(N-2\am-2)}\dd \overline{\nu}(y')\right)^p\dd r \, \dd\psi'.
	\EAL
	\ea
	
	Since $\xn\mapsto \gn\circ\xF^{-1}$ is a $C^2$ diffeomorphism between
	$\mathfrak M^+(\Sigma\cap \mathcal{V}_\xS(\xi^j,\beta_0))\cap B^{-\vartheta,p}(\Sigma\cap \mathcal{V}_\xS(\xi^j,\beta_0))$ and
	$\mathfrak M^+(B^k(0,\xb_0))\cap B^{-\vartheta,p}(B^k(0,\xb_0))$, using  \eqref{K-split1},\eqref{21} and
	Lemma \ref{besov}, we derive that
	\ba\BAL
	C^{-1}\norm{\xn}_{B^{-\vartheta,p}(\Sigma)}\leq \norm{\BBK_{\mu}[\xn]}_{L^p(\xO;\ei)}\leq C \norm{\xn}_{B^{-\vartheta,p}(\Sigma)},\EAL\label{22}
	\ea

	If $\xn\in \mathfrak{M}^+(\partial\xO)$ has compact support in $\Sigma,$ we may write $\xn=\sum_{j=1}^{m_0}\xn_j,$ where $\xn_j\in \mathfrak{M}^+(\partial\xO)$ with compact support in $\mathcal{V}(\xi^j,\frac{\beta_0}{4})\cap\xS.$
	By \eqref{K-split1}, we can easily show that
	\bal \BAL
	\norm{\BBK_{\mu}[\xn]}_{L^p(\xO;\ei)}\approx \sum_{j=1}^{m_0}\norm{\BBK_{\mu}[\xn_j]}_{L^p(\xO;\ei)}\approx  C\sum_{j=1}^{m_0} \norm{\xn_j}_{B^{-\vartheta,p}(\Sigma)}
	\approx C m_0 \norm{\xn}_{B^{-\vartheta,p}(\Sigma)}.
	\EAL \eal
The proof is complete.
\end{proof}

Using Theorem \ref{potest} and Proposition \ref{subcr}, we are ready to prove Theorem \ref{supcrK}.

\begin{proof}[\textbf{Proof of Theorem \ref{supcrK}}]
	If $\xn$ is a positive measure which vanishes on Borel sets $E\subset \Sigma$ with $\mathrm{Cap}^{\BBR^{k}}_{\vartheta,p'}$-capacity zero then there exists an increasing sequence $\{\xn_n\}$ of positive measures in $B^{-\vartheta,p}(\Sigma)$ which converges weakly to $\xn$ (see \cite{DaM}, \cite {FDeP}). By Theorem \ref{potest}, we have that $\BBK_{\mu}[\xn_n]\in L^p(\xO;\ei)$, hence we may apply Proposition \ref{subcr}
with $g(t)=|t|^{p-1}t$ to deduce that there exists a unique nonnegative weak solution $u_n$ of \eqref{mainproblempower} with $\tr(u_n)=\xn_n.$
	
	Since $\{ \nu_n \}$ is an increasing sequence of positive measures, by Theorem \ref{linear-problem}, $\{u_n\}$ is increasing and its limit is denoted by $u$. Moreover,
	\be \label{lweakform-un}
	- \int_{\Gw}u_n L_{\xm }\zeta \, \dd x + \int_{\Gw} u_n^p \zeta \, \dd x = - \int_{\Gw} \mathbb{K}_{\xm}[\xn_n]L_{\xm }\zeta \, \dd x
	\qquad\forall \zeta \in\mathbf{X}_\xm(\xO).
	\ee
	By taking $\zeta=\ei$ in \eqref{lweakform-un}, we obtain
	\bal \int_{\Gw} \left(\lambda_{\xm,\Sigma} u_{n}+u_n^p\right)\ei \,\dd x=\lambda_{\xm,\Sigma}\int_{\Gw} \BBK_{\mu}[\xn_n]\ei \,\dd x,
	\eal
	which implies that $\{u_n\}$ and $\{u_n^p\}$ are uniformly bounded in $L^1(\xO;\ei)$. Therefore $u_n \to u$ in $L^1(\Omega;\ei)$ and in $L^p(\Omega;\ei)$. By letting $n \to \infty$ in \eqref{lweakform-un}, we deduce
	\bal \int_{\Gw}-uL_\xm\zeta \, \dd x +\int_{\Gw} u^p\zeta \, \dd x=-\int_{\Gw} \BBK_{\mu}[\xn]L_\xm\zeta \,\dd x\qquad\forall \zeta\in {\bf X}_\xm(\Gw).
	\eal
	This means $u$ is the unique weak solution of \eqref{mainproblempower} with $\tr(u)=\xn$.
\end{proof} \medskip

\begin{proof}[\textbf{Proof of Theorem \ref{supcromega}}] ~
	
1. Suppose $u$ is a weak solution of \eqref{mainproblempower} with $\tr(u)=1_{\partial\xO\setminus\xS}\xn$. Let $\beta>0$. Since
	\ba\label{23}
	\ei(x)\approx C(\xb)d_{\partial \Omega}(x)\quad\text{and}\quad K_{\mu}(x,y)\approx C(\xb)d_{\partial \Omega}(x)|x-y|^{-N},\quad\forall (x,y)\in (\xO\setminus \Sigma_\xb)\times \partial \xO,
	\ea
	proceeding as in the proof of \cite[Theorem 3.1]{MV-JMPA01}, we may prove that $\xn$ is absolutely continuous with respect to the Bessel capacity $\mathrm{Cap}^{\BBR^{N-1}}_{\frac{2}{p},p'}$. \medskip
	
	2. We assume that $\xn \in \GTM^+(\partial \Omega) \cap B^{-\frac{2}{p},p}(\partial\xO)$ has compact support $F \in \partial\xO\setminus\xS.$ Then by \eqref{23}, we may apply \cite[Theorem A]{MV-JMPA01} to deduce that  $\BBK_{\mu}[\xn]\in L^p(\xO\setminus \Sigma_\xb;\ei)$ for any $\xb>0$. Denote $g_n(t)=\max\{ \min\{|t|^{p-1}t,n\},-n\}$. By applying Proposition \ref{subcr} with $g=g_n$, we deduce that there exists a unique weak solution $v_n \in L^1(\Omega;\ei)$ of
	\be \label{weaksupersolution1}\left\{ \BAL
	- L_\gm v_n+g_n(v_n)&=0\qquad \text{in }\;\Gw,\\
	\tr(v_n)&=\nu,
	\EAL \right. \ee
	such that $0 \leq v_n \leq \BBK_\mu[\nu]$ in $\Omega $.
	Furthermore, by \eqref{poi5}, $\{v_n\}$ is non-increasing. Set $v=\lim_{n \to \infty}v_n,$ then $0 \leq v \leq \BBK_\mu[\nu]$ in $\Omega$.
	
Since $v_n$ is a weak solution of \eqref{weaksupersolution1}, we have
	\be \label{lweakform-un-1}
	- \int_{\Gw}v_n L_{\xm }\zeta \, \dd x + \int_{\Gw} g_n(v_n) \zeta \, \dd x = - \int_{\Gw} \mathbb{K}_{\xm}[\xn_n]L_{\xm }\zeta \, \dd x
	\qquad\forall \zeta \in\mathbf{X}_\xm(\xO).
	\ee
By taking $\ei$ as test function, we obtain
	\be \label{ungn} \int_{\Gw}\left(\lambda_{\xm,\Sigma} v_n + g_n(v_n)\right)\ei \,\dd x=\lambda_{\xm,\Sigma} \int_{\Gw} \BBK_{\mu}[\xn]\ei \,\dd x,
	\ee
	which, together with by Fatou's Lemma, implies that $v, v^p\in L^1(\xO;\ei)$ and
	\bal \int_{\Gw}\left(\lambda_{\xm,\Sigma} v + v^p \right)\ei \,\dd x
	\leq \lambda_{\xm,\Sigma} \int_{\Gw} \BBK_{\mu}[\xn]\ei \,\dd x.
	\eal
	 Hence $v+\BBG_\xm[v^p]$ is a nonnegative $L_\xm$-harmonic function. By Representation Theorem \ref{th:Rep}, there exists a unique $\overline{\xn}\in\mathfrak{M}^+(\partial\xO)$ such that $v +\BBG_\xm[v^p]=\BBK_\mu[\overline{\xn}]$.
	Since $v \leq \BBK_\mu[\xn]$, by Proposition \ref{traceKG} (i), $\overline \nu = \tr(v) \leq \tr(\BBK_\mu[\nu]) =  \nu$ and hence $\overline{\xn}$ has compact support in $F$.

Let $\xb>0$ be small enough such that $F\cap\overline{\xS}_{4\xb}=\emptyset$. We consider a cut-off function $\psi_{\beta} \in C^\infty(\R^N)$ such that $0\leq\psi_{\xb}\leq 1$ in $\R^N$, $\psi_\xb=1$ in $\xO\setminus\xS_{\frac{\xb}{2}}$ and $\psi_{\xb}=0$ in $\overline{\xS}_\frac{\xb}{4}$. Let $\xf_0$ be the eigenfunction associated to $-\Delta$ in $\Omega$ such that $\sup_{x\in\xO}\xf_0=1.$ Let $\eta\in C^\infty(\partial\xO)$ such that $\eta=0$ on $\partial\xO\cap\xS_{2\xb}$. We consider the lifting $R[\eta]$ in \cite[(1.11)]{MV-JMPA01}. Then $R\in C^2(\overline{\xO})$ has compact support in $\overline{\xO}_{\xb_0}$ for some $\xb_0>0$ small enough. In addition, $|\nabla R[\eta] \cdot \nabla \xf_0|\lesssim \xf_0$ in $\xO,$ and $R[\eta]=\eta$ for any $x\in\partial\xO.$

Then the function $\psi_{\beta,\eta}=\psi_{\beta}R[\eta]\xf_0 \in C^{1,\xg}(\overline{\xO})\cap \mathbf{X}_\xm(\xO)$ for any $\xg\in(0,1),$ $\psi_{\beta,\eta}=0$ on $\partial\xO$ and has compact support in $\overline{\xO}\setminus\xS_\frac{\xb}{4} $. Hence, by \eqref{weakfor2} and the fact that $\frac{\partial \psi_{\beta,\xz}}{\partial {\bf n}}= \frac{\partial \xf_0}{\partial {\bf n}}\eta$ on $\partial \Omega$, we obtain
\ba\label{24}
	\int_{\Gw}(-v L_\xm\psi_{\beta,\eta} + v^p\psi_{\beta,\eta})\,\dd x=-\int_{\partial\xO}\frac{\partial \xf_0}{\partial {\bf n}} \frac{\eta}{P_\mu(x_0,y)}\, \dd \overline{\nu}(y).
	\ea
	Also,
\ba \label{25}
	\int_{\Gw}(- v_nL_\xm\psi_{\xb,\xz}+g_n(v_n)\psi_{\xb,\xz})\, \dd x=-\int_{\partial\xO}\frac{\partial \xf_0}{\partial {\bf n}} \frac{\eta}{P_\mu(x_0,y)}\, \dd \nu(y).
\ea
	
	Since $v \leq v_n\leq \BBK_{\mu}[\xn]$ and $\BBK_{\mu}[\xn]\in L^p(\xO\setminus\xS_{\frac{\xb}{16}};\ei)$, by letting $n\to\infty$ in \eqref{25}, we obtain by the dominated convergence theorem that
	\ba \label{26}
	\int_{\Gw}(-v L_\xm\psi_{\xb,\xz} + v^p\psi_{\xb,\xz})\, \dd x=-\int_{\partial\xO}\frac{\partial \xf_0}{\partial {\bf n}} \frac{\eta}{P_\mu(x_0,y)}\, \dd \nu(y).
	\ea
	From \eqref{24} and \eqref{26}, we deduce that
	
\bal
-\int_{\partial\xO}\frac{\partial \xf_0}{\partial {\bf n}} \frac{\eta}{P_\mu(x_0,y)}\, \dd \nu(y)=-\int_{\partial\xO}\frac{\partial \xf_0}{\partial {\bf n}} \frac{\eta}{P_\mu(x_0,y)}\, \dd \overline{\nu}(y),
\eal
which implies that $\xn=\overline{\xn},$ since $-\frac{\partial \xf_0}{\partial {\bf n}}\approx 1$ in $\partial\xO,$ $P_\mu(x_0,y)\approx 1$ in $\partial\xO\setminus\xS_{\frac{\xb}{4}}$ and $\xn,\;\overline{\xn}$ have compact support in $\partial\xO\setminus\overline{\xS}_{4\xb}$.

	3. If $\xn \in \GTM^+(\partial \Omega)$ vanishes on Borel sets $E\subset \partial\xO$ with zero $\mathrm{Cap}^{\BBR^{N-1}}_{\frac{2}{p},p'}$-capacity and has compact support in $\partial\xO\setminus\xS$ then there exists a nondecreasing sequence $\{\xn_n\}$  of positive measures in $B^{-\frac{2}{p},p}(\partial\xO)$ which converges to $\xn$ (see \cite{DaM}, \cite {FDeP}). Let $u_n$ be the unique weak solution of \eqref{mainproblempower} with $\tr(u_n)=\xn_n$.
	Since $\{ \nu_n \}$ is nondecreasing, by \eqref{poi5}, $\{u_n\}$ is nondecreasing. Moreover, $0 \leq u_n \leq \BBK_\mu[\nu_n] \leq \BBK_\mu[\nu]$.  Denote $u=\lim_{n \to \infty}u_n$. By an argument similar to the one leading to \eqref{ungn}, we obtain
	\bal \int_{\Gw}\left(\lambda_{\xm,\Sigma} u_{n} + u_n^p\right)\ei \, \dd x=\lambda_{\xm,\Sigma}\int_{\Gw} \BBK_{\mu}[\xn_n]\ei \, \dd x,
	\eal
	which yields that $u, u^p\in L^1(\xO;\ei)$. By the dominated convergence theorem, we derive
	\bal \int_{\Gw} \left(-u L_\xm\zeta + u^p\zeta\right)\, \dd x=-\int_{\Gw}\BBK_{\mu}[\xn]L_\xm\zeta \, \dd x\qquad\forall \zeta\in {\bf X}_\mu(\Gw),
	\eal
	and thus $u$ is the unique weak solution of \eqref{mainproblempower}. \medskip

   4.  If $\1_{F}\nu$ is absolutely continuous with respect to  $\mathrm{Cap}^{\BBR^{N-1}}_{\frac{2}{p},p'}$ for any compact set $F\subset \partial\xO\setminus\xS,$ we set $\xn_n=\1_{\partial\xO\setminus\xS_{\frac{1}{n}}}$ and $u_n$ the weak solution of \eqref{mainproblempower} with $\tr(u_n)=\xn_n$. By using an argument similar to that in case 3, we obtain the desired result.
\end{proof}

\section{Semilinear equations with a power source nonlinearity} \label{sec:powercase}
In this section we study the following problem
\ba \tag{$\text{BVP}_-^{\sigma}$} \label{power} \left\{ \BAL
- L_\gm u&= |u|^{p-1}u\qquad \text{in }\;\Gw,\\
\tr(u)&=\xs\xn,
\EAL \right. \ea
where $p>1$, $\sigma$ is a positive parameter and $\nu \in \mathfrak{M}^+(\partial\xO)$.

We remark that a positive function $u$ is a weak solution of \eqref{power} if and only if
\ba \label{u-source-sigma}
u = \BBG_\mu[u^p] + \sigma \BBK_\mu[\nu] \quad \text{a.e. in } \Omega.
\ea

In the following proposition, we give a necessary and sufficient condition for the existence of solutions to problem \eqref{power}.
\begin{proposition}\label{equivba}
Assume $\mu \leq H^2$, $p>1$ and $\nu \in \mathfrak{M}^+(\partial\xO).$ Then problem \eqref{power} admits a weak solution if and only if there exists a positive constant $C>0$ such that
\ba\label{subcrsource}
\BBG_\xm[\BBK_\xm[\xn]^p]\leq C\, \BBK_\xm[\xn] \quad \text{a.e. in } \Omega.
\ea
\end{proposition}
\begin{proof}[\textbf{Proof}]
The proof is similar to that of \cite[Proposition 6.2]{GkiNg_source} with some minor modifications, and hence we omit
it.
\end{proof}

\subsection{Preparative results } For $\xa\leq N$, set
\ba\label{Na}
\CN_{\ga}(x,y):=\frac{\max\{|x-y|,d_\xS(x),d_\xS(y)\}^\xa}{|x-y|^{N-2}\max\{|x-y|,d_{\partial \Omega}(x),d_{\partial \Omega}(y)\}^2},\quad (x,y)\in\overline{\xO}\times\overline{\xO}, x \neq y,
\ea
\bel{opN} \BBN_{\xa}[\omega](x):=\int_{\overline{\Gw}} \CN_{\xa}(x,y) \dd\omega(y), \quad \omega \in \GTM^+(\overline \Gw).\ee
Let $\alpha <N$, $b>0$, $\theta>-N+k$ and $s>1.$ We define the capacity $\text{Cap}_{\Nthb,s}^{b,\theta}$ by
\bal \text{Cap}_{\Nthb,s}^{b,\theta}(E) :=\inf\left\{\int_{\overline{\xO}}d_{\partial \Omega}^b d^\theta_\xS\gf^s\,\dx:\;\; \gf \geq 0, \;\;\Nthb[ d_{\partial \Omega}^b d_{\Sigma}^\theta\gf ]\geq\1_E\right\} \quad \text{for any Borel set } E\subset\overline{\xO}.
\eal
Here $\1_E$ denotes the indicator function of $E$. By \cite[Theorem 2.5.1]{Ad}, we have
\ba\label{dualcap}
(\text{Cap}_{\Nthb,s}^{b,\theta}(E))^\frac{1}{s}=\sup\{\omega(E):\omega \in\GTM^+(E), \|\Nthb[\omega]\|_{L^{s'}(\xO;d_{\partial \Omega}^b d^\theta_\xS)} \leq 1 \}.
\ea

We note here that if $\xm<\frac{N^2}{4}$ and $\xn\in\mathfrak{M}^+(\partial\xO),$ then, by \eqref{Greenesta} and \eqref{Martinest1}, we can easily show that
\ba \label{GN2a}
G_\mu(x,y)\approx d_{\partial \Omega}(x)d_{\partial \Omega}(y)(d_{\xS}(x)d_{\xS}(y))^{-\am} \CN_{2\am}(x,y) \quad \forall x,y \in \Omega, x \neq y
\ea
and
\ba\label{KN2a1}
K_\mu(x,y)\approx d_{\partial \Omega}(x)d_{\xS}(x)^{-\am} \CN_{2\am}(x,y) \quad \forall (x,y) \in \Omega\times\partial\xO.
\ea
Therefore, if the integral equation
\ba \label{vN2a} v=\BBN_{2\am}[(d_{\partial \Omega}d_{\xS}^{-\am})^{p+1}v^p]+\ell \BBN_{2\am}[\xn]
\ea
has a solution $v$ for some $\ell>0$ then $\tilde v=d_{\partial \Omega}d_{\Sigma}^{-\am}v$ satisfies
\ba\label{105}
\tilde v\approx\BBG_\xm[\tilde v^p]+\ell \BBK_\xm[\xn].
\ea
This, together with \cite[Proposition 2.7]{BHV}, implies that equation \eqref{u-source-sigma} has a positive solution $u$ for some $\xs>0$, whence problem \eqref{power} has a weak positive solution $u$ for some $\sigma>0$.

In order to show that \eqref{vN2a} possesses a solution, we will apply the results in \cite{KV} which we
recall here for the sake of completeness.

Let $\mathbf{Z}$ be a metric space and $\gw \in\GTM^+(\mathbf{Z}).$ Let $J : \mathbf{Z} \times \mathbf{Z} \to (0,\infty]$ be a Borel positive kernel such that $J$ is symmetric and $1/J$ satisfies a quasi-metric
inequality, i.e. there is a constant $C>1$ such that for all $x, y, z \in \mathbf{Z}$,
\ba \label{Jest} \frac{1}{J(x,y)}\leq C\left(\frac{1}{J(x,z)}+\frac{1}{J(z,y)}\right).
\ea
Under these conditions, one can define the quasi-metric $\mathbf{d}$ by
\bal
\mathbf{d}(x,y):=\frac{1}{J(x,y)}
\eal
and denote by $\GTB(x,r):=\{y\in\mathbf{Z}:\; \mathbf{d}(x,y)<r\}$ the open $\mathbf{d}$-ball of radius $r > 0$ and center $x$.
Note that this set can be empty.

For $\xo\in\GTM^+(\mathbf{Z})$ and a positive function $\phi$, we define the potentials $\BBJ[\gw]$ and $\BBJ[\gf,\gw]$ by
\bal
\BBJ[\gw](x):=\int_{\mathbf{Z}}J(x,y) \dd\xo(y)\quad\text{and}\quad \BBJ[\gf,\gw](x):=\int_{\mathbf{Z}}J(x,y)\gf(y) \dd\gw(y).
\eal
For $t>1$ the capacity $\text{Cap}_{\BBJ,t}^\gw$ in $\mathbf{Z}$ is defined for any Borel $E\subset\mathbf{Z}$ by
\bal
\text{Cap}_{\BBJ,t}^\gw(E):=\inf\left\{\int_{\mathbf{Z}}\gf(x)^{t} \dd\gw(x):\;\;\gf\geq0,\;\; \BBJ[\gf,\gw] \geq\1_E\right\}.
\eal

\begin{proposition}\label{t2.1}  \emph{(\cite{KV})} Let $p>1$ and $\gl,\gw \in\GTM^+(\mathbf{Z})$ such that
	\ba
	\int_0^{2r}\frac{\gw\left(\GTB(x,s)\right)}{s^2} \dd s &\leq C\int_0^{r}\frac{\gw\left(\GTB(x,s)\right)}{s^2} \dd s ,\label{2.3}\\
	\sup_{y\in \GTB(x,r)}\int_0^{r}\frac{\gw\left(\GTB(y,r)\right)}{s^2} \dd s&\leq C\int_0^{r}\frac{\gw\left(\GTB(x,s)\right)}{s^2} \dd s,\label{2.4}
	\ea
for any $r > 0,$ $x \in \mathbf{Z}$, where $C > 0$ is a constant. Then the following statements are equivalent.
	
	1. The equation $v=\BBJ[|v|^p,\gw]+\ell \BBJ[\gl]$ has a positive solution for $\ell>0$ small.
	
	2. For any Borel set $E \subset \mathbf{Z}$, there holds
	$
\int_E \BBJ[\1_E\gl]^p \dd \gw \leq C\, \gl(E).
	$
	
	3. The following inequality holds
	$
\BBJ[\BBJ[\gl]^p,\gw]\leq C\BBJ[\gl]<\infty\quad \gw-a.e.
$

	4. For any Borel set $E \subset \mathbf{Z}$ there holds
$
\gl(E)\leq C\, \emph{Cap}_{\BBJ,p'}^\gw(E).
$
\end{proposition}

We will point out below that $\BBN_\xa$ defined in \eqref{opN} with $\dd \gw=d_{\partial \Omega}(x)^b d_\xS(x)^\theta \1_{\Omega}(x)\,\dx$ satisfies all assumptions of $\BBJ$  in Proposition \ref{t2.1}, for some appropriate $b,\theta\in \BBR$. Let us first prove the quasi-metric inequality.

\begin{lemma}\label{ineq}
Let $\xa\leq N.$ There exists a positive constant $C=C(\xO,\Sigma,\xa)$ such that
\ba \label{dist-ineq} \frac{1}{\CN_\xa(x,y)}\leq C\left(\frac{1}{\CN_\xa(x,z)}+\frac{1}{\CN_\xa(z,y)}\right),\quad \forall x,y,z\in \overline{\xO}.
\ea
\end{lemma}
\begin{proof}[\textbf{Proof}]
\noindent \textbf{Case 1: $0<\xa\leq N$.}  We first assume that $|x-y|<2|x-z|$. Then by the triangle inequality, we have
$d_\xS(z)\leq |x-z|+d_\xS(x)\leq 2\max\{|x-z|,d_\xS(x)\}$ hence
\bal
\max\{|x-z|,d_\xS(x),d_\xS(z)\}\leq 2\max\{ |x-z|,d_\xS(x)\}.
\eal
If $|x-z|\geq d_\xS(x)$ then $|x-z|\geq d_{\partial \Omega}(x)$ which implies that $|x-z|\geq \frac{d_{\partial \Omega}(x)+|x-y|}{4}.$ Now,
\ba\nonumber
&\frac{|x-z|^{N-2}\max\{|x-z|,d_{\partial \Omega}(x),d_{\partial \Omega}(z)\}^2}{\max\{|x-z|,d_\xS(x),d_\xS(z)\}^\xa}\geq 2^{-\xa}|x-z|^{N-\xa}\geq 2^{-2N+\xa}(|x-y|+d_{\partial \Omega}(x))^{N-\xa}\\
&\qquad\qquad=2^{-2N+\xa}\frac{(|x-y|+d_{\partial \Omega}(x))^{N}}{(|x-y|+d_{\partial \Omega}(x))^\xa}\gtrsim \frac{|x-y|^{N-2}\max\{|x-y|,d_{\partial \Omega}(x),d_{\partial \Omega}(y)\}^2}{\max\{|x-y|,d_\xS(x),d_\xS(y)\}^\xa},\label{36}
\ea
since $d_{\partial \Omega}(x)\leq d_\xS(x).$

If $|x-z|\leq d_\xS(x)$ then
\ba\nonumber
\frac{|x-z|^{N-2}\max\{|x-z|,d_{\partial \Omega}(x),d_{\partial \Omega}(z)\}^2}{\max\{|x-z|,d_\xS(x),d_\xS(z)\}^\xa}&\geq 2^{-\xa-2}d_\xS(x)^{-\xa}|x-z|^{N-2}\max\{|x-y|,d_{\partial \Omega}(x)\}^2 \\
&\gtrsim\frac{|x-y|^{N-2}\max\{|x-y|,d_{\partial \Omega}(x),d_{\partial \Omega}(y)\}^2}{\max\{|x-y|,d_\xS(x),d_\xS(y)\}^\xa},\label{37}
\ea
since $d_{\partial \Omega}(y)\leq |x-y|+d_{\partial \Omega}(x)\leq 2\max\{|x-y|,d_\xS(x)\}.$
Combining \eqref{36}--\eqref{37}, we obtain \eqref{dist-ineq}.

Next we consider the case $2|x-z|\leq|x-y|$. Then $\frac{1}{2} |x-y|\leq |y-z|,$ thus by symmetry we obtain \eqref{dist-ineq}. \medskip

\noindent \textbf{Case 2: $\xa\leq 0$.} Let $b\in[0,2],$ since $d_\xS(x)\leq |x-y|+d_\xS(y)$,  it follows that
\bal
\max\{|x-y|,d_\xS(x),d_\xS(y)\}\leq |x-y|+\min\{d_\xS(x),d_\xS(y)\}.
\eal
Using the above estimate, we obtain
\ba\label{35}\BAL
&|x-y|^{N-b}\max\{|x-y|,d_\xS(x),d_\xS(y)\}^{-\xa} \\
&\lesssim |x-y|^{N-b-\xa}+\min\{d_\xS(x),d_\xS(y)\}^{-\xa}|x-y|^{N-b}\\
&\lesssim|x-z|^{N-b-\xa}+|y-z|^{N-b-\xa}+\min\{d_\xS(x),d_\xS(y)\}^{-\xa}(|x-z|^{N-b}+|y-z|^{N-b})\\
&\lesssim
\frac{|x-z|^{N-b}}{\max\{|x-z|,d_\xS(x),d_\xS(z)\}^\xa}+\frac{|z-y|^{N-b}}{\max\{|z-y|,d_\xS(z),d_\xS(y)\}^\xa}.
\EAL
\ea

 \smallskip

 Since $d_{\partial \Omega}(x)\leq |x-y|+d_{\partial \Omega}(y),$ we can easily show that $\max\{|x-y|,d_{\partial \Omega}(x),d_{\partial \Omega}(y)\}\leq |x-y|+\min\{d_{\partial \Omega}(x),d_{\partial \Omega}(y)\}.$ Hence
\bal
\frac{1}{\CN_\xa(x,y)}&=\frac{\max\{|x-y|,d_{\partial \Omega}(x),d_{\partial \Omega}(y)\}^2|x-y|^{N-2}}{\max\{|x-y|,d_\xS(x),d_\xS(y)\}^\xa} \\
&\leq  \frac{2|x-y|^{N}}{\max\{|x-y|,d_\xS(x),d_\xS(y)\}^\xa}
+\frac{2\min\{d_{\partial \Omega}(x),d_{\partial \Omega}(y)\}^2|x-y|^{N-2}}{\max\{|x-y|,d_\xS(x),d_\xS(y)\}^\xa}.
\eal
The desired result follows by the above inequality and \eqref{35}.
\end{proof}

Next we give sufficient conditions for \eqref{2.3} and \eqref{2.4} to hold.
\begin{lemma}\label{l2.3} Let $b>0$, $\theta+b>k-N$ and  $\dd \gw=d_{\partial \Omega}(x)^b d_\xS(x)^\theta \1_{\xO}(x)\,\dx$. Then
\be \label{doubling}
\gw(B(x,s))\approx \max\{d_{\partial \Omega}(x),s\}^b\max\{d_\xS(x),s\}^{\theta} s^N, \; \text{for all } x\in\xO \text{ and } 0 < s\leq 4 \mathrm{diam}(\xO).
\ee
\end{lemma}
\begin{proof}[\textbf{Proof}]
Let $\xb_0$ and $\beta_1$ be as in Subsection \ref{assumptionK} and $s<\frac{\xb_1}{32}.$ We first assume that $x\in \xS_{\frac{\xb_1}{4}}$.

\noindent\textbf{Case 1:}  $d_{\partial \Omega}(x)\geq 2s$. Let $\xG=\partial\xO$ or $\xS.$ Then $\frac{1}{2}d_{\xG}(x)\leq d_\xG(y)\leq\frac{3}{2} d_\xG(x)$ for any $y\in B(x,s)$, therefore \eqref{doubling} follows easily in this case. \medskip

\noindent \textbf{Case 2:} $d_{\partial \Omega}(x)\leq 2s$ and $d_\xS(x)\geq 2s$. By estimate (2.9) in \cite[Lemma 2.3]{BHV}, we have that
\ba\label{101}
\int_{B(x,s)\cap\xO}d_{\partial \Omega}(y)^b \dd y\approx \max\{d_{\partial \Omega}(x),s\}^b s^N.
\ea
Therefore
\bal
\int_{B(x,s)\cap\xO}d_{\partial \Omega}(y)^b d_\xS(y)^\theta\,\dy\approx d_\xS(x)^\theta\int_{B(x,s)\cap\xO}d_{\partial \Omega}(y)^b \dd y\approx \max\{d_{\partial \Omega}(x),s\}^b\max\{d_\xS(x),s\}^{\theta} s^N.
\eal

\noindent \textbf{Case 3:} $d_{\partial \Omega}(x)\leq 2s$ and $d_\xS(x)\leq 2s$. By \eqref{cover}, there exists  $\xi \in\xS $ such that $B(x,s)\cap\xO\subset V_\xS(\xi,\xb_0)$. If $y\in B(x,s),$ then $|y'-x'|<s$ and $d_\xS(y)\leq d_\xS(x)+|x-y|\leq 3s$.
Thus, by \eqref{propdist},
$\xd_\Sigma^\xi(y)\leq C_1s$ for any $y\in B(x,s)\cap\xO$, where $C_1$ depends on $\| \Sigma \|_{C^2},N$ and $k$. Therefore
\bal
\int_{B(x,s)\cap\xO}d_{\partial \Omega}(y)^b d_\xS(y)^\theta \dy &\lesssim \int_{\{|x'-y'|<s\}}\int_{\{\xd_\Sigma^\xi(y)\leq C_1s\}}(\xd_\Sigma^\xi(y))^{\theta+b} \dy''\dy'
\approx s^{N+\theta+b} \\
&\approx \max\{d_{\partial \Omega}(x),s\}^b\max\{d_\xS(x),s\}^{\theta} s^N.
\eal
Here the similar constants depend on $N,k,\| \Sigma \|_{C^2}$ and $\beta_0$. \medskip

\noindent \textbf{Case 4:} $d_{\partial \Omega}(x)\leq 2s$ and $d_\xS(x)\leq 2s$ and $\theta<0$. We have that $d_\xS(y)^\theta \geq3^\theta s^{\theta}$  for any $y\in B(x,s).$ Hence
\bal
\int_{B(x,s)\cap\xO}d_{\partial \Omega}(y)^b d_\xS(y)^\theta\,\dy\gtrsim s^\theta\int_{B(x,s)\cap\xO}d_{\partial \Omega}(y)^b \dd y\approx \max\{d_{\partial \Omega}(x),s\}^b\max\{d_\xS(x),s\}^{\theta} s^N.
\eal

\medskip

\noindent \textbf{Case 5:} $d_{\partial \Omega}(x)\leq 2s$ and $d_\xS(x)\leq 2s$ and $\theta \geq 0$. By \eqref{cover2}, there exists  $\xi \in\xS $ such that $B(x,s)\cap\xO\subset \mathcal{V}_{\Sigma}(\xi,\xb_0).$ Let $C_\xS,$ $C_{\partial\xO}$ be as in \eqref{propdist}, $A$ be as in \eqref{propdist2} and $C_2=\max\{C_\xS \| \Sigma \|_{C^2},C_{\partial\xO}\| \partial\xO \|_{C^2}\}(A+1).$

We first assume that $d_{\partial \Omega}(x)\leq \frac{s}{12NC_2}$ $d_\xS(x) \leq \frac{s}{12NC_2}.$ Set
\begin{align} \nonumber
\mathcal{A}:=\left\{\psi=(\psi',\psi'')\in\xO: |x'-\psi'|<r_0,\; |\xd(\psi)|<r_0,\; |\xd_{2,\xS}(\psi)|<r_0\right\},
\end{align}
where $r_0=\frac{s}{12N(A+1)}.$ By \eqref{propdist}, we have $\xd_\Sigma^\xi(x)\leq\frac{s}{12N(A+1)}$ and $\xd_{\partial\xO}^\xi(x)\leq\frac{s}{12N(A+1)}$. In addition for any $y\in \CA,$  we have
\bal
&|x''-y''|\leq \xd_\Sigma^\xi(x)+\xd_\Sigma^\xi(y)+\left(\sum_{i=k+1}^N|\xG_{i,\xS}^\xi(x')-\xG_{i,\xS}^\xi(y')|^2\right)^\frac{1}{2}\\
&\qquad\qquad\leq  \frac{s}{3}+(N-k)\| \Sigma \|_{C^2}|x'-y'|+ A(\xd_{2,\xS}^\xi(y)+\xd^\xi(y))<s,
\eal
where in the last inequality we used \eqref{propdist2}.
This implies that $\CA\subset B(x,s)$.  Consequently,
\ba\label{202}\BAL
&\int_{B(x,s)\cap\xO}d_{\partial \Omega}(y)^b d_\xS(y)^\theta\dy\approx \int_{B(x,s)\cap\xO}\xd^\xi(y)^b(\xd_\xS^\xi(y))^\theta \dy\gtrsim \int_{\CA}\xd^\xi(y)^b(\xd^\xi(y)+\xd_{2,\xS}^\xi(y))^\theta \dy\\
&\qquad\qquad\approx  s^{N+\theta+b} \approx C\approx \max\{d_{\partial \Omega}(x),s\}^b\max\{d_\xS(x),s\}^{\theta} s^N.
\EAL \ea

If $d_{\partial \Omega}(x)\leq \frac{s}{12NC_2}$ and $d_\xS(x)\geq \frac{s}{12NC_2}$ then
\bal
\int_{B(x,s)}d_{\partial \Omega}(y)^b d_\xS(y)^\theta \dy\geq\int_{B(x,\frac{s}{24NC_2})}d_{\partial \Omega}(y)^b d_\xS(y)^\theta \dy
\eal
and hence \eqref{202} follows by case 2.

If $d_{\partial \Omega}(x)\geq \frac{s}{12NC_2}$ then
\bal
\int_{B(x,s)}d_{\partial \Omega}(y)^b d_\xS(y)^\theta \dy\geq\int_{B(x,\frac{s}{24NC_2})}d_{\partial \Omega}(y)^b d_\xS(y)^\theta \dy
\eal
and hence \eqref{202} follows by Case 1.

Next we consider $x\in \xO\setminus\xS_{\frac{\xb_1}{4}}$ and $s<\frac{\xb_1}{32}$. Then $d_\xS(y)\approx 1$ for any $y\in \xO\cap B(x,s)$. This, together  with \eqref{101}, implies the desired result.

If $ \frac{\xb_1}{16}\leq s\leq 4 \text{diam}(\xO)$ then  $\gw(B(x,s))\approx 1$, hence estimate \eqref{doubling} follows straightforward. The proof is complete.
\end{proof}

\begin{lemma}\label{vol}
Let $\xa< N$, $b> 0$, $\theta>\max\{k-N-b,-b-\xa\}$ and $\dd \gw=d_{\partial \Omega}(x)^b d_\xS(x)^\theta$ $\1_{\Omega }(x)\,\dx$.  Then \eqref{2.3}  holds.
\end{lemma}
\begin{proof}[\textbf{Proof}]
We note that if $s\geq (4\diam(\xO))^{N-\xa}$ then $\xo(\GTB(x,s))=\xo(\overline{\xO})<\infty$, where $\GTB(x,s)$ is defined after \eqref{Jest}, namely $\GTB(x,s)=\{y \in \Omega \setminus \Sigma: \mathbf{d}(x,y)<s\}$ and $\mathbf{d}(x,y)= \frac{1}{\CN_\xa(x,y)}$.
Let $M=(4\diam\xO)^{N-\xa}.$ We first note that it is enough to show that
\ba\label{xoest1a}
\gw\left(\GTB(x,t)\right)\approx\left\{
\BAL
&d_{\partial \Omega}(x)^{b-\frac{2N}{N-2}} d_\xS(x)^{\theta+\frac{\xa N}{N-2}} s^\frac{N}{N-2}&&\quad \text{if}\; s\leq d_{\partial \Omega}(x)^N d_\xS(x)^{-\xa},\\
&s^\frac{b+N}{N}d_\xS(x)^{\theta+\xa\frac{b+N}{N}}&&\quad \text{if}\;d_{\partial \Omega}(x)^N d_\xS(x)^{-\xa}<s\leq d_\xS(x)^{N-\xa},\\
&s^\frac{b+\theta+N}{N-\xa}&&\quad \text{if}\; d_\xS(x)^{N-\xa}\leq s\leq M,\\
&M^\frac{b+\theta+N}{N-\xa}&&\quad \text{if}\;  M\leq s.
\EAL\right.
\ea
Indeed, by the above display, we can easily deduce that
\ba\label{xoest1}
\int_0^{s}\frac{\gw\left(\GTB(x,t)\right)}{t^2} \dd t\approx\left\{
\BAL
&d_{\partial \Omega}(x)^{b-\frac{2N}{N-2}}d_\xS(x)^{\theta+\frac{\xa N}{N-2}}s^\frac{2}{N-2} &&\quad \text{if}\; s\leq d_{\partial \Omega}(x)^N d_\xS(x)^{-\xa}, \\
&s^\frac{b}{N}d_\xS(x)^{\theta+\xa\frac{b+N}{N}}&&\quad \text{if}\;d_{\partial \Omega}(x)^Nd_\xS(x)^{-\xa}<s\leq d_\xS(x)^{N-\xa},\\
&s^\frac{b+\theta+\xa}{N-\xa}&&\quad \text{if}\; d_\xS(x)^{N-\xa}\leq s\leq M,\\
&M^\frac{b+\theta+\xa}{N-\xa}&&\quad \text{if}\;  M\leq s,
\EAL\right.
\ea
since $b>0$ and $b+\theta+\xa>0.$ This in turn implies \eqref{2.3}.

In order to show \eqref{xoest1a}, we will consider three cases.

\noindent \textbf{Case 1:} $s\leq d_{\partial \Omega}(x)^Nd_\xS(x)^{-\xa}$.

a) Let $y\in\GTB(x,s)$ be such that $d_{\partial \Omega}(x)\leq|x-y|$ and $d_\xS(x)\leq |x-y|.$ Then
\bal
\frac{1}{\CN_\xa(x,y)}\approx|x-y|^{N-\xa},
\eal
thus if $|x-y|^{N-\xa}\leq s\leq d_{\partial \Omega}(x)^N d_\xS(x)^{-\xa}\leq d_{\partial \Omega}(x)^{N-\xa}$ then $d_{\partial \Omega}(x)\approx d_\xS(x)\approx |x-y|.$ Hence, there exist constants $C_1,C_2$ depending only on $\xa,N$ such that
\ba\label{case1a}\BAL
& \left\{y\in\xO:|x-y|\leq C_1(d_\xS(x)^\xa d_{\partial \Omega}(x)^{-2}s)^\frac{1}{N-2},\;d_{\partial \Omega}(x)\leq|x-y|,\;d_\xS(x)\leq |x-y| \right\}\\
&\quad\subset \left\{y\in\xO:\frac{1}{\CN_\xa(x,y)}\leq s,\;d_{\partial \Omega}(x)\leq|x-y|,\;d_\xS(x)\leq |x-y| \right\}\\
&\qquad\subset \left\{y\in\xO:|x-y|\leq C_2(d_\xS(x)^\xa d_{\partial \Omega}(x)^{-2} s)^\frac{1}{N-2},\;d_{\partial \Omega}(x)\leq|x-y|,\;d_\xS(x)\leq |x-y| \right\}.
\EAL
\ea

\medskip

b) Let $y\in\GTB(x,s)$ be such that $d_{\partial \Omega}(x)\leq|x-y|$ and $d_\xS(x)> |x-y|.$ Then
\bal
\frac{1}{\CN_\xa(x,y)}\approx|x-y|^{N}d_\xS(x)^{-\xa},
\eal
thus if $|x-y|^{N}d_\xS(x)^{-\xa} \leq s,$ then $|x-y|^{N}\leq s d_\xS(x)^{\xa}\leq d_{\partial \Omega}(x)^N$. Thus  $d_{\partial \Omega}(x)\approx |x-y|.$ Hence, there exist constants $C_1,C_2$ depending only on $\xa,N$ such that
\ba\label{case1b}\BAL
&\left\{y\in\xO:|x-y|\leq C_1(d_\xS(x)^\xa d_{\partial \Omega}(x)^{-2} s)^\frac{1}{N-2},\;d_{\partial \Omega}(x)\leq|x-y|,\;d_\xS(x)> |x-y| \right\}\\
&\quad \subset \left\{y\in\xO:\frac{1}{\CN_\xa(x,y)}\leq s,\;d_{\partial \Omega}(x)\leq|x-y|,\;d_\xS(x)> |x-y|\right\}\\
&\qquad\subset \left\{y\in\xO:|x-y|\leq C_2(d_\xS(x)^\xa d_{\partial \Omega}(x)^{-2} s)^\frac{1}{N-2},\;d_{\partial \Omega}(x)\leq|x-y|,\;d_\xS(x)> |x-y| \right\}.
\EAL
\ea

\medskip

c)  Let $y\in\GTB(x,s)$ be such that $d_{\partial \Omega}(x)>|x-y|.$ Then, $d_\xS(x)\geq d_{\partial \Omega}(x)>|x-y|$ and
\bal
\frac{1}{\CN_\xa(x,y)}\approx|x-y|^{N-2}d_{\partial \Omega}(x)^2 d_\xS(x)^{-\xa}.
\eal
 Hence, there exist constants $C_1,C_2$ depending only on $\xa,N$ such that
\ba\label{case1c}\BAL
& \left\{x\in\xO:|x-y|\leq C_1(d_\xS(x)^\xa d_{\partial \Omega}(x)^{-2} s)^\frac{1}{N-2},\;d_{\partial \Omega}(x)>|x-y| \right\}\\
&\quad\subset \left\{x\in\xO:\frac{1}{\CN_\xa(x,y)}\leq s,\;d_{\partial \Omega}(x)>|x-y| \right\}\\
&\qquad\subset \left\{x\in\xO:|x-y|\leq C_2(d_\xS(x)^\xa d_{\partial \Omega}(x)^{-2} s)^\frac{1}{N-2},\;d_{\partial \Omega}(x)>|x-y|\right\}.
\EAL
\ea

Combining \eqref{case1a}--\eqref{case1c} and Lemma \ref{l2.3}, we deduce
\bal
\xo(\GTB(x,s))\approx\xo(B(x,s_1))\approx d_{\partial \Omega}(x)^{b-\frac{2N}{N-2}} d_\xS(x)^{\theta+\frac{\xa N}{N-2}} s^\frac{N}{N-2},
\eal
where $s_1=(d_\xS(x)^\xa d_{\partial \Omega}(x)^{-2}s)^\frac{1}{N-2}$.

\medskip

\noindent \textbf{Case 2:} $ d_{\partial \Omega}(x)^N d_\xS(x)^{-\xa}<s\leq d_\xS(x)^{N-\xa}.$

a) Let $y\in\GTB(x,s)$ be such that $d_{\partial \Omega}(x)\leq|x-y|$ and $d_\xS(x)\leq |x-y|.$ Then
\bal
\frac{1}{\CN_\xa(x,y)}\approx|x-y|^{N-\xa}.
\eal
Thus, if $|x-y|^{N-\xa}\leq s\leq d_\xS(x)^{N-\xa}$ then $d_\xS(x)\approx |x-y|$. Hence, there exist constants $C_1,C_2$ which depending only on $\xa,N$ such that
\ba\label{case2a}\BAL
&\left\{y\in\xO:|x-y|\leq C_1(d_\xS(x)^\xa s)^\frac{1}{N},\;d_{\partial \Omega}(x)\leq|x-y|,\;d_\xS(x)\leq |x-y| \right\}\\
&\quad\subset \left\{y\in\xO:\frac{1}{\CN_\xa(x,y)}\leq s,\;d_{\partial \Omega}(x)\leq|x-y|,\;d_\xS(x)\leq |x-y| \right\}\\
&\qquad\subset \left\{y\in\xO:|x-y|\leq C_2(d_\xS(x)^\xa s)^\frac{1}{N},\;d_{\partial \Omega}(x)\leq|x-y|,\;d_\xS(x)\leq |x-y| \right\}.
\EAL
\ea

\medskip

b) Let $y\in\GTB(x,s)$ be such that $d_{\partial \Omega}(x)\leq|x-y|$ and $d_\xS(x)> |x-y|.$ Then
\bal
\frac{1}{\CN_\xa(x,y)}\approx|x-y|^{N}d_\xS(x)^{-\xa}.
\eal
Hence, there exist constants $C_1,C_2$ depending only on $\xa,N$ such that
\ba\label{case2b}\BAL
&\left\{y\in\xO:|x-y|\leq C_1(d_\xS(x)^\xa s)^\frac{1}{N},\;d_{\partial \Omega}(x)\leq|x-y|,\;d_\xS(x)> |x-y|\right \}\\
&\quad\subset \left\{y\in\xO:\frac{1}{\CN_\xa(x,y)}\leq s,\;d_{\partial \Omega}(x)\leq|x-y|,\;d_\xS(x)> |x-y|\right \}\\
&\qquad\subset \left\{y\in\xO:|x-y|\leq C_2(d_\xS(x)^\xa s)^\frac{1}{N},\;d_{\partial \Omega}(x)\leq|x-y|,\;d_\xS(x)> |x-y| \right\}.
\EAL
\ea

\medskip

c)  Let $y\in\GTB(x,s)$ be such that $d_{\partial \Omega}(x)>|x-y|.$ Then $d_\xS(x)> |x-y|.$ In addition,
\bal
\frac{1}{\CN_\xa(x,y)}\approx|x-y|^{N-2}d_{\partial \Omega}(x)^2 d_\xS(x)^{-\xa},
\eal
\bal
|x-y|^{N-2}d_{\partial \Omega}(x)^2d_\xS(x)^{-\xa}\geq |x-y|^{N}d_\xS(x)^{-\xa},
\eal
and
\bal
|x-y|\leq\bigg(d_\xS(x)^\xa s\bigg)^\frac{1}{N}=\bigg(d_\xS(x)^\xa s\bigg)^\frac{1}{N-2}\bigg(d_\xS(x)^\xa s\bigg)^{-\frac{2}{N(N-2)}}
\leq\bigg(d_\xS(x)^\xa d_{\partial \Omega}(x)^{-2} s\bigg)^\frac{1}{N-2},
\eal
since $d_{\partial \Omega}(x)^N d_\xS(x)^{-\xa}<s$. Hence, there exist constants $C_1,C_2$ depending only on $\xa,N$ such that
\ba\label{case2c}\BAL
&\left\{y\in\xO:|x-y|\leq C_1(d_\xS(x)^\xa s)^\frac{1}{N},\;d_{\partial \Omega}(x)>|x-y| \right \}\\
&\quad \subset \left\{y\in\xO:\frac{1}{\CN_\xa(x,y)}\leq s,\;d_{\partial \Omega}(x)>|x-y| \right\}\\
&\qquad \subset \left\{y\in\xO:|x-y|\leq C_2(d_\xS(x)^\xa s)^\frac{1}{N},\;d_{\partial \Omega}(x)>|x-y| \right\}.
\EAL
\ea

Combining \eqref{case2a}-\eqref{case2c} and Lemma \ref{l2.3}, we derive
\bal
\xo(\GTB(x,s))\approx\xo(B(x,s_2))\approx s^\frac{b+N}{N}d_\xS(x)^{\theta+\xa\frac{b+N}{N}},
\eal
where $s_2=\bigg(d_\xS(x)^{\xa}s\bigg)^\frac{1}{N}.$

\medskip

\noindent \textbf{Case 3:} $  d_\xS(x)^{N-\xa}<s\leq (4\diam(\xO))^{N-\xa}.$

a) Let $y\in\GTB(x,s)$ be such that $d_{\partial \Omega}(x)\leq|x-y|$ and $d_\xS(x)\leq |x-y|.$ Then
\bal
\frac{1}{\CN_\xa(x,y)}\approx|x-y|^{N-\xa}.
\eal
Hence, there exist constants $C_1,C_2$ which depends only on $\xa,N$ such that
\ba\label{case3a}\BAL
&\left\{y\in\xO:|x-y|\leq C_1s^\frac{1}{N-\xa},\;d_{\partial \Omega}(x)\leq|x-y|,\;d_\xS(x)\leq |x-y| \right\}\\
&\quad\subset \left\{y\in\xO:\frac{1}{\CN_\xa(x,y)}\leq s,\;d_{\partial \Omega}(x)\leq|x-y|,\;d_\xS(x)\leq |x-y| \right\}\\
&\qquad\subset \left\{y\in\xO:|x-y|\leq C_2s^\frac{1}{N-\xa},\;d_{\partial \Omega}(x)\leq|x-y|,\;d_\xS(x)\leq |x-y| \right\}.
\EAL
\ea

\medskip

b) Let $y\in\GTB(x,s)$ be such that $d_{\partial \Omega}(x)\leq|x-y|$ and $d_\xS(x)> |x-y|.$ Then
\bal
\frac{1}{\CN_\xa(x,y)}\approx|x-y|^{N}d_\xS(x)^{-\xa}.
\eal

On one hand, if $\xa>0,$ we have
\bal
|x-y|^{N}d_\xS(x)^{-\xa}\leq|x-y|^{N-\xa}
\eal
and
\bal
|x-y|\leq\bigg(d_\xS(x)^\xa s\bigg)^\frac{1}{N}=s^\frac{1}{N-\xa} s^{-\frac{\xa}{N(N-\xa)}} d_\xS(x)^{\frac{\xa}{N}}
\leq s^\frac{1}{N-\xa},
\eal
since $d_\xS(x)^{N-\xa}<s$.

On the other hand, if $\xa\leq 0$ then
\bal
|x-y|^{N}d_\xS(x)^{-\xa}\geq|x-y|^{N-\xa}
\eal
and
\bal
|x-y|\leq s^\frac{1}{N-\xa}=s^\frac{1}{N} s^{\frac{\xa}{N(N-\xa)}}\leq \bigg(d_\xS(x)^\xa s\bigg)^\frac{1}{N},
\eal
since $d_\xS(x)^{N-\xa}<s$.

Hence, there exist constants $C_1,C_2$ which depends only on $\xa,N$ such that
\ba\label{case3b}\BAL
&\left \{y\in\xO:|x-y|\leq C_1s^\frac{1}{N-\xa},\;d_{\partial \Omega}(x)\leq|x-y|,\;d_\xS(x)> |x-y| \right \}\\
&\quad\subset \left\{y\in\xO:\frac{1}{\CN_\xa(x,y)}\leq s,\;d_{\partial \Omega}(x)\leq|x-y|,\;d_\xS(x)> |x-y| \right \}\\
&\qquad\subset \left \{y\in\xO:|x-y|\leq C_2s^\frac{1}{N-\xa},\;d_{\partial \Omega}(x)\leq|x-y|,\;d_\xS(x)> |x-y| \right \}.
\EAL
\ea

\medskip

c)  Let $y\in\GTB(x,s)$ be such that $d_{\partial \Omega}(x)>|x-y|.$ Then $d_\xS(x)> |x-y|$ and
\bal
\frac{1}{\CN_\xa(x,y)}\approx|x-y|^{N-2}d_{\partial \Omega}(x)^2d_\xS(x)^{-\xa}.
\eal

In view of the proof of \eqref{case2c}, we may deduce the existence of positive constants $C_1,C_2$ depending only on $\xa,N$ such that
\bal
&\left \{y\in\xO:|x-y|\leq C_1(d_\xS(x)^\xa s)^\frac{1}{N},\;d_{\partial \Omega}(x)>|x-y| \right \}\\
&\quad\subset \left\{y\in\xO:\frac{1}{\CN_\xa(x,y)}\geq s,\;d_{\partial \Omega}(x)>|x-y| \right\}\\
&\qquad\subset \left\{y\in\xO:|x-y|\leq C_2(d_\xS(x)^\xa s)^\frac{1}{N},\;d_{\partial \Omega}(x)>|x-y| \right\}.
\eal

This and \eqref{case3b} imply the existence of two positive constants $\tilde C_1,\tilde C_2$ depending only on $\xa,N$ such that
\ba\label{case3c}\BAL
&\left\{y\in\xO:|x-y|\leq \tilde C_1s^\frac{1}{N-\xa},\;d_{\partial \Omega}(x)>|x-y| \right \}\\
&\quad\subset \left \{y\in\xO:\frac{1}{\CN_\xa(x,y)}\leq s,\;d_{\partial \Omega}(x)>|x-y| \right \}\\
&\qquad\subset \left\{y\in\xO:|x-y|\leq \tilde C_2s^\frac{1}{N-\xa},\;d_{\partial \Omega}(x)>|x-y| \right\}.
\EAL
\ea

Combining \eqref{case3a}-\eqref{case3c} and Lemma \ref{l2.3}, we obtain
\bal
\xo(\GTB(x,s))\approx\xo(B(x,s_3))\approx s^\frac{b+\theta+N}{N-\xa},
\eal
where $s_3=s^\frac{1}{N-\xa}.$

The proof is complete.
\end{proof}

\begin{lemma}\label{vol-2}
Let $\xa< N$, $b> 0$, $\theta>\max\{k-N-b,-b-\xa\}$ and $\dd \gw=d_{\partial \Omega}(x)^b d_\xS(x)^\theta$ $\1_{\Omega }(x)\,\dx$.  Then \eqref{2.4}  holds.
\end{lemma}
\begin{proof}[\textbf{Proof}]
Let $y\in\GTB(x,s).$ We will consider three cases.

\noindent \textbf{Case 1:} $d_{\partial \Omega}(x)\leq 2|x-y|$ and $d_\xS(x)\leq 2|x-y|.$ We can easily show that $d_{\partial \Omega}(y)\leq 3 |x-y|,$ $d_\xS(y)\leq 3|x-y|$ and
\bal
\frac{1}{\CN_\xa(x,y)}\approx|x-y|^{N-\xa}.
\eal
Therefore $|x-y|\lesssim s^{\frac{1}{N-\xa}},$ which implies that
$ d_\xS(x)+d_\xS(y)\lesssim s^{\frac{1}{N-\xa}}$.
By \eqref{xoest1}, we can easily show that
\bal
\int_0^{s}\frac{\gw\left(\GTB(x,t)\right)}{t^2} \dd t\approx\int_0^{s}\frac{\gw\left(\GTB(y,t)\right)}{t^2} \dd t\approx s^\frac{b+\theta+\xa}{N-\xa}.
\eal

\medskip

\noindent \textbf{Case 2:} $d_{\partial \Omega}(x)\leq 2|x-y|$ and $d_\xS(x)\geq 2|x-y|.$ In this case we have that $d_{\partial \Omega}(y)\leq 3 |x-y|,$ $\frac{1}{2}d_\xS(x)\leq d_\xS(y)\leq\frac{3}{2}d_\xS(x)$ and
\bal
\frac{1}{\CN_\xa(x,y)}\approx|x-y|^{N}d_\xS^{-\xa}(x)\approx|x-y|^{N}d_\xS(y)^{-\xa}.
\eal
This implies
\bal
d_{\partial \Omega}(x)^N d_\xS(x)^{-\xa}\lesssim s \quad\text{and}\quad d_{\partial \Omega}(y)^N d_\xS(y)^{-\xa}\lesssim s.
\eal
By \eqref{xoest1}, we obtain
\bal
\int_0^{s}\frac{\gw\left(\GTB(x,t)\right)}{t^2} \dd t\approx\int_0^{s}\frac{\gw\left(\GTB(y,t)\right)}{t^2} \dd t\approx\left\{
\BAL
&s^\frac{b+\theta+\xa}{N-\xa} &&\text{if}\; d_\xS(x)^{N-\xa} \leq s,\\
&s^\frac{b}{N}d_\xS(x)^{\theta+\xa\frac{b+N}{N}} &&\text{if}\; s\leq d_\xS(x)^{N-\xa}.
\EAL\right.
\eal

\medskip

\noindent \textbf{Case 3:} $d_{\partial \Omega}(x)\geq 2|x-y|$. We first note that $d_\xS(x)\geq 2|x-y|$,
\bal
\frac{1}{2}d_{\partial \Omega}(x)\leq d_{\partial \Omega}(y)\leq\frac{3}{2}d_{\partial \Omega}(x)\quad\text{and}\quad \frac{1}{2}d_\xS(x)\leq d_\xS(y)\leq\frac{3}{2}d_\xS(x).
\eal
From \eqref{xoest1}, we infer that
\bal
\int_0^{s}\frac{\gw\left(\GTB(x,t)\right)}{t^2} \dd t\approx\int_0^{s}\frac{\gw\left(\GTB(y,t)\right)}{t^2} \dd t.
\eal

Combining Cases 1--3, we derive \eqref{2.4}.
\end{proof}

By applying Proposition \ref{t2.1} with $J(x,y)=\CN_{2\am}(x,y)$, $\dd \xo= \left(d_{\partial \Omega}(x)d_{\xS}(x)^{-\am}\right)^{p+1}\dx$ and $\dd \lambda = \dd \xn$, we obtain the following result for any $\xn\in\mathfrak{M}^+(\partial\xO)$.

\begin{theorem}\label{sourceth}   Let $p$ satisfy \eqref{p-cond}. Then the following statements are equivalent.
	
	1. The equation
\bal
v=\BBN_{2\am}[|v|^p\left(d_{\partial \Omega}d_{\xS}^{-\am}\right)^{p+1}]+\ell \BBN_{2\am}[\xn]
\eal
has a positive solution for $\ell>0$ small.
	
	2. For any Borel set $E \subset \overline{\xO}$, there holds
\bal
\int_E \BBN_{2\am}[\1_E\xn]^p \left(d_{\partial \Omega}(x)d_{\xS}(x)^{-\am}\right)^{p+1}\dx \leq C\, \xn(E).
\eal
	
	3. The following inequality holds
\bal
\BBN_{2\am}[\BBN_{2\am}[\xn]^p\left(d_{\partial \Omega}d_{\xS}^{-\am}\right)^{p+1}]\leq C\BBN_{2\am}[\xn]<\infty\quad a.e.
\eal

4. For any Borel set $E \subset \overline{\xO}$ there holds
\bal
\xn(E)\leq C\, \mathrm{Cap}_{\BBN_{2\am},p'}^{p+1,-\am(p+1)}(E).
\eal

Here we implicitly extend $\xn$ to whole $\overline{\xO}$ by setting $\xn(\xO)=0.$

\end{theorem}

\subsection{Existence and nonexistence results in the case $\xm<\frac{N^2}{4}$}

We first show that Theorem \ref{subm} is a direct consequence of Theorem \ref{sourceth}.

\begin{proof}[\textbf{Proof of Theorem \ref{subm}}] We will use Theorem \ref{sourceth} and show that statements 1--4 of the present theorem are equivalent to statements 1--4 of Theorem \ref{sourceth} respectively. By \eqref{GN2a}-\eqref{105} and \cite[Proposition 2.7]{BHV}, we can easily show that  the equation \eqref{vN2a} has a solution $v$ for some $\ell>0$ if and only if equation
\ba \label{u-sigmanu} u=\BBG_\xm[u^p]+\gs\BBK_\xm[\xn]
\ea	
has a positive solution $u$ for some $\xs>0$. This and the fact that $u$ is a weak solution of \eqref{BP-} if and only if $u$ is represented by \eqref{u-sigmanu} imply that statement 1 of Theorem \ref{sourceth} is equivalent to statement 1 of the present theorem. In addition, in light of \eqref{GN2a} and \eqref{KN2a1}, we can deduce that statements 2--3 of Theorem \ref{sourceth} are equivalent to statements 2--3 of the present theorem respectively.

Therefore, it remains to prove that  statement 4 of this theorem is equivalent to statement 4 of Theorem \ref{sourceth}. It is enough to show that for any compact subset $E \subset \Sigma$, there holds
\ba \label{Cap-equi-1} \mathrm{Cap}_{\vartheta,p'}^{\xS}(E) \approx \mathrm{Cap}_{\BBN_{2\am},p'}^{p+1,-\am(p+1)}(E),
\ea
where $\vartheta$ is defined in \eqref{gamma}. In view of  \eqref{K-split1}-\eqref{21}, we may employ a similar argument as in the proof of \cite[Estimate (6.36)]{GkiNg_source} to reach the desired result.
\end{proof}

\begin{remark} \label{existSigma} By \cite[Theorems 9.2 and 9.4]{BGT}, the following statements are valid.

(i) If $1<p<\min\{\frac{N+1}{N-1},\frac{N-\am+1}{N-\am-1}\}$ then
\ba\label{lpineq}
\int_\xO \BBK_\xm[|\xn|]^p \ei \dx \leq C(\xO,\xS,\xm,p)\,|\xn|(\xO)^p, \quad \forall \xn\in\mathfrak{M}(\partial\xO).
\ea

(ii)  If $1<p<\frac{N+1}{N-1}$  and $ \xn\in\mathfrak{M}(\partial\xO)$ has compact support in $\partial\xO\setminus\xS$ then \eqref{lpineq} holds true.

(iii)  If $1<p<\frac{N-\am+1}{N-\am-1}$  and $ \xn\in\mathfrak{M}(\partial\xO)$ has compact support in $\xS$ then \eqref{lpineq} holds true.

Hence, if one of the above cases occurs, we see that statement 2 of Theorem \ref{sourceth} holds true, which implies the existence of solution of \eqref{power} for some $\xs>0.$
\end{remark}

\begin{remark} \label{nonexistSigma}
Assume $\mu < \frac{N^2}{4}$ and $p \geq \frac{N-\am+1}{N-\am-1}$. Then for any $z \in \Sigma$ and any $\sigma>0$, problem \eqref{u-sigmanu} with $\nu=\delta_z$ does not admit any positive weak solution. Indeed, suppose by contradiction that for some $z \in \Sigma$ and $\sigma>0$, there exists a positive solution $u \in L^p(\Omega;\ei)$ of equation \eqref{u-sigmanu}. Without loss of generality, we can assume that $z =0 \in \Sigma$ and $\sigma=1$. From \eqref{u-sigmanu}, $u(x) \geq \BBK_{\mu}[\delta_0] (x)= K_\mu(x,0)$ for a.e. $x \in \Omega$. Let $\CC$ be a cone of vertex $0$ such that $\CC \subset \Omega $ and there exist $r>0$, $0<\ell<1$ satisfying for any $x \in \CC$, $|x|<r$ and $d_\Sigma(x)\geq d_{\partial \Omega}(x)> \ell |x|$. Then, by \eqref{Martinest1} and \eqref{eigenfunctionestimates},
\bal
\int_{\Omega } u(x)^p \ei(x) \dd x \geq \int_{\CC} K_\mu(x,0)^p \ei(x) \, \dd x
 &\geq \int_{\CC} |x|^{1-\am-(N-\am-1)p} \, \dd x\\
 &\approx \int_0^r t^{N-\am - (N-\am-1)p} \, \dd t.
\eal
Since $p \geq \frac{N-\am+1}{N-\am-1}$, the last integral is divergent, hence $u \not \in L^p(\Omega;\ei)$, which leads to a contradiction.
\end{remark}

\begin{remark} \label{nonexistpboundary}
Assume $\mu < \frac{N^2}{4}$ and $p \geq \frac{N+1}{N-1}$. Proceeding as in Remark \ref{nonexistpboundary}, we may show that  any $z \in \partial\xO\setminus\Sigma$ and any $\sigma>0$, problem \eqref{u-sigmanu} with $\nu=\delta_z$ does not admit any positive weak solution.
\end{remark}

By using the above capacities and Theorem \ref{sourceth}, we are able to prove Theorem \ref{th:existnu-prtO}.

\begin{proof}[\textbf{Proof of Theorem \ref{th:existnu-prtO} when $\xm<\frac{N^2}{4}$}.] The fact that statements 1--3 are equivalent follows by using a similar argument as in the proof of Theorem \ref{subm}. Hence, it remains to show that statement 4 is equivalent to statements 1--3. Since
\bal
K_\mu(x,z)\approx C\,\dist(z,\xS)d_{\partial \Omega}(x)d_{\xS}(x)^{-\am} |x-z|^{-N}, \quad \forall x \in \Omega , \,  z \in \partial \Omega\setminus\xS,
\eal
we may proceed as in the proof of \cite[estimate (6.40)]{GkiNg_source} to obtain the desired result.
\end{proof}

When $p \geq \frac{\ap+1}{\ap-1}$, the nonexistence occurs, as shown in the following remark.

\begin{remark} \label{rem3}
We additionally assume that $\xO$ is $C^3.$ If $p\geq \frac{\ap+1}{\ap-1}$ then, for any measure $\xn \in \GTM^+(\partial \Omega )$ with compact support in $\xS$ and any $\sigma>0$, there is no solution of problem \eqref{u-sigmanu}. Indeed, it can be proved by contradiction. Suppose that we can find $\sigma>0$ and a measure $\xn \in \GTM^+(\partial \Omega )$ with compact support in $\xS$ such that there exists a solution $0\leq u \in L^p(\xO;\ei)$ of \eqref{u-sigmanu}.
It follows that $\BBK_\xm[\xn]\in L^p(\xO;\ei)$. Therefore, by Proposition \ref{subcr} there is a unique nontrivial nonnegative solution $v$ of
\bal
\left\{ \BAL
- L_\gm v + |v|^{p-1}v &=0\qquad \text{in }\;\Gw,\\
\tr(v)&=\xn.\EAL\right.
\eal

Moreover, $v \leq \BBK_{\mu}[\nu]$ in $\Omega$. This, together with Proposition \ref{Martin} and the fact that $\nu$ has compact support in $\Sigma$, implies
$ v(x) \leq \BBK_{\mu}[\nu](x)  \lesssim d_{\partial \Omega}(x)\nu(\Sigma)$ for $x$ near $\partial \Omega\setminus\xS$. Therefore, by Theorem \ref{remov-1}, we have that $v \equiv 0$, which leads to a contradiction.

\end{remark}

\begin{remark} \label{rem4}
If $\am>1$ and $p\geq \frac{\am+1}{\am-1}$ then for any measure $\xn \in \GTM^+(\partial \Omega)$ with compact support in $\partial\xO$ and any $\sigma>0$, there is no solution of \eqref{u-sigmanu}. Indeed, it can be proved by contradiction. Suppose that we can find a measure $\xn \in \GTM^+(\partial \Omega )$ with compact support in $\partial \Omega\setminus\xS$ and $\sigma>0$ such that there exists a solution $0\leq u \in L^p(\xO;\ei)$ of \eqref{u-sigmanu}. Then by Theorem \ref{th:existnu-prtO}, estimate \eqref{GKp<K} holds for some constant $C>0$.

For simplicity, we assume that $0\in \Sigma$. Let $\{x_n\} \subset \xO$ be such that $\dist(x_n,\supp \xn)>\xe>0$ for any $n\in\BBN$ and $x_n\to 0$ as $n \to \infty$. Then there exists
a positive constant $C_1=C_1(\xe,\xO,\xS,\xm)$ such that
\bal
\begin{aligned}
C&\geq\frac{\int_{\xO} G_\xm(x_n,y)\BBK_\xm[\xn](y)^p \dy}{\BBK_\xm[\xn](x_n)}\\
&\gtrsim C_1 d_{\partial \Omega}(x_n)^{-1}d_\xS(x_n)^{\am}\xn(\partial \Omega)^{p-1} \int_{\xO} (d_{\partial \Omega}(y)d_\xS(y)^{-\am})^{p} G_\mu(x_n,y)\dy.
\end{aligned}
\eal

Set
\bal
F(x_n,y):=d_{\partial \Omega}(x_n)^{-1}d_\xS(x_n)^{\am}(d_{\partial \Omega}(y)d_\xS(y)^{-\am})^{p} G_\mu(x_n,y).
\eal
Then
\bal
\liminf_{n\to\infty} F(x_n,y)\gtrsim d_{\partial \Omega}(y)^{p+1}|y|^{2\am-N-\am(p+1)}.
\eal
Let $\CC$ be a cone of vertex $0$ such that $\CC \subset \Omega $ and there exist $r>0$, $0<\ell<1$ satisfying for any $x \in \CC$, $|x|<r$ and $d_\Sigma(x)\geq d_{\partial \Omega}(x)> \ell |x|$. Combining all above we have that
\bal
C\gtrsim \int_{\mathcal{C}} d_{\partial \Omega}(y)^{p+1}|y|^{2\am-N-\am(p+1)} \dd y&\gtrsim\int_{\mathcal{C}}|y|^{p+1+2\am-N-\am(p+1)} \dd y\\
&\approx\int_0^{r}s^{p(1-\am)+\am}\dd s=+\infty,
\eal
since $p\geq \frac{\am+1}{\am-1}.$ This is clearly a contradiction.
\end{remark}

\subsection{Existence results in the case $\xS=\{0\}$ and $\xm= \frac{N^2}{4}$}
Let $0<\xe<N$. For any $(x,y)\in\overline{\xO}\times\overline{\xO}$ such that $x \neq y,$ we set
\bal
\CN_{1,\xe}(x,y):=\frac{\max\{|x-y|,|x|,|y|\}^{N}}{|x-y|^{N-2}\max\{|x-y|,d_{\partial \Omega}(x),d_{\partial \Omega}(y)\}^2}
+\max\{|x-y|,d_{\partial \Omega}(x),d_{\partial \Omega}(y)\}^{-\xe}	
\eal
 and
\bal
\CN_{N-\xe}(x,y):=\frac{\max\{|x-y|,|x|,|y|\}^{N-\xe}}{|x-y|^{N-2}\max\{|x-y|,d_{\partial \Omega}(x),d_{\partial \Omega}(y)\}^2}.
\eal

Put
\bal \begin{aligned}
G_{H^2,\xe}(x,y) &:= |x-y|^{2-N} \left(1 \wedge \frac{d_{\partial \Omega}(x)d_{\partial \Omega}(y)}{|x-y|^2}\right) \left(1 \wedge \frac{|x||y|}{|x-y|^2} \right)^{-\frac{N}{2}} \\
&+\frac{d_{\partial \Omega}(x)d_{\partial \Omega}(y)}{(|x||y|)^{\frac{N}{2}}}  \max\{|x-y|,d_{\partial \Omega}(x),d_{\partial \Omega}(y)\}^{-\xe}, \quad x,y \in \Omega \setminus \{0\}, \, x \neq y
\end{aligned} \eal
and
\ba \label{tilG}
\tilde G_{H^2,\xe}(x,y):=d_{\partial \Omega}(x)d_{\partial \Omega}(y)(|x||y|)^{-\frac{N}{2}} \CN_{N-\xe}(x,y), \quad \forall x,y \in \Omega \setminus \{0\}, \, x \neq y.
\ea
Note that
\bal \begin{aligned}
\left|\ln\left(\min\left\{|x-y|^{-2},(d_{\partial \Omega}(x)d_{\partial \Omega}(y))^{-1}\right\}\right)\right|
&\leq C(\Omega,\xe) \max\{|x-y|,d_{\partial \Omega}(x),d_{\partial \Omega}(y)\}^{-\xe},
\end{aligned}
\eal
which, together with \eqref{Greenestb}, implies
\ba \label{GGe}
G_{H^2}(x,y)\lesssim G_{H^2,\xe}(x,y), \quad \forall x,y \in \Omega, \, x \neq y.
\ea

Next, from the estimates
\bal
G_{H^2,\varepsilon}(x,y) \approx d_{\partial \Omega}(x)d_{\partial \Omega}(y)(|x||y|)^{-\frac{N}{2}} \CN_{1,\varepsilon}(x,y), \quad x,y \in \Omega, \, x \neq y,
\eal
\bal
\CN_{1,\xe}(x,y)\leq C(\xe,\xO) \CN_{N-\xe}(x,y),\quad x,y \in \Omega, \, x \neq y,
\eal
we obtain
\ba \label{GGe1}
G_{H^2,\xe}(x,y)\lesssim \tilde G_{H^2,\xe}(x,y), \quad \forall x,y \in \Omega, \, x \neq y.
\ea

Set
\bal
\tilde\BBG_{H^2,\xe}[|u|^p](x):=\int_{\xO} \tilde G_{H^2,\xe}(x,y) |u|^p \, \dd y, \\ \BBN_{N-\varepsilon}[\tau](x):=\int_{\xO} \CN_{N-\varepsilon}(x,y) \,\dd\tau(y).
\eal
Proceeding as in the proof of Theorem \ref{sourceth}, we obtain the following result

\begin{theorem}\label{theoremint2}
 Let $0<\xe<2,$ $1<p<\frac{N+2-2\xe}{N-2}$ and $\xn \in \GTM^+(\partial\xO)$. Then the following statements are equivalent.
	
	1. The equation
	\be \label{eq:uH2e} u=\tilde\BBG_{H^2,\xe}[u^p]+\xs d_{\partial \Omega}|\cdot|^{-\frac{N}{2}}\BBN_{N-\varepsilon}[\xn]
	\ee
	has a positive solution for $\xs>0$ small.
	
	2. For any Borel set $E \subset \overline{\xO}$, there holds
	\bal
\int_E \BBN_{N-\varepsilon}[\1_E\xn]^p(x) \phi_{H^2,\Sigma}(x)^{p+1}\, \dx \leq C\xn(E).
\eal
	
	3. The following inequality holds
	\bal
\BBN_{N-\varepsilon}[\BBN_{N-\varepsilon}[\xn]^p\left(d_{\partial \Omega}d_{\xS}^{-\frac{N}{2}}\right)^{p+1}]\leq C\,\BBN_{N-\varepsilon}[\xn]<\infty\quad a.e.
\eal

	4. For any Borel set $E \subset \overline{\xO}$ there holds
\bal
\xn(E)\leq C\, \mathrm{Cap}_{\BBN_{N-\xe},p'}^{p+1,-\frac{N}{2}(p+1)}(E).
\eal
\end{theorem}

\begin{theorem}\label{singl}
We assume that at least one of the statements 1--4 of Theorem \ref{theoremint2} is valid. Then problem \eqref{BP-} with $\mu=H^2$ admits a positive weak solution for $\xs>0$ small.
\end{theorem}
\begin{proof}[\textbf{Proof}]
 By Theorem \ref{theoremint2}, there exists a solution $u$ to equation \eqref{eq:uH2e} for $\xs>0$ small.  By \eqref{GGe} and \eqref{GGe1}, we have
 $u\gtrsim \BBG_{H^2}[u^p]+\xs\BBK_{H^2}[\xn]$.
By \cite[Proposition 2.7]{BHV}, we deduce that equation
\ba \label{eq:uH2}
u=\BBG_{H^2}[u^p]+\xs\BBK_{H^2}[\xn].
\ea
has a solution for $\xs>0$ small. This means that admits a positive weak solution for $\xs>0$.
\end{proof}

\begin{remark}\label{hardycrit} Let $\xS=\{0\}\subset\partial\xO,$ $\xm=\frac{N^2}{4},$  $\xn=\xd_0$ and $1<p<\frac{N+2}{N-2}.$ Then there exists $\xe>0$ small enough such that $1<p<\frac{N+2}{N-2+2\xe}\leq \frac{N+2-2\xe}{N-2}.$ In addition, we have
\bal
\int_\xO \BBN_{N-\varepsilon}[\xd_0]^p(x) \phi_{H^2,\Sigma}(x)^{p+1}\, \dx \lesssim\int_\xO |x|^{(p+1)(1-\frac{N}{2})-p\xe}\, \dx<\infty.
\eal
Hence statement 2 of Theorem \ref{theoremint2} is satisfied. This and Theorem \ref{singl} imply that equation \eqref{eq:uH2} has a solution for $\xs>0$ small.
\end{remark}

\begin{proof}[\textbf{Proof of Theorem \ref{th:existnu-prtO} when $\xS=\{0\}$ and $\xm=\frac{N^2}{4}$}.]
Let $\xe>0$ be small enough such that $1<p<\frac{N+2-2\xe}{N-2}.$ Let $K=\supp(\xn)\Subset \partial\xO\setminus\{0\}$ and $\tilde \xb=\frac{1}{2}\dist(K,\{0\})>0$. By  \eqref{Martinest2}, we can easily show that
\ba\label{62}
\BBK_{H^2,\xe}[\xn]:=d_{\partial \Omega}d_{\xS}^{-\frac{N}{2}}\BBN_{N-\xe}[\xn]\approx \BBK_{\frac{N^2}{4}}[\xn].
\ea
Hence, by Proposition \ref{equivba}, \eqref{62}, Theorems \ref{theoremint2} and \ref{singl}, it is enough to show that
\ba\label{63}
\tilde\BBG_{H^2,\xe}[\BBK_{H^2,\xe}[\xn]^p]\approx \BBG_{H^2}[\BBK_{H^2}[\xn]^p] \quad \text{in } \Omega .
\ea
By \eqref{GGe} and \eqref{GGe1}, it is sufficient to show that
 \ba\label{64}
\tilde\BBG_{H^2,\xe}[\BBK_{H^2}[\xn]^p]\lesssim \BBG_{H^2}[\BBK_{H^2}[\xn]^p] \quad  \text{in } \Omega .
\ea

Indeed, on one hand, since $1<p<\frac{N+2-2\xe}{N-2}$, for any $x\in \xO$ there holds
\ba\label{65}\BAL
&\int_{B(0,\frac{\tilde \xb}{4})\cap\xO}\tilde G_{H^2,\xe}(x,y)\BBK_{H^2}[\xn](y)^p \dd y\approx \xn(K)^p \int_{B(0,\frac{\tilde \xb}{4})\cap\xO} \tilde G_{H^2,\xe}(x,y)d_{\partial \Omega}(y)^p|y|^{-\frac{p N}{2}} \dd y\\
&\lesssim  \xn(K)^p d_{\partial \Omega}(x)|x|^{-\frac{N}{2}}\int_{B(0,\frac{\tilde \xb}{4})\cap\xO} |x-y|^{-\xe}(d_{\partial \Omega}(y)|y|^{-\frac{N}{2}})^{p+1} \dd y\\
& \quad +\xn(K)^p d_{\partial \Omega}(x)|x|^{-\frac{N}{2}}\int_{B(0,\frac{\tilde \xb}{4})\cap\xO} |x-y|^{-N+1}d_{\partial \Omega}(y)^p|y|^{N-\xe-\frac{(p+1)N}{2}} \dd y \\
& \lesssim\xn(K)^p d_{\partial \Omega}(x)|x|^{-\frac{N}{2}}.
\EAL
\ea
The implicit constants in the above inequalities depend only on $\xO,K,\tilde \beta,p,\xe$.

On the other hand, we have
\ba\label{66}\BAL
\int_{B(0,\frac{\tilde \xb}{4})\cap\xO} G_{H^2}(x,y)\BBK_{H^2}[\xn](y)^p \dd y&\gtrsim\xn(K)^p d_{\partial \Omega}(x)|x|^{-\frac{N}{2}}\int_{B(0,\frac{\tilde \xb}{4})\cap\xO} d_{\partial \Omega}(y)^{p+1} \dd y \\
&\gtrsim\xn(K)^p d_{\partial \Omega}(x)|x|^{-\frac{N}{2}},
\EAL
\ea
where the implicit constants in the above inequalities depend only on $\xO,K,\tilde \beta,p$.

Hence by \eqref{65} and \eqref{66}, we have that
\ba\label{67}
\int_{B(0,\frac{\tilde \xb}{4})\cap\xO} \tilde G_{H^2,\xe}(x,y)\BBK_{H^2}[\xn](y)^p \dd y\lesssim\int_{B(0,\frac{\tilde \xb}{4})\cap\xO} G_{H^2}(x,y)\BBK_{H^2}[\xn](y)^p \dd y \quad\forall x\in \xO.
\ea

Next, by \eqref{Greenestb} and \eqref{tilG}, for any $x \in \Omega $ and $y \in \Omega \setminus B(0,\frac{\tilde \beta}{4})$ we have
\bal
\tilde G_{H^2,\xe}(x,y) \approx d_{\partial \Omega}(x)d_{\partial \Omega}(y)(|x||y|)^{-\frac{N-2}{2}} \CN_{N}(x,y) \lesssim G_{H^2}(x,y).
\eal
This and \eqref{62} yield
\ba\label{68}
\int_{\xO\setminus B(0,\frac{\tilde \xb}{4})} \tilde G_{H^2,\xe}(x,y) \BBK_{H^2}[\xn](y)^p \, \dd y\lesssim\int_{\xO\setminus B(0,\frac{\tilde \xb}{4})} G_{H^2}(x,y)\BBK_{H^2}[\xn](y)^p \, \dd y \quad\forall x\in \xO.
\ea

Combining \eqref{67} and \eqref{68}, we deduce \eqref{64}. The proof is complete.
\end{proof}

\begin{remark} \label{partO-N}
	If $p<\frac{N+1}{N-1}$, by using a similar argument to the one in Remark \ref{hardycrit}, we obtain that statement 2 of Theorem \ref{theoremint2} holds true. Consequently, under the assumptions of Theorem \ref{th:existnu-prtO}, equation \eqref{u-sigmanu} has a positive solution for $\gs>0$ small.
\end{remark}


\appendix\section{Barriers} \label{app:A}
\setcounter{equation}{0}

Let $\xO$ be a $C^3$ open bounded domain. Then there exists $\xb_0>0$ depending on $C^3$ characteristic of $\xO$ such that for any $x\in \xO_{\xb_0}=\{x\in\xO:\;\;d_{\partial \Omega}(x)<\xb_0\}$ the followings hold.

(i) There exists a unique $\xs(x)\in\partial\xO$ such that
\bal d_{\partial \Omega}(x)=|x-\xs(x)|,  \quad
\xs(x)=x-d_{\partial \Omega}(x)\nabla d_{\partial \Omega}(x) \quad \text{and} \quad
\nabla d_{\partial \Omega}(x)=\frac{x-\xs(x)}{|x-\xs(x)|}.
\eal

(ii) $\xs(x)\in C^2(\overline{\xO}_{\xb_0})$ and $d_{\partial \Omega}\in C^3(\overline{\xO}_{\xb_0}).$

(iii) For any $i={1,...,N}$ there holds
\bal
|\nabla\xs_{i}(x)\cdot\nabla d_{\partial \Omega}(x)|\leq \norm{D^2 d_{\partial \Omega}}_{L^\infty(\overline{\xO}_{\xb_0})}d_{\partial \Omega}(x).
\eal

For any $(x,z)\in \overline{\xO}_{\xb_0}\times\partial\xO,$ set
\bal
d_z(x)=\sqrt{d_{\partial \Omega}^2(x)+|\xs(x)-z|^2}.
\eal

Then
\bal
\frac{1}{2}|x-z|\leq d_z(x)\leq\sqrt{5}|x-z|.
\eal

Finally, for any $0<R\leq\xb_0,$ we set
\bal
\mathscr{B}(z,R):=\{x\in\overline{\xO}_{\xb_0}:\;d_z(x)<R\}
\eal
Let $\xb_1$ be the constant in Proposition \ref{expansion}. In addition, we choose $\xb_1$ small enough such that
\bal
\frac{1}{2}d_\xS(x)\leq \tilde d_\xS(x)\leq 2d_\xS(x)\quad\text{in}\;\xO\cap \xS_{\xb_1}.
\eal
\begin{proposition}\label{barr} Let $\xb_2=\frac{1}{16}\min\{\xb_1,\xb_0\},$ $R_0\in (0,\xb_2]$ and $0<R\leq R_0.$ For any $z\in \overline{\Sigma}_{R_0}\cap\partial\xO,$ there is a supersolution $w:=w_{R,z}$ of \eqref{eq:power1} in $\mathscr{B}(z,R)$ such that
\bal &w \in C(\Gw\cap B(z,R)), \quad  \lim_{x\in\xO\cap\mathscr{B}(z,R),\;x\rightarrow\xi}\frac{w(x)}{\tilde W(x)}=0\quad \text{for any } \;\;\xi\in\partial\xO\cap \mathscr{B}(z,R),\\
&w(x)\to \infty \text{ as } \dist (x,F)\to 0, \text{ for any compact subset } F\subset \xO\cap \prt \mathscr{B}(z,R).
\eal
More precisely, for $\gamma \in (\am,\ap)$, $w$ can be constructed as
	\be \label{BAR1}
	w(x)= \left\{\BA{lll} \Gl(R^2- d_{z}(x)^2)^{-b}e^{Md_{\partial \Omega}(x)}d_{\partial \Omega}(x)\tilde{d}_\Sigma(x)^{-\gamma} \quad & \text{ if }\;\mu<H^2, \\[2mm]
	\Gl(R^2- d_{z}(x)^2)^{-b}e^{Md_{\partial \Omega}(x)}d_{\partial \Omega}(x)\tilde{d}_\Sigma(x)^{-H}\sqrt{|\ln \frac{\tilde d_\Sigma(x)}{16R_0}|} \;&\text{ if }\;\xm=H^2,
	\EA\right.
	\ee
	with $M<0$ depending only on the $C^2$ characteristic of $\partial\xO,$ $b> \frac{2(p+1)-2(p-1)\min\{\xg,0\}}{p-1}$ and $\Gl>0$ large enough depending only on $\xg,N,b,p,M,R_0,$ the $C^2$ characteristic of $\Sigma$ and the $C^3$ characteristic of $\partial\xO.$
\end{proposition} 
\begin{proof}[\textbf{Proof}] Without loss of generality, we assume $z=0\in \overline{\Sigma}_{R_0}\cap\partial\xO.$
	
\noindent  \textbf{Case  1: $\xm<H^2$.} Set
	\bal w(x):=\Gl(R^2-d_0^2(x))^{-b} d_{\partial \Omega}(x)e^{Md_{\partial \Omega}(x)}\tilde{d}_\Sigma(x)^{-\gamma} \quad \text{for } x \in \xO\cap\mathscr{B}(0,R),
	\eal
	where $\gamma>0, b$ and $\Gl>0$ will be determined later on.

Then, by straightforward computation and using Proposition \ref{expansion}, we obtain

	\be\label{F3}-L_{\xm }w+w^p=\Gl(R^2-d_0(x)^2)^{-b-2}d_{\partial \Omega}(x) e^{Md_{\partial \Omega}(x)}\tilde{d}^{-\gamma-2}_\Sigma(x)(I_1+I_2+I_3+I_4),
	\ee
	where
\bal
I_1&:=-\tilde{d}_\xS^2 \left(4b(b+1)|\nabla d_0|^2d_0^2+2b(R^2-d_0^2)(|\nabla d_0|^2+d_0\xD d_0)\right), \\
I_2&:=-(R^2-d_0^2)^{2}\bigg(\xg^2-\xg(N-k)+\xm+\xg(\xg+1)f_2-\xg f_3+\xm f_1-2\xg Md_{\partial \Omega}\bigg),\\
I_3&:=-(R^2-d_0^2)^{2}\tilde{d}^2_\xS d_{\partial \Omega}^{-1}\bigg(\xD d_{\partial \Omega}(1+Md_{\partial \Omega})+2M+M^2d_{\partial \Omega}\bigg), \\
I_4&:=-4b(R^2-d_0^2)\frac{d_0}{d_{\partial \Omega}}(\tilde{d}_\xS^2\nabla d_0 \nabla d_{\partial \Omega}(1+Md_{\partial \Omega})-\xg \tilde{d}_\xS d_{\partial \Omega}\nabla d_0\cdot\nabla \tilde{d}_\xS),\\
I_5&:=\xL^{p-1}(R^2-d_0^2)^{-(p-1)b+2}e^{M(p-1)d_{\partial \Omega}}d_{\partial \Omega}^{p-1}\tilde{d}_\Sigma^{-(p-1)\gamma+2}.
\eal

By (i)--(iii), we have
\ba\label{F3-0a}
|I_1|\leq C_1(R_0,b,\xO,N)\tilde{d}^2_\Sigma.
\ea

Also,
	\be \label{F3-0b} \begin{aligned}
	  |I_4| \leq C_2(R_0,\xO,N,M,\xg,b)(R^2-d_0^2)\tilde{d}_\Sigma.
	\end{aligned} \ee
Next we choose $\gamma \in (\am,\ap),$ then $\xg^2-\xg(N-k)+\xm <0$. In addition, there exist $0<\ge_0<1,$ $\delta_0>0$ and $M<0$ such that if $\tilde d_\Sigma\leq\gd_0$ then
\bal
\xD d_{\partial \Omega}(1+Md_{\partial \Omega})+2M+M^2d_{\partial \Omega} <-\ge_0
\eal
and by \eqref{frag},
\bal
&\xg^2-\xg(N-k)+\xm+\xg(\xg+1)f_2-\xg f_3+\xm f_1 -2\xg Md_{\partial \Omega}<-\xe_0.\\
\eal
It follows that if $\tilde d_\Sigma\leq\gd_0$ then
\be \label{F3-000} I_2 \geq \epsilon_0(R^2-d_0^2)^2.
\ee
We set
	\bal &\CA_1:=\left\{x\in \xO\cap\mathscr{B}(0,R):\tilde d_\Sigma(x)\leq c_1(R^2-d_0(x)^2)\right\} \quad \text{where } c_1=\frac{\epsilon_0}{4\max\{\sqrt{C_1},C_2\}},\\
	& \CA_2:=\left\{x\in\xO\cap\mathscr{B}(0,R):d_\Sigma(x)\leq \gd_0^{\phantom{^4}}\right\}, \quad
	\CA_3:=\{x\in\xO\cap\mathscr{B}(0,R):\tilde d_\Sigma(x)\geq\gd_0\}.
	\eal
	
	In  $\CA_1 \cap \CA_2$, by \eqref{F3-0a}, \eqref{F3-0b} and \eqref{F3-000}, we have
	\be \label{F3-1}
	I_1 + I_2 + I_3 +I_4\geq \frac{\xe_0(R^2-d_0^2)^2}{2}.
	\ee
	
	In $\CA_1^c\cap \CA_2$, we have $\tilde d_\Sigma \geq c_1(R^2-d_0^2)$.  If $d_{\partial \Omega}(x)\leq c_2(R^2-d_0(x)^2)^2,$ where
\bal
c_2=\min\left\{\frac{\xe_0}{3C_1},\frac{\xe_0^2c_1^2}{9C_2^2}\right\},
\eal
then we can show that
\bal
I_3\geq c_2^{-1}\frac{\xe_0}{2}\tilde d_\xS^2+ c_2^{-\frac{1}{2}}c_1\frac{\xe_0}{2}\tilde d_\xS(R^2-d_0^2),
\eal
This, together with \eqref{F3-0a} and \eqref{F3-0b}, implies \eqref{F3-1}. If $d_{\partial \Omega}(x)\geq c_2(R^2-d_0(x)^2)^2,$ then by Proposition \ref{expansion}, $\tilde d_\xS(x)\geq c_2c_3(\xb_1,\xS)(R^2-d_0(x)^2)^2.$ Therefore,
\bal
I_5\geq c_4(R_0,M,p,\xg,c_1,c_2,c_3)\xL^{p-1}(R^2-d_0^2)^{-(p-1)b+2+2(p-1)-2(p-1)\min\{\xg,0\}+2}d_\xS.
\eal
If  we choose $b>\frac{2(p+1)-2(p-1)\min\{\xg,0\}}{p-1}=:b_0$, then there exists $\Gl$ large enough depending on $c_4,R_0,b,p,\xg$ such that
	\be\label{F4} \begin{aligned}
	&I_5 \geq I_1+I_4.
	\end{aligned} \ee
This and \eqref{F4} yield
	\be \label{F4-0}
	I_1 +I_2+ I_3 + I_4+I_5 \geq 0.
	\ee
	
Similarly we may show that \eqref{F4-0} is valid in $\CA_3$ for some positive constant $\xL$ depending on $M,R_0,b,p,\xg,\xO,\xS.$

Combining the above estimates, we deduce that for $\gamma \in (\am,\ap)$, $b>b_0$ and $\Gl>0$ large enough, there holds
	\be \label{F11}
	-L_{\xm }w + w^p\geq 0\qquad\text{in }\; \xO\cap\mathscr{B}(0,R).
	\ee

\noindent  \textbf{Case 2: $\xm=H^2$.} First we note that $\frac{\tilde d_\Sigma}{16R_0}\leq\frac{1}{2}$ in  $\xO\cap\mathscr{B}(0,R).$ Set
	\bal w(x):=\Gl(R^2- d_0(x)^2)^{-b}d_{\partial \Omega}(x)e^{Md_{\partial \Omega}(x)}\tilde{d}_\Sigma(x)^{-H}\bigg(-\ln \frac{\tilde d_\Sigma(x)}{16R_0}\bigg)^{\frac{1}{2}} \quad \text{for } x \in \xO\cap\mathscr{B}(0,R),
	\eal
	where $\gamma>0, b$ and $\Gl>0$ will be determined later on. Then, by straightforward calculations we have
\be \label{F12}
	-L_{\xm }w +w^p =\Gl(R^2-d_0^2)^{-b-2}\tilde d^{-H-2}_\Sigma \left(-\ln \frac{\tilde d_\Sigma}{16R_0}\right)^{-\frac{3}{2}}(\tilde I_1 + \tilde I_2 + \tilde I_3 + \tilde I_4),
	\ee
	where
	\bal
	\tilde I_1&:=-\tilde{d}_\xS^2\left(\ln \frac{\tilde d_\Sigma}{16R_0}\right)^{2}\left(4b(b+1)|\nabla d_0|^2d_0^2+2b(R^2-d_0^2)(|\nabla d_0|^2+d_0\xD d_0)\right),\\
\tilde I_2&:=-(R^2-d_0^2)^{2}\bigg(\frac{1}{2}(-\ln \frac{\tilde d_\Sigma}{16R_0})(f_3-(N-k-1)f_2)-\frac{1}{4}\\
&\qquad\qquad+\left(\ln \frac{\tilde d_\Sigma}{16R_0}\right)^{2}\left(H(H+1)f_2-H f_3+H^2 f_1-2H Md_{\partial \Omega}\right)+M(-d_{\partial \Omega}\ln \frac{\tilde d_\Sigma}{16R_0})\bigg),\\
\tilde I_3&:=-(R^2-d_0^2)^{2}\bigg(\ln \frac{\tilde d_\Sigma}{16R_0}\bigg)^{2}\tilde{d}^2_\xS d_{\partial \Omega}^{-1}\bigg(\xD d_{\partial \Omega}(1+Md_{\partial \Omega})+2M+M^2d_{\partial \Omega}\bigg),\\
\tilde I_4&:=-4b(R^2-d_0^2)\frac{d_0}{d_{\partial \Omega}}\bigg(-\ln \frac{\tilde d_\Sigma}{16R_0}\bigg)\Bigg(\bigg(-\ln \frac{\tilde d_\Sigma}{16R_0}\bigg)\tilde{d}_\xS^2\nabla d_0\nabla d_{\partial \Omega}(1+Md_{\partial \Omega})\\
&\qquad\qquad-H \bigg(-\ln \frac{\tilde d_\Sigma}{16R_0}\bigg)\tilde{d}_\xS d_{\partial \Omega}\nabla d_0\cdot\nabla \tilde{d}_\xS
-\frac{1}{2}\tilde d_\xS\nabla d_0\cdot\nabla  \tilde{d}_\xS\Bigg),\\
\tilde I_5&:=\xL^{p-1}(R^2-d_0^2)^{-(p-1)b+2}\bigg(\ln \frac{\tilde d_\Sigma}{16R_0}\bigg)^{\frac{p-1}{2}+2}e^{M(p-1)d_{\partial \Omega}}d_{\partial \Omega}^{p-1}\tilde{d}_\Sigma^{-(p-1)H+2}.
\eal

By (i)--(iii) and the fact that $-\ln \frac{\tilde d_\Sigma}{16R_0}\geq \ln 2$ , we have
\ba\label{F3-0aa}
|\tilde I_1|\leq \tilde C_1(R_0,b,\xO,N)\tilde{d}^2_\Sigma\left(\ln \frac{\tilde d_\Sigma}{16R_0}\right)^{2}.
\ea

Also,
	\be \label{F3-0bb} \begin{aligned}
	  |\tilde I_4| \leq \tilde C_2(R_0,\xO,N,M,k,b)(R^2-d_0^2)\tilde{d}_\Sigma\bigg|\ln \frac{\tilde d_\Sigma}{16R_0}\bigg|.
	\end{aligned} \ee
Next we choose $\delta_0>0$ and $M<0$ such that if $\tilde d_\Sigma\leq\gd_0$ then
\bal
\tilde I_3 >\xe_0(R^2-d_0^2)^{2}\left(\ln \frac{\tilde d_\Sigma}{16R_0}\right)^{2}\tilde{d}^2_\xS d_{\partial \Omega}^{-1}
\eal
and
\ba\label{F3-000b}
&\tilde I_2>\xe_0(R^2-d_0^2)^{2}.
\ea

We set
	\bal &\tilde\CA_1:=\left\{x\in \xO\cap\mathscr{B}(0,R):d_\Sigma(x)\leq \tilde c_1\frac{(R^2-d_0(x)^2)}{|\ln \frac{\tilde d_\Sigma(x)}{16R_0}|}\right\} \quad \text{where } \tilde c_1=\frac{\epsilon_0}{4\max\{\sqrt{\tilde C_2},\tilde C_3\}},\\
	& \tilde\CA_2:=\left\{x\in\xO\cap\mathscr{B}(0,R):d_\Sigma(x)\leq \gd_0^{\phantom{^4}}\right\}, \quad
	\tilde\CA_3:=\{x\in\xO\cap\mathscr{B}(0,R):d_\Sigma(x)\geq\gd_0\}.
	\eal
	
	In  $\tilde\CA_1 \cap \tilde\CA_2$, by \eqref{F3-0aa}, \eqref{F3-0bb} and \eqref{F3-000b}, we have
	\be \label{F3-1b}
	\tilde I_1 +\tilde I_2 +\tilde I_3 +\tilde I_4\geq \frac{\xe_0(R^2-d_0^2)^2}{2}.
	\ee
	
In $\tilde \CA_1^c\cap \tilde\CA_2$, we have $\tilde d_\Sigma \geq \tilde c_1\frac{(R^2-d_0^2)}{|\ln \frac{\tilde d_\Sigma}{16R_0}|}$. If $d_{\partial \Omega}(x)\leq \tilde c_2(R^2-d_0(x)^2)^2,$ where
\bal
\tilde c_2=\min\left\{\frac{\xe_0}{3\tilde C_1},\frac{\xe_0^2{\tilde c}_1^2}{9{\tilde C}_2^2}\right\}.
\eal
Then, we can easily show that  \eqref{F3-1b} is valid. The rest of
the proof is the same as in case 1 with obvious modifications so we omit it. The proof is complete.	
\end{proof}


\begin{thebibliography}{99}

\bibitem{Ad} D.R. Adams, L.I. Hedberg, Function Spaces and Potential Theory, Springer, New York, 1996.

\bibitem{An1} A. Ancona, {\em Theori\'e du potentiel sur les graphes et les	vari\'et\'es}, in Ecole d'\'et\'e de Probabilit\'es de Saint-Flour XVIII-1988, Springer Lecture Notes in Math, {\bf 1427} (1990), 1--112.

\bibitem{An2} A. Ancona, {\em Negatively curved manifolds, elliptic operators and the Martin boundary}, Ann. of Math. (2), {\bf 125} (1987), 495--536.

\bibitem{BGT} G. Barbatis, K. T. Gkikas and A. Tertikas, \emph{Heat and Martin kernel estimates for Schr\"{o}dinger operators
with critical Hardy potentials,} preprint.

	
	\bibitem{BHV} M.-F. Bidaut-V\'eron, G. Hoang, Q.-H. Nguyen and L. V\'eron, \emph{An elliptic semilinear equation with source term and boundary measure data: The supercritical case,} J. Funct. Anal. {\bf 269} (2015), no. 7, 1995--2017.




\bibitem{BVi} M. F. Bidaut-V\'eron and L. Vivier, {\em An elliptic semilinear equation with source term involving boundary measures: the subcritical case},
	Rev. Mat. Iberoamericana {\bf 16} (2000), 477-513.

\bibitem{BM} H. Brezis and M. Marcus, \emph{Hardy's inequalities revisited}, Ann. Sc. Norm. Super. Pisa Cl. Sci. {\bf 25}  (1997), 217--237.

\bibitem{CV-JDE} H. Chen and L. V\'eron, {\emph Laurent Schr\"odinger operators with Leray-Hardy potential singular on the boundary}, J. Differential Equations {\bf 269} (2020), no. 3, 2091--2131.

\bibitem{CQZ-PAFA} H. Chen, A. Quaas and F. Zhou, {\emph Solutions of nonhomogeneous equations involving Hardy potentials with singularities on the boundary}, Pure Appl. Funct. Anal. {\bf 5} (2020), no. 4, 899--924.

\bibitem{CV} H. Chen and L. V\'{e}ron, \emph{Boundary singularities of semilinear elliptic equations with Leray-Hardy potential},  Commun. Contemp. Math. {\bf 24} (2022), no. 7, Paper No. 2150051, 37 pp. 

\bibitem{CheZho} H. Chen and F. Zhou, \emph{Isolated singularities for elliptic equations with Hardy operator and source nonlinearity}, Discrete Contin. Dyn. Syst. {\bf 38}, (2018), 2945--2964.


\bibitem{DD} J. D\'avila and L. Dupaigne, \emph{Hardy-type inequalities,} J. Eur. Math. Soc. (JEMS) {\bf 6} (2004), no. 3, 335--365.


\bibitem{DaM} G. Dal Maso,  \emph{On the integral representation of certain local functionals}, Ricerche Mat. {\bf 32} (1983), 85--113.


\bibitem{DD1} J. D\'avila and L. Dupaigne, {\em Comparison results for PDEs with a singular potential}, Proc. Roy. Soc. Edinburgh Sect. A {\bf 133} (2003), 61--83.
	
\bibitem{DD2} J. D\'avila and L. Dupaigne, {\em Hardy-type inequalities}, J. Eur. Math. Soc. {\bf 6} (2004), 335--365.

	
\bibitem{DN} L. Dupaigne and G. Nedev, {\em Semilinear elliptic PDE’s with a singular potential}, Adv. Differential Equations {\bf 7} (2002), 973--1002.



\bibitem{FF} M. Fall and F. Mahmoudi, \emph{Weighted Hardy inequality with higher dimensional singularity on the boundary.} Calc. Var. Partial Differential Equations {\bf 50} (2014), no. 3-4, 779-798.

\bibitem{Fal} M. M. Fall, \emph{Nonexistence of distributional super solutions of a semilinear elliptic equation with Hardy potential}, J. Funct. Anal. {\bf 264} (2013), 661--690.
		
\bibitem{FDeP} D. Feyel and A. de la Pradelle, \emph{Topologies fines et compactifications associ\'{e}es  certains espaces de Dirichlet}, Ann. Inst. Fourier (Grenoble) {\bf 27} (1977), 121--146.

\bibitem{FMT} S. Filippas, L. Moschini and A. Tertikas, \emph{Sharp two-sided heat kernel estimates for critical Schr\"odinger operators on
bounded domains}, Comm. Math. Phys. {\bf 273} (2007), no. 1, 237--281.

\bibitem{GkiNg_2019} K. T. Gkikas and P.-T. Nguyen, \emph{On the existence of weak solutions of semilinear elliptic equations and systems with hardy potentials}, Journal of Differential Equations {\bf 266}  (2019), 833--875.

\bibitem{GkiNg_absorption} K. Gkikas and P.-T. Nguyen, \emph{Semilinear elliptic Schr\"odinger equations with singular potentials and absorption terms}, preprint.


\bibitem{GkiNg_source} K. Gkikas and P.-T. Nguyen, \emph{Semilinear elliptic Schr\"odinger equations with singular potentials and source terms}, preprint.


	

\bibitem{GkV} K. T. Gkikas and L. V\'eron, \emph{Boundary singularities of solutions of semilinear elliptic equations with critical Hardy potentials}, Nonlinear Anal. \textbf{121} (2015), 469--540.


\bibitem{GV} A. Gmira and L. V\'eron,\emph{Boundary singularities of solutions of some nonlinear elliptic equations,} Duke Math. J. {\bf 64} (1991) 271--324.

\bibitem{KV} N.J. Kalton and I.E. Verbitsky, \emph{Nonlinear equations and weighted norm inequality}, Trans. Amer. Math. Soc. {\bf 351} (1999) 3441--3497.

\bibitem{Marcus} M. Marcus, \emph{Estimates of Green and Martin
kernels for Schr\" odinger operators with singular potential in Lipschitz
domains}, Ann. Inst. H. Poincar\'e Anal. Non Lin\'eaire {\bf 36} (2019),
1183--1200.

\bibitem{hardy-marcus} M. Marcus, V. J. Mizel and Y. Pinchover, \emph{On the best constant for Hardy's inequality in $\mathbb{R}^n$}, Trans. Amer. Math. Soc. {\bf 350} (1998), 3237--3255.


\bibitem{MarMor} M. Marcus and V. Moroz, \emph{Moderate solutions of semilinear elliptic equations with Hardy potential under minimal restrictions on the potential}, Ann. Sc. Norm. Super. Pisa Cl. Sci. {\bf 18} (2018), 39--64.

\bibitem{MarNgu} M. Marcus and P.-T. Nguyen, \emph{Moderate solutions of semilinear elliptic equations with Hardy potential}, Ann. Inst. H. Poincar\'e Anal. Non Lin\'eaire {\bf 34} (2017), 69--88.

\bibitem{MT} M. Marcus and P.-T. Nguyen, \emph{Schr\"odinger equations with singular potentials: linear and nonlinear boundary value problems.} Math. Ann. {\bf 374} (2019), no. 1-2, 361--394.


\bibitem{MVbook} M. Marcus and L. V\'eron, {\em Nonlinear second order elliptic equations involving measures}, De Gruyter Series in Nonlinear Analysis and Applications, 2013.

\bibitem{MV-Pisa} M. Marcus and L. V\'eron, \emph{Boundary trace of positive solutions of supercritical semilinear elliptic equations in dihedral domains}, Ann. Sc. Norm. Super. Pisa Cl. Sci. {\bf 15} (2016), 501--542.

	
\bibitem{MV-JMPA01} M. Marcus and L. V\'{e}ron, \emph{Removable singularities and boundary trace},  J. Math. Pures Appl. {\bf 80}  (2001), 879--900.

	
\bibitem{Vbook} L. V\'{e}ron, \emph{Singularities of Solutions of Second Order Quasilinear Equations}, Pitman Research Notes in Math. Series 353, (1996).
	
	
\bibitem{Stein} E. M. Stein, \em{Singular integrals and differentiability properties of functions}, Princeton University Press, 1970.

\end{thebibliography}
\end{document}